\documentclass[11pt,reqno]{article}
\usepackage{fullpage}
\usepackage[utf8]{inputenc}
\usepackage[colorlinks,
            linkcolor=red,
            anchorcolor=red,
            citecolor=blue
            ]{hyperref}

\usepackage{graphicx, amsmath, fullpage, amssymb, amsthm, amsfonts, color, algorithm,mathtools,bbm}
\mathtoolsset{showonlyrefs}
\usepackage{enumerate}
\usepackage{cases}
\newtheorem{theorem}{Theorem}[section]
\newtheorem{lemma}[theorem]{Lemma}
\usepackage{comment}
\usepackage{enumitem}
\usepackage{colonequals}
\usepackage{bm}

\newtheorem{definition}[theorem]{Definition}

\newtheorem{question}[theorem]{Question}

\newtheorem{prop}[theorem]{Proposition}

\newtheorem{cor}[theorem]{Corollary}
\newtheorem{remark}[theorem]{Remark}
\newcommand{\Var}{\mathrm{Var}}
\newcommand{\Cov}{\mathrm{Cov}}

\newcommand{\CC}{\mathbb{C}}

\newcommand{\EE}{\mathbb{E}}

\newcommand{\NN}{\mathbb{N}}
\newcommand{\PP}{\mathbb{P}}

\newcommand{\RR}{\mathbb{R}}

\newcommand{\ZZ}{\mathbb{Z}}
\DeclareSymbolFont{bbold}{U}{bbold}{m}{n}
\newcommand{\One}{\mathbbm{1}}

\newcommand{\OU}{\mathrm{OU}}
\newcommand{\PDBOU}{\mathrm{PDBOU}}
\newcommand{\PDBR}{\mathrm{PDBR}}
\newcommand{\entries}{\mathrm{entries}}

\newcommand{\Resamp}{\mathrm{Resamp}}
\newcommand{\Unif}{\mathrm{Unif}}

\newcommand{\sA}{\mathcal{A}}
\newcommand{\sB}{\mathcal{B}}
\newcommand{\sC}{\mathcal{C}}
\newcommand{\sE}{\mathcal{E}}
\newcommand{\sL}{\mathcal{L}}
\newcommand{\sN}{\mathcal{N}}
\newcommand{\sF}{\mathcal{F}}
\newcommand{\sX}{\mathcal{X}}
\newcommand{\sR}{\mathcal{R}}
\newcommand{\sT}{\mathcal{T}}

\newcommand{\sym}{\mathrm{sym}}
\renewcommand{\vec}{\mathrm{vec}}
\newcommand{\rank}{\mathrm{rank}}

\newcommand{\Dom}{\mathrm{Dom}}
\renewcommand{\Re}{\mathrm{Re}}
\renewcommand{\Im}{\mathrm{Im}}

\newcommand{\Ex}{\mathop{\mathbb{E}}}

\newcommand{\scc}{\mathrm{sc}}

\newcommand{\cpxi}{{\bm i}}

\newcommand{\GOE}{\mathrm{GOE}}
\newcommand{\DBM}{\mathrm{DBM}}

\newcommand{\ind}{\mathbbm{1}}

\DeclareMathOperator*{\argmax}{arg\,max}

\definecolor{darkgreen}{rgb}{0.0, 0.5, 0.0}

\newif\ifnotes
\notestrue

\usepackage{authblk}

\author{Xiangyi Zhu\thanks{Email: \texttt{xzhu96@jhu.edu}.}\,\,\,}
\author{Dmitriy Kunisky\thanks{Email: \texttt{kunisky@jhu.edu}.}}

\affil{Department of Applied Mathematics \& Statistics, Johns Hopkins University}

\title{Generalized noise sensitivity of eigenvectors: All eigenvectors, inhomogeneous variance profiles, and dependent resampling}
\date{February 17, 2026}

\begin{document}

\maketitle
\thispagestyle{empty}

\begin{abstract}
    Chatterjee (2016) proved, as an application of his general framework relating superconcentration and chaos, that after the entries of an $n \times n$ matrix drawn from the Gaussian unitary ensemble undergo an entrywise Ornstein-Uhlenbeck (OU) process for time greater than $n^{-1/3}$, the top eigenvector of the matrix becomes almost completely decorrelated from its initial position.
    More recently, Bordenave, Lugosi, and Zhivotovskiy (2020) showed that the same happens under a discrete resampling model, once more than $n^{5/3}$ randomly chosen entries of a Wigner random matrix are resampled.
    We generalize these results in several directions: (1) we analyze the decorrelation of \emph{any} eigenvector under continuous and discrete resampling dynamics, (2) we analyze the discrete resampling process for generalized Wigner matrices with inhomogeneous variance profiles, (3) we analyze a combination of continuous and discrete resampling where an OU process is repeatedly run for a certain time on randomly chosen entries, and (4) we analyze a dependent version of resampling where entries grouped into ``blocks'' of arbitrary shapes are resampled together.
    In each case, we show that a given eigenvector decorrelates provided that enough entries have been resampled or that the associated dynamics have been run for a long enough time.
    Our proofs take a different approach from prior work, relying more directly on the characterization of eigenvectors as derivatives of eigenvalues and reducing the problem of establishing eigenvector noise sensitivity to variants of standard and robust properties of random matrices such as bounds on eigenvalue spacings and eigenvector delocalization.
\end{abstract}

\newpage
\thispagestyle{empty}
\tableofcontents
\thispagestyle{empty}

\clearpage

\section{Introduction}
\pagenumbering{arabic}

The magnitude and distribution of fluctuations in statistics of large random objects is one of the central concerns of probability theory.
In the surprising phenomenon of \emph{superconcentration}, the variances of certain such statistics can be exceptionally small, beyond the bounds guaranteed by classical and general-purpose concentration-of-measure estimates.
Superconcentration has been observed for statistics of random matrices, first-passage percolation, Gaussian random polymers, and spin glasses.
For an overview of early developments, see the beginning of the book-length treatment \cite{chatterjeeSuperconcentration}, and \cite{bray1987chaotic,tracy1994level,benjamini1999noise,benjamini2011first, johansson2000shape} for references on specific instances of superconcentration.
In the above work, building on \cite{chatterjee2008chaos,chatterjee2009disorder}, Chatterjee developed a general theory of superconcentration, establishing that in various settings its appearance coincides with that of \emph{chaos}, a heightened sensitivity of statistics of a random object to small perturbations of the underlying randomness (also sometimes referred to as \emph{noise sensitivity}).

One simple illustrative example Chatterjee used for his theory is that of the top eigenvector of a random matrix \cite[Section 3.6]{chatterjeeSuperconcentration}.
We write $\lambda_1(X) \geq \cdots \geq \lambda_n(X)$ for the ordered eigenvalues of a Hermitian random matrix, and $v_1(X), \dots, v_n(X)$ for the associated unit eigenvectors (provided the spectrum of $X$ is simple, which it will be almost surely in all examples we study here).
Chatterjee showed that if $X(t)$ is an $n \times n$ Hermitian random matrix such that $X(0)$ is drawn from the \emph{Gaussian unitary ensemble} and $X(t)$ evolves according to \emph{stationary Dyson Brownian motion} (equivalently, an entrywise Ornstein-Uhlenbeck process), then
\[ \lim_{n \to \infty} \EE\bigg[\langle v_1(X(0)), v_1(X(t)) \rangle^2\bigg] = 0 \]
provided $t = t(n) \gg n^{-1/3}$.
In words, a matrix need move only for a very short time along the Dyson Brownian motion process in order for the top eigenvector to \emph{decorrelate} completely from its original position.
This gives a concrete sense in which the top eigen\emph{vector} is a chaotic statistic of a random matrix, and Chatterjee's theory relates this to the superconcentration of the related top eigen\emph{value} statistic, a non-trivial but well-known fact established in earlier work in random matrix theory (e.g., \cite{tracy1994level,Ledoux2003,Ledoux2007,ledoux2010small}).

It is natural to wonder how much this intriguing result depends on the various details of the setup.
The results of \cite{bordenave2020noise} adapted Chatterjee's result to a very different notion of sensitivity, replacing $X = X(0)$ not by $X(t)$ its destination along Dyson Brownian motion, but by a matrix $Y$ formed by \emph{resampling} some number $k = k(n)$ of the entries of $X$ chosen uniformly at random (also retaining the symmetry of $X$).
Their main result shows that, likewise,
\[ \lim_{n \to \infty} \EE\bigg[\langle v_1(X), v_1(Y) \rangle^2\bigg] = 0 \]
provided that $k = k(n) \gg n^{5/3}$: only a vanishingly small fraction of the $\Theta(n^2)$ total entries of $X$ need to be resampled before the top eigenvector decorrelates.
Their result also applies to Wigner matrices (having i.i.d.\ entries up to symmetry) more generally, and has since been further adapted to rectangular matrices \cite{wang2022resampling}, sparse matrices \cite{bordenave2022noise}, and the second-to-top eigenvectors of random graphs \cite{lee2020noise}.

We recall the simple intuition behind this result, as presented by \cite{bordenave2020noise}.
Write $Y(0) = X$ and $Y(j)$ for $j \geq 1$ for the discrete-time matrix-valued stochastic process where entries are chosen and resampled uniformly at random one at a time.
The eigenvalues $\lambda_{\alpha}(Y(j))$ evolve accordingly in discrete time.
Since the derivatives of the eigenvalue functions $\lambda_{\alpha}$ are given by the matrix-valued functions $v_{\alpha}v_{\alpha}^{\top}$ for matrices with simple spectrum (see our Section~\ref{sec:diff-eig}), and the eigenvectors of each Wigner matrix $Y(j)$ are delocalized in the sense that $\|v_{1}(Y(j))\|_{\infty} \leq n^{-1/2 + o(1)}$ with high probability, we see (say by taking a first-order expansion, as detailed in \cite{bordenave2020noise}) that resampling one more entry of $Y(j)$ changes each $\lambda_{\alpha}(Y(j))$ by a random quantity of order roughly $n^{-1}$.
After $k$ steps, $\lambda_{\alpha}(Y(j))$ should thus have moved a distance $k^{1/2} / n$, since these steps have random signs and should behave as though they are independent.
We expect the top eigenvectors to decorrelate once $\lambda_1(Y(j))$ ``collides with'' the adjacent eigenvalue $\lambda_2(Y(j))$.
The initial spacing between these two is known to be $\lambda_1(X) - \lambda_2(X) \asymp n^{-1/6}$ with high probability (see our Section~\ref{sec: eigenvalue_spacing_p} for discussion of such results), and thus we expect this collision to have occurred once $k^{1/2} / n \gg n^{-1/6}$, which upon rearranging gives the correct threshold $k \gg n^{5/3}$.

Our goal in this paper is to address several questions left open by the above prior works \cite{bordenave2020noise,lee2020noise,bordenave2022noise,wang2022resampling}, which are natural to ask in light of the simple and robust intuition for noise sensitivity of the top eigenvector under resampling given in the above argument.
Below, we describe at a high level the three main generalizations we prove here, as well as the ideas of the underlying proof techniques.

\paragraph{Overview of main results}
First, it is natural to expect that a similar analysis should apply to \emph{all} eigenvectors of a random matrix, not just the top one, if we modify the final calculation to include the appropriate eigenvalue spacing; the typical scale of $\lambda_{\alpha}(X) - \lambda_{\alpha - 1}(X)$ depends non-trivially on $\alpha$, but such properties have been studied thoroughly in random matrix theory.
Extending such sensitivity results to other eigenvectors has been mentioned as an open problem not amenable to existing techniques by \cite{bordenave2020noise,bordenave2022noise}.
Our first contribution is to circumvent those difficulties with a new proof technique and prove noise sensitivity for \emph{all} eigenvectors.

Second, the above heuristic does not depend on the identical distribution of entries in a Wigner matrix.
Indeed, we will see in our second contribution that this assumption can be at least somewhat relaxed to a large class of \emph{generalized Wigner matrices} having independent but not identically distributed entries subject only to mild regularity conditions.

Third, and perhaps of the greatest conceptual interest, the heuristic argument does not appear to depend on the $k$ resampled entries being chosen uniformly at random among all entries, or indeed at random at all.
It seems reasonable to expect that the same sensitivity applies to arbitrary fixed ``patterns'' of resampled entries.
We pose this as a question, which we have not been able to answer in full but on which we will present some initial progress.
\begin{question}
    Let $X = X^{(n)}$ be an $n \times n$ Wigner random matrix (say, with i.i.d.\ uniformly subgaussian entries) and $Y = Y^{(n)}$ be formed by resampling \emph{any} $k = k(n)$ deterministically chosen entries of $X$ (for each $n$, and with the choice of entries not depending on the random outcome of $X$).
    Is it true that, regardless of what entries are chosen, if $k = k(n) \gg n^{5/3}$ then
    \[ \lim_{n \to \infty} \EE\bigg[\langle v_1(X), v_1(Y) \rangle^2\bigg] = 0 \, \text{?} \]
\end{question}
\noindent
We take a first step in the study of this intriguing possibility by studying certain \emph{dependent resampling} schemes.
In particular, we suppose that instead of choosing entries to resample uniformly at random, the entries are first partitioned into deterministic blocks, and those blocks are resampled uniformly at random.
According to the size of these blocks, we allow a certain amount of bounded dependence in which entries are resampled.
In our third contribution we show that (omitting some further details) provided that the blocks are all of a certain size at most $n^{\delta}$ for an absolute constant $\delta > 0$, the top eigenvector decorrelates once a total of more than $n^{5/3}$ entries have been resampled.
Further, this may be combined with our other generalizations to apply \emph{mutatis mutandis} to all eigenvectors of generalized Wigner matrices as well.

\paragraph{Overview of proof techniques}
To show these generalizations, we develop an alternative approach to applying Chatterjee's theory to discrete resampling procedures on random matrices.
Actually, Chatterjee himself developed some general such tools for noise sensitivity of functions of independent random variables under resampling of their inputs in \cite[Chapter 7]{chatterjeeSuperconcentration}.
In our random matrix context, that approach can be viewed as replacing the Dyson Brownian motion Markov process with what Chatterjee refers to as the \emph{independent flips} Markov process, a resampling process where entries are resampled whenever an associated independent Poisson clock rings.\footnote{The term ``flips'' alludes to the natural special case of Boolean values, but the same ideas apply to resampling for general product measures, as discussed in \cite{chatterjeeSuperconcentration}.}
However, Chatterjee's tools rely on the functions under consideration being smooth, which the functions evaluating eigenvalues ($\lambda_{\alpha}$) and eigenvectors ($v_{\alpha}$) of a matrix are not.
Roughly speaking, here we implement Chatterjee's analysis of the independent flips process while allowing for functions that fail to be smooth only on a sufficiently ``thin'' set (in our case, the set of matrices with repeated or near-repeated eigenvalues) so long as this set is avoided by paths between matrices (in our random matrix application) and their resampled versions.
As a further demonstration of the flexibility of this approach, we analyze a process that is a kind of hybrid of Dyson Brownian motion and discrete resampling, where matrix entries are advanced along entrywise Ornstein-Uhlenbeck processes whenever associated independent Poisson clocks ring.

To the best of our knowledge, all previous work on the noise sensitivity of top eigenvectors relied in various ways on the variational characterization of the top eigenvector, $v_1(X) = \argmax_{\|v\| = 1} v^{\top} X v$.
On the other hand, our method merely uses that $\lambda_{\alpha}$ is a ``mostly smooth'' function of a symmetric matrix whose derivative is $v_{\alpha}v_{\alpha}^{\top}$, allowing all eigenvectors to be treated by basically identical means and reducing the task of establishing noise sensitivity to that of establishing some modest variations of standard properties of the spectra of random matrices, like eigenvalue superconcentration, eigenvalue spacing statistics, and eigenvector delocalization.
See Section~\ref{sec: proof_technique} for further details on the properties our method relies on.

We hope that these tools will allow for a more general understanding both of noise sensitivity of eigenvectors, which our results suggest is actually a quite universal phenomenon holding robustly over large classes of random matrix distributions, and of the superconcentration--chaos relationship for non-smooth functions under discrete resampling in settings outside of random matrix theory.

\subsection{Main results}

We now state our main results precisely.
We will work with the following distributions of random matrices, the first one very special case and the second a general class.
\begin{definition}[Gaussian orthogonal ensemble]
    The \emph{Gaussian orthogonal ensemble}, denoted $\GOE(n)$, is the law of $G \in \RR^{n \times n}_{\sym}$ whose entries on and above the diagonal are independent and distributed as $G_{ij} \sim \sN(0, 1 + \One\{i = j\})$.
\end{definition}
\noindent
We note that the original result of Chatterjee considered the similar but complex-valued \emph{Gaussian unitary ensemble}.
Our method should apply equally well to that case, but we restrict our attention to the real case for the sake of simplicity.

\begin{definition}[Sub-Gaussian generalized Wigner matrix]
\label{def:gen-wigner}
We call a random matrix $X \in \RR^{n \times n}_{\sym}$ a \emph{sub-Gaussian generalized Wigner matrix} with parameters $(c_1, c_2, K)$ if the following properties hold:
\begin{enumerate}
    \item The entries of $X$ on and above the diagonal, $(X_{ij})_{1 \leq i \leq j \leq n}$, are independent.
    \item $\EE X_{ij} = 0$ for all $i, j \in [n]$.
    \item The entrywise variances $\sigma_{ij}^2 \colonequals \EE X_{ij}^2$ satisfy the following bounds:
    \begin{align}
        \sigma_{ij}^2 &\in [c_1, c_2] \text{ for all } i, j \in [n], \\
        \sum_{j = 1}^n \sigma_{ij}^2 &= n \text{ for all } i \in [n]. \label{eq:var-reg}
    \end{align}
    \item $\|X_{ij}\|_{\psi_2} \leq K$ for all $i, j \in [n]$, i.e., each $X_{ij}$ is sub-Gaussian with the same variance proxy.\footnote{Here the $\psi_2$ norm is defined as $\| X_{ij}\|_{\psi_2} \coloneqq \inf \{ K>0 :\EE[\exp(X^2/K^2)] \leq 2 \}$.}
    \item $X_{ij}$ admits a density with respect to Lebesgue measure for all $i, j \in [n]$.
\end{enumerate}
    If furthermore $\sigma_{ij}^2 = 1$ for all $i, j \in [n]$, then we call $X$ a \emph{Wigner matrix}.
\end{definition}
\noindent
Note that a GOE matrix is also a sub-Gaussian generalized Wigner matrix, so in fact all matrices we work with in this paper are sub-Gaussian generalized Wigner matrices.

\begin{remark}[Justification of assumptions]
    Let us comment on the role that these various assumptions play in our analysis.
    
    Condition 3, and in particular the condition in \eqref{eq:var-reg}, puts such a random matrix in the ``semicircle universality class'': over a sequence of such $X = X^{(n)}$, we will have that the empirical spectral distribution of $\frac{1}{\sqrt{n}}X^{(n)}$ converges to the semicircle distribution, just as in the case where $\sigma_{ij}^2 = 1$ for all $i, j \in [n]$ (see \cite{erdHos2012rigidity}, for instance). However, allowing a general variance profile comes at a cost for the control of the minimal eigenvalue spacing (at least using the best results currently known), a tradeoff we discuss in Section~\ref{prel: minimal_spacing}.
    
    Condition 4 on sub-Gaussianity for us plays an important role in ensuring the delocalization of all eigenvectors, as we discuss in Section~\ref{sec: eigenvector_delocal}.
    Indeed, heavy-tailed entries can lead to eigenvectors essentially generated by the presence of one very large matrix entry, which are then highly localized on that entry's indices.

    Finally, Condition 5 is not usually included in the definition of generalized Wigner matrices, but is quite important for us since we rely heavily on the ``near-smoothness'' of the functions outputting the eigenvalues of a matrix and it will be important that our matrices have simple spectra almost surely, falling in the set where these functions are in fact smooth.
    See Section~\ref{sec: taylor_expansion} for the point in our argument where this is used.
    On the other hand, discrete matrices like ones with entries drawn from, say, $\Unif(\{\pm 1\})$, have repeated eigenvalues with small but positive probability.
\end{remark}

Next, let us make a few preliminary definitions for working with the eigenvalues of such matrices.
As mentioned, for $X \in \RR^{n \times n}_{\sym}$, we write $\lambda_1(X) \geq \cdots \geq \lambda_n(X)$ for its ordered (real) eigenvalues.
When these eigenvalues are simple, we write $v_{\alpha}(X)$ for the eigenvector of unit norm, unique up to sign, associated to $\lambda_{\alpha}(X)$.
In particular, while $v_{\alpha}(X)$ is not quite well-defined, $v_{\alpha}(X)v_{\alpha}(X)^{\top}$ is.
We always use Greek letters for the indices of eigenvalues and eigenvectors.
As in many results on random matrices, we will be interested in the distance of an eigenvalue's index from the edge of the spectrum:
\begin{equation}
    \hat{\alpha} \colonequals \min\{\alpha, n + 1 - \alpha\}.
\end{equation}
Throughout we will also work with the following related function, which can be used to describe the best known (to our knowledge) variance bounds on various eigenvalues: when working with a sub-Gaussian generalized Wigner matrix with parameters $(c_1, c_2, K)$, for constants $A_1 > 0$ and $A_2 > 1$ depending only on these parameters, we set
\begin{equation}
    \label{eq:F}
    F(n, \alpha) \colonequals \left\{ \begin{array}{ll} A_1 & \text{if } \hat{\alpha} = 1 \text{ (i.e., if } \alpha \in \{1, n\}\text{)}, \\ (\log n)^{A_2 \log \log n} & \text{if } \hat{\alpha} \geq 2 \end{array}\right\}.
\end{equation}
We will see that these $F(n, \alpha)$ appear in bounds on $\Var[\lambda_{\alpha}(X)]$, which in turn leads to their appearance in the thresholds we establish.
See Corollary~\ref{cor: eigenvalue_bound} for these variance estimates.
In Section~\ref{prel: variance_bound} we discuss in detail the choice of $F(n,\alpha)$ in the
generalized Wigner setting and its dependence on $\alpha$.
For an intuitive reading of our results, it suffices to think of $F(n, \alpha)$ as being some factor of sub-polynomial scaling, $F(n, \alpha) = n^{o(1)}$.

\subsubsection{Ornstein-Uhlenbeck process on GOE matrices}\label{sec: OU_process}

Our first result is a direct generalization (aside from switching from GUE to GOE matrices) of the result of Chatterjee in \cite{chatterjeeSuperconcentration} that initiated the study of resampling stability of eigenvectors.
While Chatterjee's result only treated the top eigenvector, we will show that an analogous result holds for all eigenvectors.
Recall that this concerns the sensitivity of eigenvectors under the following process on the underlying matrix:
\begin{definition}[Dyson Brownian motion]
    We write $\DBM(n)$ for the law of the stochastic process $W(t) \in \RR^{n \times n}_{\sym}$ where $W_{ij}(t) / \sqrt{1 + \One\{i = j\}}$ are independent standard Brownian motions for $1 \leq i \leq j \leq n$.
\end{definition}
\begin{definition}[Matrix Ornstein-Uhlenbeck process]\label{def: OU_process}
    Let $G \sim \GOE(n)$ and let $W(t) \sim \DBM(n)$ be a standard symmetric matrix Brownian motion, independent of $G$.
    Then, the \emph{matrix Ornstein-Uhlenbeck (OU)} process, sometimes also called \emph{stationary Dyson Brownian motion}, is the stochastic process
    \[ G(t) = e^{-\tau t}G + e^{-\tau t} W(e^{2\tau t} -1), \]
    defined for all $t \geq 0$, for a parameter $\tau > 0$.
    In this case we write $G(t) \sim \OU(n, \tau)$.
\end{definition}

\begin{theorem}\label{thm: OU_process}
    Let $G(t) \sim \OU(n, 1)$ and $\alpha = \alpha(n) \in [n]$.
    Suppose that $t = t(n)$ is such that
    \[ (1 - e^{-t}) \cdot \frac{\hat{\alpha}^{2/3} n^{1/3}}{F(n, \alpha)} = \omega(1). \]
    Then, we have
    \begin{align}
         \EE \left[ \langle v_\alpha(G(0)), v_\alpha(G(t))  \rangle^2 \right] = o(1).
    \end{align}
\end{theorem}
\noindent
The interesting choices of $t = t(n)$ are $ t = o(1)$, in which case we have $1 - e^{-t} \approx t$.
Thus, the result says, neglecting the subpolynomial $F(n, \alpha)$ factor, that the eigenvector $v_{\alpha}$ decorrelates once we move for time $t \gg \hat{\alpha}^{-2/3}n^{-1/3}$ along the OU process.
In particular, for $\alpha = 1$ this recovers (a version for GOE matrices of) the result of \cite{chatterjeeSuperconcentration} mentioned earlier.

\subsubsection{Poisson-driven block Ornstein-Uhlenbeck process on GOE matrices}\label{sec: Poi_OU_intro}

Moving towards more discrete resampling dynamics, we first consider a variant matrix OU process.
The idea is as follows: we partition the entries into disjoint ``blocks,'' and consider repeatedly choosing a random block and moving its entries by some time $\tau$ along corresponding OU processes.
As we will see, when $\tau$ is very small this just behaves like the OU process, while when it is very large it behaves like completely resampling blocks chosen at random.

Let us formalize this idea.
First, we define the block patterns we will allow.
\begin{definition}[Block size]
    \label{def: block-size}
    Let $B \subseteq [n] \times [n]$.
    We define
    \[ \nu(B) \colonequals |B| + \#\{i \in [n]: (i, i) \in B\}. \]
\end{definition}
\begin{definition}[Admissible partition]
    \label{def:blocks}
    We call a family of subsets $B_1, \dots, B_m \subseteq [n] \times [n]$ an \emph{admissible partition} if the following properties hold:
    \begin{enumerate}
        \item $B_a$ is \emph{symmetric} for each $a \in [m]$, i.e., $(i, j) \in B_a$ if and only if $(j, i) \in B_a$.
        \item $B_1 \cup \cdots \cup B_m = [n] \times [n]$.
        \item The $B_a$ are pairwise disjoint.
        \item There exists $\nu \geq 1$ such that, for all $a \in [m]$, we have $\nu(B) = \nu$.
    \end{enumerate}
    We often denote such a partition by $\sB = \{B_1, \dots, B_m\}$, we call $B_a$ the \emph{blocks} of such a partition, and we call $\nu$ its \emph{size parameter}.
    We also write $\sA_k$ for the set of all unions of $k$ of the sets in $\sB$, for each $k \in [m]$.
    We will often bring up the example of the partition that puts, up to symmetry, every entry into its own block, which we denote
    \[ \sB_{\entries} \colonequals \{ \{(i,j),(j,i)\} : 1 \leq i < j \leq n \} \cup \{ \{(i, i)\}: 1 \leq i \leq n \}. \]
\end{definition}

\begin{remark}
    It seems likely that the last two conditions on disjointness and equal size (with double-counting diagonal entries) could be relaxed somewhat.
    But, as we will see, this would make our calculations considerably more complicated and likely would require stronger restrictions on $m$ and $\tau$, so here we work with the above definition.
\end{remark}

We now formalize the matrix-valued process we sketched above.
The main object is the following stochastic process.
\begin{definition}[Poisson-driven block Ornstein-Uhlenbeck process]\label{def: PDBOU}
Let $\sB = \{B_1, \dots, B_m\}$ be an admissible partition of $[n] \times [n]$.
To each $B \in \sB$, associate an independent ``Poisson clock'' with rate $\eta > 0$, i.e., a Poisson point process on with rate $\eta$ on $\RR_{\geq 0}$, yielding a random set of points $T_B = \{t_{B, 1} < t_{B, 2} < \cdots \} \subset \RR_{\geq 0}$.
Let $K_B: \RR_{\geq 0} \to \NN$ be the associated counting process, $K_B(t) = \#\{i: t_{B, i} \leq t\}$, giving the (random) number of times that the clock of block $B$ has rung by time $t$.
Let $G(t) \sim \OU(n, \tau)$.
The \emph{Poisson-driven block Ornstein-Uhlenbeck (PDBOU) process} with parameters $(\sB, \eta, \tau)$ is the stochastic process $\widetilde{G}(t) \in \RR^{n \times n}_{\sym}$ with entries
\[ \widetilde{G}(t)_{ij} = G\left(\sum_{B \in \sB: (i, j) \in B} K_B(t)\right)_{ij}. \]
In this case we write $\widetilde{G}(t) \sim \PDBOU(n, \sB, \eta, \tau)$.
\end{definition}
\noindent
In words, under the PDBOU, the total amount that entry $(i, j)$ advances along the OU process has a contribution of $\tau$ for every time the clock has rung of a block to which $(i, j)$ belongs.
(In particular, $\widetilde{G}(t)$ is a jump process: entries only change at the times when some clock rings.)
For an admissible partition $\sB$ there will only be one such block, but the definition is sensible even if blocks overlap, a possibility our results do not treat but that we will sometimes discuss.
We discuss the probabilistic properties of the PDBOU in general in Section~\ref{prel: POU} without the structural assumption of $\sB$ being an admissible partition.

\begin{theorem}\label{thm: main_POU}
Let $\delta \in (0, 1/12)$. Let $\sB^{(n)}$ be an admissible partition of $[n] \times [n]$ for each $n$ with size parameter $\nu = \nu(n)$, let $\eta = \eta(n) > 0$, $\tau = \tau(n) > 0$ be parameters, $\alpha = \alpha(n) \in [n]$, and let $\widetilde{G} = \widetilde{G}^{(n)} \sim \PDBOU(n, \sB, \eta, \tau)$.
Suppose that $t = t(n) > 0$ is such that
\begin{align}
   \nu &= O\left(\frac{n^{\delta}}{1 \wedge \tau}\right), \\
   \eta t  &\leq {e^{\tau}} \log \frac{1}{1-e^{-\tau}}, \label{equ: teta}\\
    t \cdot \eta(1 \wedge \tau)^2 \cdot \frac{ \hat{\alpha}^{2/3}   n^{1/3}}{F(n, \alpha)} &= \omega(1) \label{equ: pou_tF}.
\end{align}
Then, we have
    \begin{align}
         \EE \left[ \langle v_\alpha(\widetilde{G}(0)), v_\alpha(\widetilde{G}(t))  \rangle^2 \right] =
             o(1).
    \end{align}
\end{theorem}

\begin{remark}[Block size condition]
The choice of $\delta \in (0, 1/12)$ in the restriction on $\nu$ is suboptimal and written this way to give a clean statement here. 
  A more precise sufficient condition that follows from the proof is to have $\delta, \epsilon > 0$ such that
  \begin{align}
        \nu = o\left(\frac{n^{5/6 - \delta}}{\log n} \hat{\alpha}^{-1/3}\right), \quad \nu = O\left(\frac{n^{1/4-\delta/2 - \epsilon}}{\sqrt{1 \wedge \tau}}\right), \quad \nu = O\left(\frac{n^{\delta/4 - \epsilon}}{1 \wedge \tau}\right).
  \end{align}
   See Theorem~\ref{thm: sum_of_inverse_spacing} for the source of the first condition, and see Lemma~\ref{lem: delta_partial_o1} and the end of the proof of Theorem~\ref{thm: main_POU} for the source of the latter two conditions.
\end{remark}

It is illuminating to consider the two extremes of entries advanced along OU processes individually by short increments at a fast rate, or by long increments at a slow rate.
In both of the below examples, suppose for the sake of simplicity that $\sB = \sB_{\entries}$, so that each entry is in its own block along with its symmetric counterpart and $\nu = 2$.
\begin{itemize}
    \item If we take $\eta = \eta(n) \to \infty$ and $\tau = \tau(n) \to 0$ such that $\eta \tau^2 = 1$, then the result simplifies to recover the threshold established in Theorem~\ref{thm: OU_process} by some modifications. Indeed, the entrywise dynamics converge to a system of independent Ornstein-Uhlenbeck processes running simultaneously, effectively behaving as a unified process despite the block structure.
    
    We note that, despite this analogy, because the standard OU process admits simple analytic descriptions of its transition kernel and Dirichlet form, its analysis for our purposes is much simpler and avoids many of the technicalities giving rise to the extra conditions above for discrete resampling.
 
    \item If we fix, say, $\eta $ a constant and take $\tau$ large, then whenever an entry's clock rings, the entry is effectively almost resampled entirely (by running an OU process for a long time $\tau$), and further rings of that clock no longer change those entries. 
Consequently, the condition in \eqref{equ: teta} simplifies to $t \eta \sim 1$, where $t \eta$ represents the expected number of resampling events per entry by time $t$. This constraint effectively restricts the occurrence of repeated resampling—which becomes redundant in this regime—and ensures that the total number of resampled entries remains small relative to $n^2$. 

Thus, by time $t$, approximately $\Theta(\eta t n^2)$ entries will have been resampled. Comparing with the above conditions, we expect the eigenvector $v_{\alpha}$ to decorrelate once the number of resampled entries exceeds $n^{5/3} / \hat{\alpha}^{2/3}$.
    We will see below that this is indeed the correct threshold for actual discrete resampling dynamics, even for a broader class of random matrices with independent entries.
\end{itemize}

The PDBOU process gives an interpolation between these extremes where either all entries are changed by a small amount (the standard continuous OU process) or a few entries are changed by a large amount (the independent resampling process).
In the PDBOU process, the parameter $\eta$ controls the number of entries affected per unit time (along with the block size $\nu$), while $\tau$ controls the amount an entry changes each time it is affected by the resampling process.
The appearance of the quantity $\eta \tau^2$ (for small $\tau$) in the main condition involving the time parameter $t$ might be viewed as expressing the way that these parameters interact in determining the overall behavior of the process.

\subsubsection{Block resampling on generalized Wigner matrices}\label{sec: independent_resampling}

The general resampling process that we arrived at in the limit $\tau \to \infty$ above is sensible to define more generally.
For example, it makes sense for our sub-Gaussian generalized Wigner matrices (Definition~\ref{def:gen-wigner}), whose entries are independent and can be individually resampled while maintaining the same joint distribution.
We define this process as follows.
\begin{definition}[Block resampling]\label{def: block_resampling}
    \label{def:block-resample}
    Let $X \in \RR^{n \times n}_{\sym}$ be a sub-Gaussian generalized Wigner matrix, let $Y$ be an independent copy of $X$, and let $\sB$ be an admissible partition of $[n] \times [n]$ with $\sA_k$ as defined in Definition~\ref{def:blocks}.
    For each $A \in \sA_k$, we define
    \begin{equation} 
        \label{equ: resample_rule}
    X^A_{ij} \colonequals \begin{cases}
      Y_{ij} &\text{if } (i,j) \in A, \\
      X_{ij} &\text{otherwise}.
  \end{cases}
  \end{equation}
  We write $\Resamp(X; \sB, k)$ for the law of $X^A$ for $A \sim \Unif(\sA_k)$.
\end{definition}
\noindent
In words, we choose $k$ blocks from $\sB$ uniformly at random and resample all entries in $X$ that belong to those blocks.
We discuss the independent resampling process in Section~\ref{prel: independent_resample} without structural assumptions on $\mathcal{B}$, but for Theorem~\ref{thm: resample_block_1}, we specialize to admissible partitions only.
In particular, parts of our calculation can still be carried out when the blocks of $\sB$ are not disjoint, but the crucial \emph{variance identity} we rely on, stated in \eqref{equ: var_identity_dij_c}, becomes much more opaque in this case.

\begin{theorem}
\label{thm: resample_block_1}
There exists an absolute (small) constant $\delta > 0$ such that the following holds.
Let $X = X^{(n)} \in \RR^{n \times n}_{\sym}$ be sub-Gaussian generalized Wigner matrices with fixed parameters $(c_1, c_2, K)$ not depending on $n$. 
Let $\alpha = \alpha(n) \in [n]$.
Let $\sB = \sB^{(n)}$ be an admissible partition of $[n] \times [n]$ with size parameter $\nu = \nu(n)$ and let $k = k(n) \in [|\sB^{(n)}|]$ be another parameter.
Suppose that
\begin{align*} 
\nu &= O(n^{\delta}), \\
k \nu \cdot  \frac{\hat{\alpha}^{2/3}  n^{-5/3}}{F(n, \alpha)} &= \omega(1).
\end{align*}
Let $Y = Y^{(n)} \sim \Resamp(X; \sB, k)$.
Then, we have
    \begin{align}
         \EE\left[\langle v_\alpha(X), v_\alpha(Y) \rangle^2 \right] = o(1).
    \end{align}
Further, if $X$ is a Wigner matrix (i.e., with all entries of variance 1), then we may take any $\delta \in (0, 1/12)$.\footnote{The $\delta$ for which our result holds for generalized Wigner matrices must also be smaller than ${1}/{12}$ , but there are further restrictions on its value likely making it need to be even smaller, coming from absolute constants implicit in the results of \cite{benigni2022optimal} which we have not tried to make explicit.}
\end{theorem}

When $\mathcal{B} = \sB_{\entries}$, then our model resamples $k$ entries chosen at random (on and above the diagonal, with their symmetric counterparts replaced accordingly), and we obtain the following, which may be viewed as (one side of) the result of \cite{bordenave2020noise} extended to all eigenvectors, answering a question posed in that work, and further extended to generalized Wigner matrices.

 \begin{cor}\label{cor: cor_resample_entries}
    Let $X = X^{(n)} \in \RR^{n \times n}_{\sym}$ be sub-Gaussian generalized Wigner matrices with fixed parameters $(c_1, c_2, K)$ not depending on $n$.
    Let $\alpha = \alpha(n) \in [n]$ and $k = k(n) \in \NN$.
    Suppose that
    \begin{align}
        k \cdot  \frac{\hat{\alpha}^{2/3}  n^{-5/3}}{F(n, \alpha)} \to \infty.
    \end{align}
    Let $Y = Y^{(n)} \sim \Resamp(X; \sB_{\entries}, k)$.
    Then,
    \begin{align}
         \EE \left[ \langle v_\alpha(X), v_\alpha(Y) \rangle^2 \right] = o(1).
    \end{align}
    \end{cor}

\subsection{Proof techniques}\label{sec: proof_technique}

In general, Chatterjee's superconcentration--chaos theory relies on \emph{variance identities} satisfied by various Markov processes.
For instance, as we will see in Lemma~\ref{lem: var_identity_OU}, when $G(t) \sim \OU(1)$, then we have, for $f: \RR^{n \times n}_{\sym} \to \RR$ suitably regular,
\[ \Var (f(G)) = \int_0^\infty {e^{-t}}\, \EE[  \langle  \nabla f(G(0)), \nabla f (G(t)) \rangle ]\,  dt. \]
While Chatterjee analyzes the top eigenvector by appealing to some general corollaries for $f$ given by a supremum of linear forms, we observe instead that, provided one treats non-smoothness carefully, one may simply take $f = \lambda_{\alpha}$ in the above.
We then have $\nabla f(X) = v_{\alpha}(X) v_{\alpha}(X)^{\top}$ whenever $\lambda_{\alpha}(X)$ is a simple eigenvalue, which occurs almost surely for all eigenvalues for $X \sim \GOE(n)$.
Theorem~\ref{thm: OU_process} then follows from some elementary manipulations of the resulting identity, together with properties of the Dirichlet form of the OU process appearing in the integrand on the right-hand side and known superconcentration bounds on $\Var(\lambda_{\alpha}(G))$.

Chatterjee proposed to treat resampling dynamics (his \emph{independent flips} Markov process) by replacing the OU process above with a resampling process driven by Poisson clocks, as sketched earlier.
These processes yield analogous variance identities as well, but where the right-hand side above is replaced by an expression involving finite differences of $f$.
For smooth $f$, one may approximate these by derivatives: by a Taylor expansion,
\begin{equation}
f(Y) - f(X) \approx \langle Y - X, \nabla f(X) \rangle. \label{eq:fin-diff-approx}
\end{equation}
Chatterjee's implementation of this idea may \emph{almost} be applied directly to our setting: by \cite[Theorem 7.7]{chatterjeeSuperconcentration}, given a $\sC^2$ function $f: \mathbb{R}^{n \times n}_{\mathrm{sym}} \to \mathbb{R}$, we have the bound 
\begin{align}\label{equ: Chatterjee_equ_1}
    \mathbb{E}\langle \nabla f(X), \nabla f(X^A) \rangle \leq
\frac{{n^2}}{k}\operatorname{Var}(f(X))
+
O\left(n^2 \sup_{X \in \RR^{n \times n}_{\sym}} \|\nabla f(X)\|_{\ell^{\infty}} \sup_{X \in \RR^{n \times n}_{\sym}} \|\nabla^2 f(X)\|_{\ell^{\infty}}\right),
\end{align}
derived by the argument sketched above.
Here the notations $X$ and $X^A$ are the same as in Theorem~\ref{thm: resample_block_1} for the case $\mathcal{B} = \sB_{\entries}$.

The issue is that $f = \lambda_{\alpha}$ is not $\sC^2$, failing to be differentiable at matrices with repeated eigenvalues.
Further, its second derivatives depend on the eigenvalue spacings, and diverge near these singular points (the first derivatives are bounded where they exist, but their worst-case bound coming from the trivial $\|v_{\alpha}(X)\| \leq 1$ also dramatically exceeds the typical delocalized behavior $\|v_{\alpha}(X)\| \leq n^{-1/2 + o(1)}$, which leads to a more quantitative version of the same issue).
Thus, even with various smoothing tricks, the above kind of bound cannot be used in a black-box fashion for our purposes without further understanding of how to work around the non-smoothness of these spectral functions, in particular controlling some notion of the distance between the points where we seek to evaluate $\lambda_{\alpha}$ and its non-smooth points, the matrices with repeated eigenvalues.

To do this, we ``unpack'' Chatterjee's proof of the above bound and carry out a more delicate version with attention to these new quantitative issues.
Specifically, we must understand the finite differences $f(Y) - f(X)$ mentioned above, where $X$ and $Y$ differ on individual entries (for the entrywise resampling setting of \cite{bordenave2020noise}) or blocks of a small number of entries (for our dependent resampling schemes).
Suppose that they differ on at most $\nu$ entries; this will coincide with the size parameter of the partition $\sB$ of the matrix entries into blocks when we study block resampling.

We study these differences by considering a \emph{resampling path},
\[ X(s) \colonequals (1 - s)X + sY. \]
As we will see, general algebraic-geometric considerations show that $\lambda_{\alpha}$ is almost surely smooth along any such path for $X$ and $Y$ arising in our analysis (following from the assumption in Definition~\ref{def:gen-wigner} that the entries of $X$ have a density).
Thus, we are justified in carrying out the approximation in \eqref{eq:fin-diff-approx} by a Taylor expansion,
\begin{align*}
    f(Y) - f(X)
    &= f(X(1)) - f(X(0)) \\
    &= \langle Y - X, \nabla f(X) \rangle + O\left( \nu^2 \cdot \|Y - X\|_{\ell^{\infty}}^2 \cdot \sup_{s \in [0, 1]} \|\nabla^2 f(X(s))\|_{\ell^{\infty}}\right).
\end{align*}
In our version of the argument, the worst-case bounds on derivatives of $f = \lambda_{\alpha}$ from Chatterjee's approach (which are intractably large) will become bounds on the typical values of maximum of derivatives along resampling paths.
These derivatives involve two natural spectral quantities of a matrix: the entries of the eigenvectors and the spacing of the eigenvalues.
So, the technical task our method finally boils down to is to control eigenvector delocalization and eigenvalue spacing \emph{uniformly along resampling paths}.
After establishing these tools, our results follow in similar spirit to the simpler analysis of GOE matrices (Theorem~\ref{thm: OU_process}) above from superconcentration inequalities for the eigenvalues.

Ultimately, our analysis depends only on quite standard information about the underlying random matrix $X$ and the associated resampling paths.
Namely, we only use the following properties:
\begin{enumerate}
    \item Bounds on the eigenvalue variance $\Var(\lambda_{\alpha}(X))$.
    \item Uniform delocalization of all eigenvectors along resampling paths.
    \item Uniformly large spacing of consecutive eigenvalues along resampling paths.
\end{enumerate}
Aside from the special role of resampling paths, these are well-studied properties of random matrices that are believed to enjoy strong universality over large classes of matrix ensembles.
So, we believe our method should likewise generalize straightforwardly, modulo technical challenges, to establish noise sensitivity for various other random matrix distributions.
In our case, we establish Property~1 through the rigidity estimates of \cite{erdHos2012rigidity}, Property~2 from local laws proved by \cite{ajanki2017universality}, and Property~3 from Property 2 together with spacing estimates proved by \cite{benigni2022optimal}.

Perhaps one natural next step for future work would be to establish that the same ingredient results still hold and can be combined in the same way when $X$ is a generalized Wigner matrix without the normalization condition $\sum_{j} \sigma_{ij}^2 = n$, in which case its empirical spectral distribution need not be close to the semicircle law.

\paragraph{Comparison with previous techniques}

While the work of Chatterjee in \cite{chatterjeeSuperconcentration} develops the general methodology using variance identities that we will rely on, its approach to treating the specific function $f(X) = \lambda_1(X)$ is to view it as a maximum, $f(X) = \max_{\|v\| = 1} v^{\top}Xv = \max_{\|v\| = 1} \langle X, vv^{\top}\rangle$.
Various special tools are developed there for the superconcentration--chaos equivalence for such $f$.
Unfortunately, for reasoning about eigenvalues of matrices, this approach seems restricted to the maximum (or minimum) eigenvalues.
As we have mentioned, that work also studies discrete resampling for smooth functions, and our approach essentially adapts those results to non-smooth spectral functions.

On the other hand, \cite{bordenave2020noise,bordenave2022noise,lee2020noise,wang2022resampling} focus on discrete resampling and work with variance identities specific to that setting, similar to analogous results of \cite{chatterjeeSuperconcentration} for the ``independent flips'' process though without explicitly relating resampling to a Markov process (see, e.g., Lemmas 1, 2, and 3 in \cite[Section 2]{bordenave2020noise}).
Essentially, their proof of their analog of our Corollary~\ref{cor: cor_resample_entries} amounts also to showing that the expression arising in a discrete variance identity is related to the gradient of $f(X) = \lambda_1(X)$, which is (at points where $f$ is smooth) $\nabla f(X) = v_1(X)v_1(X)^{\top}$.
However, this relationship remains only implicit in their calculations, making it difficult to generalize the same approach to other eigenvalues.

\begin{remark}
    We emphasize that the above works \cite{bordenave2020noise,bordenave2022noise,lee2020noise,wang2022resampling} all also prove complementary results (in their respective settings) stating that if the number $k$ of entries resampled is sufficiently small, then the resulting top eigenvectors are nearly perfectly aligned.
    Our methods do not seem well-suited to establishing such claims, and we leave to future work the generalization of those results to all eigenvectors, generalized Wigner matrices, and block resampling schemes.
\end{remark}

\subsection{Organization}

The rest of the paper is organized as follows. In Section~\ref{sec: OU_proof}, we prove Theorem~\ref{thm: OU_process}. Sections~\ref{sec: taylor_expansion}, \ref{sec: eigenvector_delocal}, and~\ref{sec: eigenvalue_spacing} prove various preliminaries about eigenvalues along resampling paths:  Section~\ref{sec: taylor_expansion} proves the almost sure smoothness of each eigenvalue along such a path and carries out the corresponding Taylor expansion; Section~\ref{sec: eigenvector_delocal} bounds the first derivatives of eigenvalues along resampling paths, which reduces to establishing uniform eigenvector delocalization along these paths; and Section~\ref{sec: eigenvalue_spacing} bounds the second derivatives along paths, which reduces to establishing uniform eigenvalue spacing bounds. Lastly, using these tools, in Section~\ref{sec: resampling_dynamics} we prove Theorems~\ref{thm: main_POU} and~\ref{thm: resample_block_1} about discrete resampling dynamics.

\section{Preliminaries}

\subsection{Notation} 
\label{sec: notation}

We write $[n] \colonequals \{1, \dots, n\}$ and $\RR_+ \colonequals \{x \in \RR: x \geq 0\}$.
For a vector of non-negative integers $K \in \ZZ_{\geq 0}^m$, we denote
\begin{align*}
    |K| &\colonequals \sum_{i = 1}^m K_i, \\
    K! &\colonequals \prod_{i = 1}^m K_i!.
\end{align*}

We use $\RR^{n \times n}_{\sym}$ to denote the space of real symmetric matrices of size $n \times n$.
For $X \in \RR^{n \times n}_{\sym}$, we write $\lambda_1(X) \geq \cdots \geq \lambda_n(X)$ for its ordered eigenvalues, $\lambda(X) = (\lambda_1(X), \dots, \lambda_n(X))$ for the vector of these values, and $v_{\alpha}(X)$ for the eigenvector associated to $\lambda_{\alpha}(X)$, provided this eigenvalue is simple.
Whenever we refer to eigenvectors, we always assume they are unit vectors.
We use ordinary letters ($i, j$) for the indices of matrix entries, and Greek letters ($\alpha, \beta$) for the indices of eigenvalues and eigenvectors of matrices.
For such a ``spectral index,'' we denote
\[ \hat{\alpha} \colonequals \min \{ \alpha,  n-\alpha+1 \}, \]
and similarly for $\beta$ when it is used in this way.

We define the following additional parameters, which will play important roles in our arguments:
\begin{align}\label{equ: S_alpha}
    S_\alpha(X) &\coloneqq \sum_{\beta =1,  \beta \neq \alpha}^n \frac{1}{|\lambda_\alpha(X) - \lambda_\beta(X)|}, \\
    M(X) &\coloneqq  \max_{\beta \in [n]}  \| v_\beta(X) \|_\infty, \label{equ: M} \\
    \Delta_1(X) &\coloneqq \lambda_1(X) - \lambda_2(X), \label{equ: Delta_1} \\
    \Delta_\alpha(X) &\coloneqq \min \{  \lambda_\alpha(X) - \lambda_{\alpha+1}(X),  \lambda_{\alpha-1}(X) - \lambda_{\alpha}(X)\} \text{ for } \alpha \in \{2, \dots, n - 1\}, \label{equ: Delta_alpha}\\
    \Delta_n(X) &\coloneqq \lambda_{n - 1}(X) - \lambda_n(X).\label{equ: Delta_n}
\end{align}
We will also use the operator, Frobenius, entrywise $\ell^{\infty}$ norms, and $\ell^0$ pseudo-norm of a matrix, defined respectively as
\begin{align}
\|X\| &= \|X\|_{\mathrm{op}} = \max_{\alpha  = 1}^n|\lambda_{\alpha}(X)|, \\
\|X\|_F &= \left(\sum_{\alpha = 1}^n \lambda_{\alpha}(X)^2\right)^{1/2} = \left(\sum_{i, j = 1}^n X_{ij}^2\right)^{1/2}, \\
\|X\|_{\ell^{\infty}} &= \max_{i, j = 1}^n |X_{ij}|, \\
\|X\|_0 &= \left|\{(i,j) \in [n] \times [n]: X_{ij}\neq 0\} \right|.
\end{align}

\subsection{Markov processes and semigroups}

We begin by introducing the standard basic definitions of Markov processes and semigroups, as covered in any standard reference such as \cite{Kallenberg-1997-FoundationsModernProbability} as well as the notations we use for these objects.

Let $(\sX, \sF)$ be a measurable space equipped with a probability measure $\mu$.
For $f: \sX \to \RR$ measurable, we denote
\[ \mu(f) \colonequals \int fd\mu. \]

A \emph{Markov process} $(X(t))_{t \geq 0}$ on the \emph{state space} $\sX$ is a random process that satisfies the Markov property that for every bounded measurable function $f: \sX \to \RR$, and $t,s \geq 0$, there is a bounded measurable function $P_sf$ such that \begin{align}
    \EE[  f (X(t+s) ) \mid ( X(r) )_{0 \leq r \leq t}  ] = (P_sf) (X(t) ),
\end{align}
where we call the family of operators $(P_s)_{s \geq 0}$ the associated \emph{Markov semigroup}. 
We define the \emph{generator} $\sL$ of the process as 
\begin{align}
    \sL f \colonequals \lim_{t \downarrow 0} \frac{P_tf - f}{t},
\end{align}
for every $f \in L^2(\mu)$ that has the above limit existing in $L^2(\mu)$.
We write $\Dom(\sL) \subseteq L^2(\mu)$ for this set, the \emph{domain} of the generator.
We say that $\mu$ is \emph{stationary} for the process if, for all $t \geq 0$ and $f$ bounded and measurable, $\mu(P_tf) = \mu(f)$.
We denote the inner product associated to $\mu$ as
\[ \langle f, g \rangle_{\mu} \colonequals \int fg\,d\mu. \]

The following are some additional properties of a Markov process or its semigroup that will be useful to us.

\begin{definition}[Reversibility]\label{def: reversible}
    A Markov semigroup $(P_s)_{s \geq 0}$ with stationary measure $\mu$ is called \emph{reversible} if, for all $f, g \in L^2(\mu)$,
    \begin{align}
    \langle f, P_t  g\rangle_\mu = \langle P_t  f, g\rangle_\mu,
\end{align}
    i.e., if each $P_t$ is self-adjoint as an operator on $L^2(\mu)$.
\end{definition}

\begin{definition}[Ergodicity]\label{def: ergofic}
    A Markov semigroup $(P_s)_{s \geq 0}$ is called \emph{ergodic} if, for every $f \in L^2(\mu)$, $P_t f \to \mu(f)$ as $t \to \infty$, with convergence in the norm of $L^2(\mu)$.
\end{definition}

\begin{definition}[Dirichlet form]
    The \emph{Dirichlet form} of a Markov process with generator $\sL$ and stationary measure $\mu$ is the bilinear operator
    \[ D(f, g) \colonequals -\langle f, \sL g \rangle_{\mu}, \]
    defined for all $f, g \in \Dom(\sL)$.
\end{definition}

The following simple identity is at the heart of Chatterjee's theory as developed in \cite{chatterjeeSuperconcentration} and related works.
\begin{lemma}[Covariance identity, Lemma 2.1 of \cite{chatterjeeSuperconcentration}]\label{lem: cov_identity}
Let $P_t$ be a reversible ergodic Markov semigroup with stationary measure $\mu$ and generator $\sL$. Let $f,g \in \Dom(\sL)$.
Then,
\begin{align}
    \mathrm{Cov}_{\mu}(f, g) = \langle f - \mu(f), g - \mu(g) \rangle_{\mu} =  \int_0^{\infty} D(f, P_t g)\, dt,
\end{align}
provided that the derivative can be moved inside the integral when differentiating $D(f, P_t g)$ with respect to $t$, and the ``heat equation'' $\frac{d}{dt}P_tg = \sL P_t g$ holds.
\end{lemma}
\noindent
The two technical conditions at the end will hold in all of the examples we work with.
We will also use the following accompanying results on the behavior of the integrand:
\begin{lemma}
    \label{lem: dirichlet_mon}
    In the setting of Lemma~\ref{lem: cov_identity}, we have that $D(f, P_t f)$ is a non-negative, non-increasing function of $t \geq 0$.
\end{lemma}

\noindent
For the sake of completeness, we give a proof of this standard ``energy dissipation'' property in Appendix~\ref{sec: variance_identity}.

As in Chatterjee's method, our proof technique will revolve around using Lemma~\ref{lem: cov_identity} to give a formula for $\Var_{\mu}(f)$ (taking $f = g$ in the Lemma) and relating the right-hand side to derivatives of $f$.
We next describe how the resulting \emph{variance identities} look for the three processes featuring in our main results.

\subsection{Variance identities}
While the above discussion was completely general, now we focus on the specific state space $\sX = \RR^{n \times n}_{\sym}$ that will apply to our examples.

\subsubsection{Ornstein-Uhlenbeck process}\label{prel: OU}

From Lemma~\ref{lem: cov_identity}, we now derive the variance identity for the OU process from Definition~\ref{def: OU_process}.
We have that $\mu = \GOE(n)$ is the stationary distribution of such a process by construction.

Lemma 3.3 and Lemma 3.5 together in \cite{chatterjee2008chaos} state essentially the same variance identity, and we borrow these results which give a simpler collection of assumptions on $f$ than would result from directly applying the general Lemma~\ref{lem: cov_identity}.
This version in particular will allow us to directly apply the result to $f = \lambda_{\alpha}$ later.

\begin{lemma}[Variance identity for OU process]\label{lem: var_identity_OU}
Let $G(t) \sim \OU(n, \tau)$. For any absolutely continuous function $f: \RR^{n \times n}_{\sym} \to \RR $ with gradient $\nabla f$ defined almost everywhere and  $\EE [\|\nabla f(G(0))\|^2] < \infty$, the Dirichlet form is given by
\begin{align}
        D (f, P_tf) =  \Ex_{G \sim \mu} \langle  \nabla f(G),  \nabla P_tf (G) \rangle,
    \end{align}
    which is a non-negative, non-increasing function of $t$.
    Furthermore, 
    \begin{align}
        \Var (f(G)) = \int_0^\infty {e^{-\tau  t}}   \Ex_{G(t) \sim \OU(n, \tau)} \langle  \nabla f(G(0)), \nabla f (G(t)) \rangle \,dt.
    \end{align}
\end{lemma}
\noindent
We give a proof in Appendix~\ref{sec: variance_identity}. 

\subsubsection{Poisson-driven block Ornstein-Uhlenbeck process}\label{prel: POU}

We now give the same derivation for the PDBOU process from Definition~\ref{def: PDBOU}.
In this case we have that $\mu = \GOE(n)$ is again the stationary distribution.

Actually, we may derive a variance formula for both of our notions of block resampling without the full structure of an admissible partition (Definition~\ref{def:blocks}).
Instead, we only make the following assumption:
\begin{definition}
    We call $\sB = \{B_1, \dots, B_m\}$ a \emph{covering} of $[n] \times [n]$ if the following properties hold:
    \begin{enumerate}
        \item Each $B_a \in \sB$ is \emph{symmetric}, in the same sense as in Definition~\ref{def:blocks}.
        \item $B_1 \cup \cdots \cup B_m = [n] \times [n]$.
    \end{enumerate}
\end{definition}
\noindent
In particular, we do not make assumptions about the sizes of the $B_i$ or assume that they are disjoint.
Let $\widetilde{G}(t) \sim \PDBOU(n, \sB, \eta, \tau)$.

Instead of analyzing the process through the continuous-time description of Definition~\ref{def: PDBOU}, we approach the variance identity from a combinatorial direction by conditioning on the set of blocks that have been resampled, which can be treated naturally in terms of the Poisson counting processes associated to the Poisson clocks involved.

Accordingly, we adapt the notation from Definition~\ref{def: PDBOU} to keep track of all of these processes together in a vector.
For any time $t$, we record the collection of ring counts of each clock as a vector $K(t) = (K_B(t))_{B \in \sB} \in \ZZ_{\geq 0}^{\sB}$. 
We use the letter $K$, without the time parameter, as a specific possible outcome of this vector, $K \in \ZZ_{\geq 0}^{\sB}$.
Finally, we denote by $e_B \in \ZZ_{\geq 0}^{\sB}$ the vectors of the standard basis of this space indexed by $\sB$, i.e., having $(e_B)_B = 1$ and $(e_B)_{B^{\prime}} = 0$ for all $B^{\prime} \neq B$.
We also define 
\begin{align}\label{equ: bar_K}
    \Bar{K}_{ij} \coloneqq \sum_{B \in \sB: \, (i,j) \in B}K_{B},
\end{align}
defining a matrix $\Bar{K} \in \ZZ^{n \times n}_{\sym}$ associated to a $K \in \ZZ^{\sB}$ describing the number of rings of clocks of all blocks containing each entry $(i, j) \in [n] \times [n]$.
Note that by the symmetry assumption of the blocks of $\sB$, this will indeed be a symmetric matrix.

To a fixed $K \in \ZZ_{\geq 0}^{\sB}$, we may then associate a partially resampled matrix $G(K) \in \RR^{n \times n}_{\sym}$ with entries
\begin{align}\label{equ: G_ij_K}
    G(K)_{ij} = e^{-\tau 
    \Bar{K}_{ij}}  G_{ij} + e^{-\tau \Bar{K}_{ij}} W(e^{2\tau \Bar{K}_{ij}} -1),
\end{align}
where $G, G^{\prime} \sim \GOE(n)$ are independent as in the definition of the OU process (Definition~\ref{def: OU_process}).
Then, conditional on $K(t) = K$, $\widetilde{G}(t)$ has the law of $G(K)$.
The variance identity is then as follows.

\begin{lemma}[Variance identity for PDBOU process]\label{thm: var_identity_poi_OU}
Let $\sB$ be a covering of $[n] \times [n]$ and $\widetilde{G}(t) \sim \PDBOU(n, \sB, \eta, \tau)$.
For any absolutely continuous function $f: \RR^{n \times n}_{\sym} \to \RR $ with gradient $\nabla f$ defined almost everywhere and  $\EE\|\nabla f({G(0)})\|^2] < \infty$, the Dirichlet form is given by
\begin{align}
    D(f,P_tf) 
    &=\frac{\eta}{2} \sum_{B \in \mathcal{B} } \sum_{K \in \ZZ_{\geq 0}^{\sB}} \left( \prod_{C \in \mathcal{B}} e^{-\eta  t}  \frac{{(\eta t)}^{K_{C}}}{K_{C}!} \right) \EE \left[ \Delta_{B} f \Delta_{B} f^K \right] \label{equ: D_use_POU_c},
\end{align}
where \begin{align}
  \Delta_{B} f &\coloneqq f(G(0)) - f(G(e_B)), \label{equ:K_B_f} \\
  \Delta_{B} f^K &\coloneqq f(G(K)) -f(G(K+e_{B})).  \label{equ:K_B_f_K}
\end{align}
Further, $D(f, P_tf)$ is a non-negative, non-increasing function of $t$, and
\begin{align}
  \Var(f(G)) 
  &=    \frac{\eta}{2} \sum_{B \in \mathcal{B}} \sum_{K \in \ZZ_{\geq 0}^{\sB}} \left( \int_0^\infty \PP[K(t) = K] \, dt \right) \EE \left[ \Delta_{B} f \Delta_{B} f^K \right] \label{eq:var_identity_POU_1} \\
  &=  \frac{1}{2}  \sum_{N=0}^{\infty}\frac{N!}{m^{N+1}}   \sum_{B \in \mathcal{B}}  \sum_{\substack{K\in\ZZ_{\geq 0}^{\sB}\\ |K| =N}} \frac{1}{K!}  \EE \left[ \Delta_{B} f \Delta_{B} f^K \right]. \label{eq:var_identity_POU_2}
    \end{align}
\end{lemma}
\noindent
We emphasize that the formula \eqref{equ: D_use_POU_c} for the Dirichlet form is what we use in the proof of Theorem~\ref{thm: main_POU} and the closed forms for the variance identity are given for the sake of completeness.
We give the proof in Appendix~\ref{sec: variance_identity}. 

\subsubsection{Poisson-driven block resampling process}\label{prel: independent_resample}

Finally, we give the variance identity for a process related to the block resampling described in Definition~\ref{def:block-resample}, under the same relaxation of the block structure $\sB$ as above in Section~\ref{prel: POU}.
Recall that Definition~\ref{def:block-resample} itself did not refer to a Markov process; rather, we will construct such a process as a tool for the analysis of block resampling.
Here we follow an idea of Chatterjee's for similar purposes, as described in \cite[Chapter 7]{chatterjeeSuperconcentration}.

This auxiliary process, which we define below, is similar to the PDBOU, except that we fully resample the entries in a block whenever its clock rings.
Below and in this whole section, we let $\mu$ be the law of some sub-Gaussian generalized Wigner matrix, which will be the stationary measure of our process.
\begin{definition}[Poisson-driven block resampling process]\label{def: PDBR}
Let $\sB = \{B_1, \dots, B_m\}$ be a covering of $[n] \times [n]$.
To each $B \in \sB$, associate an independent Poisson clock with rate $1$.
Let $K_B(t): \RR_{\geq 0} \to \ZZ_{\geq 0}$ be the associated counting process and let $\bar{K}(t) \in \ZZ_{\sym}^{n \times n}$ be constructed from this $K(t) \in \ZZ_{\geq 0}^{\sB}$ as in \eqref{equ: bar_K}.
Let $X = X^{(0)}, X^{(1)}, X^{(2)}, \dots \sim \mu$ be countably many independent copies of $X$.
The \emph{Poisson-driven block resampling (PDBR) process} with parameters $\sB, \mu$ is the stochastic process
\[ X(t)_{ij} = X^{(\bar{K}_{ij}(t))}_{ij}. \]
In this case we write $X(t) \sim \PDBR(\sB, \mu)$.
\end{definition}

Let us mention a few technical details concerning this process.
First, the stationary distribution associated to this process is clearly $\mu$, and in particular, unlike the cases of the OU and PDBOU processes, is no longer necessarily $\GOE(n)$.
Also, as the PDBR is a pure jump process, the domain of its generator $\sL$ is $\Dom(\sL) = L^2(\mu)$, so the issues of restricting to the domain that we have circumvented above do not appear here.

   \begin{lemma}[Variance identity for PDBR process]\label{thm: var_identity}
   Let $\sB$ be as above and let $\mu$ be the law of a sub-Gaussian generalized Wigner matrix. 
   For any $f \in L^2(\mu)$, the Dirichlet form of $\PDBR(\sB, \mu)$ satisfies
   \begin{align}
         D(f,P_t f)  &= \frac{1}{2} \sum_{k=0}^{m-1} (1 - e^{-t})^{k}  e^{-t(m - k)} \sum_{B \in  \mathcal{B}}  \sum_{A \in \mathcal{A}_{k, B}} n_k(A) \cdot \EE\left[ \Delta_{B} f \Delta_{B} f^A \right],
   \end{align}
   where we denote \begin{align} 
      \mathcal{A}_{k, B} &\coloneqq \{A \in \mathcal{A}_{k} : B \not\subseteq A \}, \\
       \mathbbm{1}(K)_{B} &\colonequals \mathbbm{1}\{{K}_{B} \geq 1 \}, \\
     n_k(A) &\coloneqq \# \left\{ \mathbbm{1} K \in \{0,1\}^{\sB} : |\mathbbm{1} K| = k,  \bigcup_{B: \mathbbm{1} K_B = 1} B = A \right\},\\
      \Delta_{B} f  &\coloneqq f(X) - f(X^{ B}), \label{equ:D_B_1}\\
      \Delta_{B} f^A &\coloneqq f(X^A) - f(X^{A \cup B}). \label{equ:D_B_2}
    \end{align}
    Then, the above expression is a non-negative, non-increasing function of $t$.
   Furthermore,
    \begin{align}\label{equ: var_identity_dij_c}
        \Var (f(X) )
= \frac{1}{2m} \sum_{k=0}^{m-1} \frac{1}{\binom{m-1}{k}}  \sum_{B \in  \mathcal{B}}  \sum_{A \in \mathcal{A}_{k, B}} n_k(A) \cdot \EE\left[ \Delta_{B} f \Delta_{B} f^A \right].
    \end{align}
\end{lemma}
\noindent
We give the proof in Appendix~\ref{sec: variance_identity}.
We note that \eqref{equ: var_identity_dij_c} is all we will use in our calculations; we give the expression for the Dirichlet form just to emphasize its similarity to \eqref{equ: var_identity_dij_c}.
A version of the identity~\eqref{equ: var_identity_dij_c} also appears in \cite{bordenave2020noise} and in \cite[Lemma 7.8]{chatterjeeSuperconcentration} for the case of resampling individual coordinates of random vectors. 

When the blocks of $\sB$ are disjoint, then our identity above may also be derived from such a version by just viewing these blocks as ``coarsened'' inputs into the function $f$; we give the more general version to emphasize that this part of the analyze can still be applied to overlapping blocks.
Generalizing \cite[Lemmas 7.9 and 7.10]{chatterjeeSuperconcentration} in the same way, we find the following estimate on the summands in terms of the variance of $f$:

\begin{cor}\label{cor: T_k_mono}
    Let $f \in L^2(\mu)$.
   For each $k = 0, \dots, m - 1$, define \begin{align}\label{equ: def_Tk}
       T_k \coloneqq \frac{1}{\binom{m-1}{k}} \sum_{B \in  \mathcal{B}}   \sum_{A \in \mathcal{A}_{k, B}} n_k(A) \cdot \EE\left[ \Delta_{B} f  \Delta_{B} f^A \right].
   \end{align}
    Then,
   \begin{align}
       T_0 \geq T_1 \geq \dots \geq T_{m-1} \geq 0.
   \end{align}
   In addition, for each $0 \leq k \leq m-1$, \begin{align}
       T_k \leq \frac{2m}{k+1} \Var(f(X)).
   \end{align}
\end{cor}
\noindent
We will see later that these $T_k$ are convenient to use as they involve only resampling exactly $k$ blocks, the same setting that appears in our original description of block resampling in Definition~\ref{def:block-resample}.
In this way, we may use the PDBR process as a ``bridge'' between our simpler block resampling process and the variance identities we obtain from Markov processes.

Note that when the blocks of $\sB$ are disjoint, then each $A \in \mathcal{A}_k$ is obtained as a unique union of $k$ blocks, and hence $n_k(A) = 1$.
In this case, the quantities $T_k$ above reduce to ones fully analogous to those in \cite{chatterjeeSuperconcentration}.

\subsection{Differentiating eigenvalues and eigenvectors}
\label{sec:diff-eig}

The following well-known result gives the formulas for the first two derivatives of an eigenvalue of a matrix with respect to its entries, which will be used repeatedly throughout the paper.

\begin{prop}[{\cite[Theorems 1 and 4]{371a9e14-3dd5-3f4a-b3ef-88e02cbdde26}}]\label{prop: derivative}
    Let $X \in \RR^{n \times n}_{\sym}$ have $\lambda_{\alpha}(X)$ being a simple eigenvalue (i.e., with multiplicity 1).
    Then, $\lambda_{\alpha}: \RR^{n \times n}_{\sym} \to \RR$, viewed equivalently as a function of the vector $(X_{ij})_{1 \leq i \leq j \leq n} \in \RR^{n(n + 1) / 2}$, is a smooth function in an open neighborhood of $X$, and its first two derivatives at $X$ are
    \begin{align}
         \partial_{ij} \lambda_{\alpha}(X) &\colonequals \frac{\partial\lambda_\alpha}{\partial X_{ij}}(X) \\
         &=  (v_\alpha (X) )_i (v_\alpha (X) )_j, \label{equ: derivative_formula} \\
         \partial_{ij} \partial_{ab} \lambda_{\alpha}(X) 
         &\colonequals  \frac{\partial^2\lambda_\alpha}{\partial X_{ij} \partial X_{ab}}(X) \\ 
         &= (\lambda_\alpha I - X )^+_{ja} \cdot (v_\alpha (X) )_i (v_\alpha (X) )_b  + (\lambda_\alpha I - X )^+_{bi} \cdot (v_\alpha (X) )_j (v_\alpha (X) )_a. \label{equ: derivative_formula_2}
    \end{align}
    Here, $Y^+$ denotes the Moore-Penrose pseudoinverse of a matrix.
    The pseudoinverse in the result may further be expanded as
     \begin{align}
    (\lambda_\alpha I - X )^+ &= \sum_{\beta \neq \alpha} \frac{1}{\lambda_\alpha(X) - \lambda_\beta(X)}v_\beta(X)v_\beta(X)^T, \\
    \left((\lambda_\alpha I - X )^+\right)_{ij} &= \sum_{\beta \neq \alpha} \frac{(v_{\beta}(X))_i (v_{\beta}(X))_j}{\lambda_\alpha(X) - \lambda_\beta(X)}. \label{equ: moore-penrose}
\end{align}
\end{prop}

\subsection{Concentration inequalities for sub-Gaussian matrices}

In this section we record concentration bounds for several matrix norms that will be used throughout.

\begin{lemma}[Operator norm bound {\cite[Theorem~2.4]{rudelson2015delocalization}}]\label{lem: norm_inequality}
Let $X$ be a sub-Gaussian generalized Wigner matrix with parameters $(c_1,c_2,K)$. There exists a constant $C=C(K,c_1)$ such that for any $s\ge 1$,
\begin{align}\label{equ: norm_inequality}
    \PP\!\left(\|X\|> C s \sqrt{n}\right)\le 2\exp\!\left(-2 s^2 n\right).
\end{align}
\end{lemma}

\begin{lemma}[Entrywise $\ell^{\infty}$ norm bound {\cite[Section~2.7.3]{vershynin2018high}}]\label{lem: infty_norm_inequality}
Let $X$ be a sub-Gaussian generalized Wigner matrix with parameters $(c_1, c_2, K)$. There exists an absolute constant $c>0$ such that for any $t\ge 0$,
\begin{align}\label{equ: infty_1}
    \PP\left(\|X\|_{\ell^{\infty}} > t\right)\le 2 n^2 \exp\left(-\frac{c t^2}{K^2}\right).
\end{align}
In particular, if $X_{ij} \sim N(0, \sigma_{ij}^2)$ with $\sigma_{ij}^2 \leq \sigma^2$, then \begin{align}\label{equ: infty_2}
    \PP\left(\|X\|_{\ell^{\infty}} > t\right)\le 2 n^2 \exp\left(-\frac{ t^2}{2\sigma^2}\right).
\end{align}
If $\|X\|_0 \leq N$ almost surely for some $N \in \NN$, then the factor $n^2$ may be replaced by $N$.
\end{lemma}

\begin{cor}[Entrywise maximum moment bounds]\label{cor: moment_bounds}
Under the assumptions of Lemma~\ref{lem: infty_norm_inequality}, there exists an absolute constant $c>0$ such that
\begin{align}
    \EE \| X\|_{\ell^{\infty}}^2 &\leq \frac{K^2}{c} \bigg( \log(n)+1 \bigg),\\
    \EE \| X\|_{\ell^{\infty}}^4 &\leq \frac{K^4}{c^2} \bigg( \log(n)+1 \bigg)^2.
\end{align}
In the Gaussian case from the Lemma (i.e., if $X_{ij} \sim N(0, \sigma_{ij}^2)$ with $\sigma_{ij}^2 \leq \sigma^2$), then we also have
   \begin{align}
       \EE \| X\|_{\ell^{\infty}}^2 &\leq 6 \sigma^2 \bigg( \log(n) + 1 \bigg) , \\
       \EE \| X\|_{\ell^{\infty}}^4 &\leq  32 \sigma^4 \bigg( \log(n)+1 \bigg)^2.
   \end{align}
\end{cor}
\noindent
We give the proof of Corollary~\ref{cor: moment_bounds} in Appendix~\ref{sec: concentration_ine}.

\subsection{Eigenvalue spacings and variances}\label{sec: eigenvalue_spacing_p}

In this section, we present known estimates concerning the locations of the eigenvalues of generalized Wigner matrices $X$ under various distributional assumptions, and prove auxiliary estimates for the purpose of our main theorems. 
In particular, we focus on three main properties: (1) eigenvalue rigidity, (2) minimum eigenvalue spacing, (3) bounds on variances of individual eigenvalues.  

We call the \emph{variance profile} of a generalized Wigner matrix $X$ the collection of $\sigma_{ij}^2 = \Var(X_{ij})$.
We refer to generalized Wigner matrices under our assumptions as having \emph{normalized variances} if the variance profile satisfies
\[ \sum_{j = 1}^n \sigma_{ij}^2 = n \text{ for all } i \in [n]. \]
Recall that all of our results are under the assumption of normalized variances, but we will use this terminology to point out below where this assumption is important.
We refer to generalized Wigner matrices as having \emph{identical variances} (i.e., being a Wigner matrix) if we have the more restrictive
\[ \sigma_{ij}^2 = 1 \text{ for all } i, j \in [n]. \]
GOE matrices do not quite have identical variances, but all results for matrices with identical variances that we mention are also straightforwardly transferred to GOE matrices.

Also, in addition to our assumption that the entries are \emph{sub-Gaussian}, we further introduce the following broader conditions on the entry distributions that appear in prior work.
We say that a random matrix $X \in \RR^{n \times n}_{\sym}$ has:
\begin{itemize}
  \item \emph{sub-exponential entries} if there exist constants $C,c > 0$ such that
  \[
       \PP(|X_{ij}| > t) \le C e^{-c t} \text{ for all } t \ge 0 \text{ and all } i,j \in [n],
  \]
  \item and has \emph{sub-Weibull entries} if there exists a constant $c>0$ such that 
  \[ \EE  \left[  \exp(|X_{ij}|^c) \right]  \leq  \frac{1}{c} \text{ for all } i,j \in [n]. \]
\end{itemize}
As in our results, when considering a sequence of matrices $X = X^{(n)} \in \RR^{n \times n}_{\sym}$, we always think of the constant parameters in these definitions as not varying with $n$.
The sub-exponential condition is more general than the sub-Gaussian condition, and the sub-Weibull condition (for appropriate $c$) includes both.

For the sake of brevity, throughout this section, when discussing a given generalized Wigner matrix $X$, we denote $\lambda_\alpha \colonequals \lambda_\alpha(X)$ for all $\alpha \in [n]$. 

\subsubsection{Eigenvalue rigidity}

The semicircle limit theorem (in various versions) states that, for sequences generalized Wigner matrices $X = X^{(n)}$ with normalized variances and any of the entrywise concentration properties discussed above, almost surely the empirical distribution of eigenvalues of $X / \sqrt{n}$ converges weakly to the semicircle distribution, i.e., that having density 
\[ \varrho_{sc}(x) = \frac{1}{2\pi} \sqrt{(4-x^2)_+}. \]
\emph{Rigidity} results make much stronger statements that individual eigenvalues are all close to their \emph{classical positions}, the location on the interval $[-2, 2]$ with the same quantile with respect to the semicircle density.
The following is a useful result of this kind.

\begin{theorem}[{\cite[Theorem 2.2]{erdHos2012rigidity}}]\label{thm: rigidity_eigenvalue}
Let $X$ be a  generalized Wigner matrix with normalized variances and sub-exponential entries.
Let \( \gamma_\beta = \gamma_{\beta}(n) \) be the \emph{classical position} of the \( \beta \)th eigenvalue, defined implicitly by, for each $\beta \in [n]$,
\begin{align}\label{equ: rho_sc}
    n \int_{\gamma_\beta}^{2} \varrho_{\text{sc}}(x)  dx = \beta.
\end{align} 
   Then, there exist constants $A_0>1$, $c, C > 0$, and $0<\phi<1$ depending only on the constants in the sub-exponential assumption such that, for any $L$ with \begin{align}\label{equ: thm_L}
        A_0 \log \log n \leq L \leq \frac{\log(10n)}{10 \log \log n},
   \end{align} 
   we have \begin{align}\label{equ: rigidity_prob}
       \mathbb{P} \left[|\lambda_\beta-\sqrt{n}\gamma_\beta| \ge (\log n)^L\hat{\beta}^{-1/3}n^{-1/6} \text{ for some } \beta \in [n] \right]
\le C \exp\left(-c(\log n)^{\phi L}\right),
   \end{align}
   for all sufficiently large $n$.
\end{theorem}

Note, however, that the spacing between the classical positions $\sqrt{n}\gamma_{\beta}$ is, say for bulk indices $\beta \in [\epsilon n, (1 - \epsilon)n]$, of order $\Theta(n^{-1/2})$.
Thus, the typical fluctuations from the classical positions can be greater than the spacing between the classical positions, so this result does not directly imply anything non-trivial about the spacing of consecutive eigenvalues (the same happens at the edges with different exponents).
However, it will still be useful to use it to control the spacing of eigenvalues with indices sufficiently separated: if $|\alpha - \beta|$ is large enough, then we may safely estimate $\lambda_\alpha$ by $\gamma_\alpha$ with high probability and small error and likewise $\lambda_{\beta}$ by $\gamma_{\beta}$, and thereby obtain a lower bound on $|{\lambda}_\alpha- {\lambda}_\beta |$. 
We obtain the following:

\begin{cor}\label{cor: lambda_i-lambda_j}  In the setting of Theorem~\ref{thm: rigidity_eigenvalue}, there exist absolute constants $c^{\prime}, C^{\prime} > 0$ such that, for each $\alpha \in [n]$, 
\begin{align}\label{equ: lambdai-lambdaj}
    &\PP\left[ |{\lambda}_\alpha- {\lambda}_\beta | \geq c^{\prime}    |\beta - \alpha|  n^{-1/2} \text{ for all } \alpha, \beta \in [n] \text{ with } |\alpha - \beta| \geq C^{\prime}(\log n)^L\right] \\
    &\hspace{1cm} \geq 1 - C \exp\left( - c (\log n)^{\phi  L} \right),
\end{align}
for all sufficiently large $n$.
\end{cor}
\noindent
We give the proof of Corollary~\ref{cor: lambda_i-lambda_j} in Appendix~\ref{sec: eigidity_eigenvalues}.

\subsubsection{Minimum eigenvalue spacings}\label{prel: minimal_spacing}

As mentioned above, general rigidity estimates do not give control of the gaps between consecutive eigenvalues.
We now present some other results that do give such control.
We begin with the sub-Gaussian Wigner matrix case, where Nguyen, Tao, and Vu \cite{nguyen2017random} obtained the following strong result.

\begin{lemma}[{\cite[Corollary 2.2]{nguyen2017random}}]\label{lem: min_eigenvalue_spacing_2} Let $X$ be a generalized Wigner matrix with identical variances and sub-Gaussian entries. For any $\delta_0 > 0$, there exists a constant $C = C(\delta_0)$ such that, for all $\delta \in (0, \delta_0)$ and any $\alpha \in [n - 1]$,
\begin{align} 
\PP \left( (\lambda_{\alpha} - \lambda_{\alpha+1})   \leq n^{-1/2-\delta} \right) \leq C n^{-\delta}.
\end{align}
\end{lemma}
\noindent
In particular, for our purposes, this result may be applied to GOE matrices.

Under the normalized variance condition and sub-exponential tail assumptions (weaker conditions than the above), Benigni and Lopatto \cite{benigni2022optimal} obtain the following bound on each consecutive eigenvalue spacing.

\begin{lemma}[\cite{benigni2022optimal}, Proposition 5.7]\label{lem: min_eigenvalue_spacing}   
 Let $X$ be a generalized Wigner matrix with normalized variances and sub-exponential entries.
There exists $\delta_0 \in (0,1)$ such that for all $\delta \in (0, \delta_0)$, there exists constants $C = C(\delta) > 0$ and $c = c  (\delta) > 0$ such that, for any $\alpha \in [n-1]$, 
\begin{align}      \PP \left(  |\lambda_\alpha -         \lambda_{\alpha+1}| < n^{-\delta - 1/6}  \hat{\alpha}^{-1/3} \right) < Cn^{-\delta -c}.    \end{align} 
\end{lemma}

We emphasize the difference in the admissible choices of $\delta$ in the two results. Under identical variances, Lemma~\ref{lem: min_eigenvalue_spacing_2} allows arbitrary $\delta > 0$. In contrast, under normalized variances, Lemma~\ref{lem: min_eigenvalue_spacing} only proves the existence of a small $\delta_0 > 0$ such that $\delta \in (0, \delta_0)$ are permitted; further, no explicit value is given for this $\delta_0$ by \cite{benigni2022optimal}. 
This lack of an explicit $\delta_0$ is the reason why, in Theorem~\ref{thm: resample_block_1}, we only assert the existence of a small constant $\delta>0$ in the generalized Wigner case, whereas in the Wigner case we may take any $\delta < 1/12$.
(The $\delta$ constant in our results has the same name as but different meaning from the $\delta$ in these results.)

\subsubsection{Variance bounds for eigenvalues}\label{prel: variance_bound}

Finally, to actually use the variance identities to establish noise sensitivity of eigenvectors, we will need to control the variances of individual eigenvalues.
Dallaporta \cite{dallaporta2012eigenvalue} obtains variance bounds for individual eigenvalues for
  Wigner matrices under an exponential tail condition and the assumption of matching four moments with the Gaussian ensembles in different regimes: at the edge, in the bulk, and in the intermediate regime.
We only bring up these results for the sake of comparison, so let us state the bound for the top eigenvalue in the GOE case.

\begin{lemma}[\cite{dallaporta2012eigenvalue}, Theorem 2]
    Let $X$ be a random matrix with i.i.d.\ sub-exponential entries whose first four moments match those of $X \sim \GOE(n)$.
    Then, there exists a constant $c>0$ depending only on the entrywise sub-exponential tail bound such that \begin{align}
         \Var(\lambda_1) \leq c n^{-1/3}. 
    \end{align}
\end{lemma}

For the specific case of $X \sim \GOE(n)$, Ledoux and Rider \cite{ledoux2010small} proved the same bound earlier as well.
Using eigenvalue rigidity as discussed in Theorem~\ref{thm: rigidity_eigenvalue}, Bordenave, Lugosi, and Zhivotovsky \cite{bordenave2020noise} extended the bound to generalized Wigner matrices with identical variances and sub-Weibull entries.

\begin{lemma}[\cite{bordenave2020noise}, Lemma 4]\label{lem: top_lambda_wigner}
Let $X$ be a generalized Wigner matrix with identical variances and sub-Weibull entries. There exists a constant $c>0$ depending only on the constants in the sub-Weibull assumption such that, for all sufficiently large $n$,
    \begin{align}
        \Var(\lambda_1) \leq c n^{-1/3}.
    \end{align}
\end{lemma}

To obtain a general variance bound for each individual eigenvalue for our sub-Gaussian generalized Wigner matrices, we derive the following general estimate, which follows from Theorem~\ref{thm: rigidity_eigenvalue}.

\begin{cor}
\label{cor: eigenvalue_bound}
Let $X$ be a sub-Gaussian generalized Wigner matrix.
There exists a constant $C > 0$ depending only on the parameters of $X$ such that, for every $\alpha \in [n]$, we have
    \begin{align}\label{equ: var_bound_alpha}
        \Var({\lambda}_\alpha) \leq C  (\log n)^{C \log \log n}\hat{\alpha}^{-2/3}n^{-1/3}.
    \end{align}
    When $X$ has identical variances and $\hat{\alpha} = 1$, then we further have
    \[ \Var({\lambda}_\alpha) \leq Cn^{-1/3}. \]
    Phrased differently, for all $\alpha \in [n]$ we have
    \begin{align}
        \Var(\lambda_{\alpha}) \leq F(n, \alpha) \hat{\alpha}^{-2/3} n^{-1/3},
    \end{align}
    where $F(n, \alpha)$ is as defined in \eqref{eq:F}.
\end{cor}
\noindent
The second claim is an immediate consequence of Lemma~\ref{lem: top_lambda_wigner}, and we give the proof for the first part of Corollary~\ref{cor: eigenvalue_bound} in Appendix~\ref{sec: eigidity_eigenvalues}.
It seems likely that the extra sub-polynomial factors could be removed by similar arguments to those used in \cite{bordenave2020noise}, but, since such factors are of sub-leading order in the thresholds in our results, we do not explore such generalizations here.

\section{Sensitivity under Ornstein-Uhlenbeck dynamics: Proof of Theorem~\ref{thm: OU_process}}\label{sec: OU_proof}

Motivated by Lemma~3.5 in \cite{chatterjee2008chaos}, we first use the OU semigroup to prove the monotonicity of 
 $ \mathbb{E} \left[ \langle v_\alpha(G(0)), v_\alpha ( G(t)) \rangle^2 \right]  $ for $ t \in \RR_+$.
We note here that, for any given $t$, $G(t)$ has the law $\GOE(n)$, and since almost surely all eigenvalues of such a random matrix are simple, $v_{\alpha}(G(t))v_{\alpha}(G(t))^{\top}$ is indeed almost surely well-defined.

\begin{lemma}\label{lem: OU_mono_inner}
   For each $\alpha \in [n]$, $ \mathbb{E} \left[ \langle v_\alpha(G(0)), v_\alpha ( G(t)) \rangle^2 \right] \geq 0 $ is non-increasing in $t \in \RR_+$.
\end{lemma}

\begin{proof}[Proof of Lemma~\ref{lem: OU_mono_inner}]
Fix $\alpha \in [n]$, and for simplicity, write $v \colonequals v_\alpha$ as the proof is identical for all $\alpha \in [n]$. 
Let us also abbreviate $G \colonequals G(0)$ and $\mu \colonequals \GOE(n)$.
We first rewrite $ \mathbb{E} \left[ \langle v(G), v( G(t)) \rangle^2 \right] $ by the linearity of expectation, obtaining \begin{align}
     \mathbb{E} \left[ \langle v(G), v( G(t)) \rangle^2 \right]  &= \sum_{i,j \in [n]} \EE \left[ v_i(G)v_j(G) v_i( G(t)) v_j( G(t))  \right] \\
     &\equalscolon \sum_{i,j \in [n]} \EE \left[ f_{ij}(G)   f_{ij}( G(t))\right]. \label{equ: inner_fij}
\end{align}
     Then, for each $i,j \in [n]$, since $P_t = P_{t/2}^2$ by the semigroup property, we have
     \begin{align}
        \EE \left[ f_{ij}(G)   f_{ij}( G(t))\right] 
        &= \langle f_{ij}, P_t f_{ij} \rangle_{\mu} \\
        &= \langle P_{t/2}f_{ij}, P_{t/2} f_{ij} \rangle_{\mu} \\
        &= \Ex_{G \sim \mu} \left[ \left( P_{t/2}  f_{ij}(G) \right)^2 \right] \geq 0. \label{equ: f_ij_GGt}
    \end{align}
Use the stationarity of the process and the semigroup property $P_t P_s = P_{t+s}$ again, we also have
\begin{align}
    \EE \left[(P_t f(G))^2\right]-\EE \left[(P_{t+s} f(G))^2\right]
&=
\EE \left[(P_t f(G(s)))^2\right]
-\EE \left[(\EE \left[P_t f(G(s))\middle|G\right])^2\right] \\
&=
\EE \left[  \Var \left(P_t f(G(s))\middle|G\right)\right] \geq 0, 
\end{align}
which, together with \eqref{equ: f_ij_GGt}, shows that $ \EE \left[ f_{ij}(G)   f_{ij}( G(t))\right]$ is non-increasing in $t\in \RR_+$. 
Therefore, in \eqref{equ: inner_fij}, since each term is non-increasing, we have that $ \mathbb{E} \left[ \langle v_\alpha(G), v_\alpha ( G(t)) \rangle^2 \right] $ is non-increasing over $t \in \RR_+$.
\end{proof}

\begin{proof}[Proof of Theorem~\ref{thm: OU_process}]
Fix $\alpha \in [n]$, and for simplicity we again write $\lambda \colonequals \lambda_{\alpha}$ and $v \colonequals v_{\alpha}$. To apply Lemma~\ref{lem: cov_identity}, we first need to check that the function $\lambda(X)$ satisfies the assumptions.
By Weyl's inequality, for any $X,Y \in \RR^{n \times n}_{\sym}$, \begin{align}
     |  \lambda(X) - \lambda(Y) | =   |  \lambda(X) - \lambda\left(X - (X-Y)\right) | \leq \| X-Y\| \leq \|X - Y\|_F,
\end{align}
which implies $\lambda$ is 1-Lipschitz, and thus, absolutely continuous. In addition, $X \sim \GOE(n)$ has a simple spectrum almost surely, so we can apply the result given in \eqref{equ: derivative_formula} that $ \partial_{ij}  \lambda(G) =  (v_\alpha (G) )_i (v_\alpha (G) )_j$ and obtain \begin{align}
   \EE  [\|\nabla \lambda(G)\|^2] =\mathbb{E} \left[ \|v(G)\|^2\right] = 1 < \infty.
 \end{align}
Thus, we can apply the variance identity in Lemma~\ref{lem: cov_identity}, 
\begin{align}\label{equ: var_OU_c}
\Var(\lambda(G)) 
&=  \int_0^\infty  e^{-s}  \mathbb{E} \left[ \langle v(G), v( G(s)) \rangle^2 \right] ds \\ 
&\geq  \int_0^{t}  e^{-s}  \mathbb{E} \left[ \langle v(G), v(G(s)) \rangle^2 \right] ds 
\intertext{$ \mathbb{E} \left[ \langle v(G), v( G(t)) \rangle^2 \right] \geq 0 $ is monotone non-increasing as stated in Lemma~\ref{lem: OU_mono_inner}, whereby}  
&\geq \mathbb{E} \left[ \langle v(G), v(G(t))  \rangle^2 \right]  \int_0^{t}  e^{-s} ds \\    
&= \mathbb{E} \left[ \langle v(G), v(G(t))  \rangle^2 \right]  \left( 1 -  e^{-t}\right). 
\end{align}
Applying the upper bound on $ \operatorname{Var}(\lambda(G)) $ from Corollary~\ref{cor: eigenvalue_bound}, we have \begin{align}
    \mathbb{E} \left[ \langle v(G), v(G(t))  \rangle^2 \right]  \leq \frac{F(n, \alpha)\hat{\alpha}^{-2/3}n^{-1/3}}{ 1 -  e^{-t}},
\end{align}
and the result follows.
\end{proof}

\section{Smoothness and Taylor expansion of eigenvalues along resampling paths}\label{sec: taylor_expansion}

We now begin to gather our tools for working with block resampling processes.
As we have discussed in Section~\ref{sec: proof_technique}, our proofs for the block resampling models work by, for some $X$ and $Y$ formed by modifying the coordinates in a block $B$ of entries of $X$ (either moving along an entrywise OU process for the PDBOU process, or fully resampling for the PDBR process), performing an expansion of the difference $\lambda_{\alpha}(X) - \lambda_{\alpha}(Y)$.
Since $\lambda_{\alpha}$ is not smooth everywhere and has diverging second derivative around the points where it is not smooth, we must be more careful than we would be for uniformly smooth functions.
To do this, we consider a \emph{resampling path},
\[ X(s) = (1 - s)X + sY, \]
and consider $\lambda_{\alpha}(X(s))$.
It suffices to show that $\lambda_{\alpha}$ is smooth along this path, and to control the derivatives along the path.

Our main ``soft'' tool for establishing smoothness is as follows.
This is an intuitive consequence of the fact that the set of matrices with a repeated eigenvalue is an algebraic variety of codimension~2, but we give a careful and concrete proof in Appendix~\ref{sec: simplicity_eigenvalues}.

\begin{prop}\label{prop: F_B}
    Define \begin{align}
        E_{\sym} \colonequals \{X \in \RR^{n \times n}_{\sym}: X \text{ has a repeated eigenvalue}\}.
    \end{align}
   Fix a symmetric set of matrix entry positions $B \subseteq [n]^2$, and write $\RR^B_{\sym}$ for the set of symmetric matrices with non-zero entries only in positions in $B$. Then \begin{align}
    F_B 
&\colonequals \{(X, \Delta) \in \RR^{n \times n}_{\sym} \times \RR^{B}_{\sym}: \text{there exists } s \in \RR \text{ such that } X + s\Delta \in E_{\sym}\}
\end{align}
is contained in a proper real algebraic variety in $\RR^{n \times n}_{\sym} \times \RR^B_{\sym}$ and therefore has Lebesgue measure zero.
\end{prop}
\noindent
As a direct consequence, we obtain the following.

\begin{lemma}\label{lem: smooth_lambda}
 Suppose $X, Y$ as above are random matrices such that $Y$ differs from $X$ only on entries in some symmetric set of matrix entry positions $B \subseteq [n]^2$.
 Suppose that the vector formed from $(X_{ij})_{1 \leq i \leq j \leq n}$ and $(Y_{ij})_{1 \leq i \leq j \leq n, (i, j) \in B}$ has a joint density with respect to Lebesgue measure.
 Then, almost surely, $X(s)$ has a simple spectrum for all $s \in [0, 1]$. 
 Consequently, also almost surely, $\lambda_\alpha (X(s) )$ is smooth at all $s \in [0, 1]$ and for every $\alpha \in [n]$.
\end{lemma}
Therefore, we can obtain uniform derivative bounds together with a Taylor remainder estimate for the eigenvalues along $X(s)$.
Below we use the quantities $M(X)$ and $S_{\alpha}(X)$ defined in Section~\ref{sec: notation}.

\begin{cor}\label{cor: derivative_bound}
  Almost surely, for each $\alpha \in [n]$, $s \in [0,1]$, and $(i,j), (a,b) \in [n] \times [n]$, we have
   \begin{align}
         &|\partial_{ij}  \lambda_\alpha (X(s) ) | \leq M(X(s))^2, \label{equ: derivative_bound_1} \\
         & | \partial_{ij}  \partial_{ab}  \lambda_\alpha(X(s) ) | \leq 2 \cdot S_\alpha(X(s)) \cdot M(X(s))^4,\label{equ: derivative_bound_2}
    \end{align}
    where $ S_\alpha$ and $M$ are defined in \eqref{equ: S_alpha} and \eqref{equ: M}, respectively.
    Further, letting $F_\alpha(s) \coloneqq \lambda_\alpha (X(s))$, $F_{\alpha}$ is almost surely differentiable at $s = 0$, and we have
    \begin{align}\label{equ: lambda_M_2} 
       & \left|\lambda_\alpha(Y) - \lambda_\alpha(X) \right|  \leq  \nu(B) \| Y-X \|_{\ell^{\infty}}   \sup_{s \in [0, 1]} M( X(s) )^2,
   \\[4pt]  
   & \left|\lambda_\alpha(Y) - \lambda_\alpha(X) - F_\alpha'(0) \right|  \leq  \nu(B)^2    \| Y-X \|_{\ell^{\infty}}^2 \sup_{s \in [0, 1]}   ( S_\alpha(X(s)) \cdot M(X(s))^4 ). \label{equ: R_ij_S_a_M^d}
    \end{align}
\end{cor}

\begin{proof}
 By Lemma~\ref{lem: smooth_lambda}, for all $\alpha \in [n]$, each $s \mapsto \lambda_\alpha (X(s))$ is almost surely smooth, so $\partial_{ij}  \lambda_\alpha (X(s) ) $ and $\partial_{ij}^2  \lambda_\alpha (X(s) ) $ as in \eqref{equ: derivative_formula} exist. We first prove \eqref{equ: derivative_bound_1} and \eqref{equ: derivative_bound_2}. We write the proof for 
 $X(s)$ when $s = 0$, and the argument is identical for all $s \in [0,1]$. The pointwise derivative bounds then follow directly from the formulas \eqref{equ: derivative_formula} and \eqref{equ: derivative_formula_2} for the derivatives of an eigenvalue with respect to matrix entries:
\begin{align}\label{equ: main_inequality_1}
   \left| \partial_{ij}  \lambda_\alpha(X) \right|
    = \left| (v_\alpha (X))_i (v_\alpha (X) )_j \right|  
    \leq \max_{\beta \in [n]} \| v_\beta (X) \|_\infty^2 = M(X)^2,
\end{align} 
and
\begin{align}
   | \partial_{ij}  \partial_{ab}  \lambda_\alpha(X) |  \leq  2  \sum_{\beta \neq \alpha} \left|  \frac{1}{\lambda_\alpha(X) - \lambda_\beta(X)}  \right| \cdot  \max_{\beta \in [n]} \|v_\beta  (X) \|_\infty^4 = 2  S_\alpha(X) \cdot M(X)^4. \label{equ: second_deriv_1}
\end{align}

We now prove \eqref{equ: lambda_M_2} and \eqref{equ: R_ij_S_a_M^d}. Fix any $\alpha \in [n]$ and let $F_\alpha(s) \coloneqq \lambda_\alpha (X(s))$.
Again, by Lemma~\ref{lem: smooth_lambda}, almost surely $F_\alpha(s)$ is a smooth function of $s$.
By Taylor's theorem with the Lagrange bound on the remainder,
\begin{align}\label{equ: lagrange_taylor_f_1}
     | F_\alpha(1) - F_\alpha(0)| &\leq \sup_{s \in [0,1]} |F_\alpha'(s)|,\\
     |F_\alpha(1) - F_\alpha(0) - F_\alpha^\prime (0)| &\leq \frac{1}{2} \sup_{s \in [0,1]} |F_\alpha^{\prime \prime}(s)|. \label{equ: lagrange_taylor_f_2}
\end{align} 
Next, we calculate $F_\alpha^{\prime}(s) $ and $F_\alpha^{\prime \prime}(s)$.  
Note that, for all $(i, j) \in B$,
\begin{align}\label{equ: non_zero_dXt}
    \left( \frac{d}{ds} X(s) \right)_{ij} = \left( \frac{d}{ds} X(s) \right)_{ji} = Y_{ij} - X_{ij},
\end{align}
with all the other entries of this derivative equal to 0.
Then, by the multivariate chain rule, we have 
\begin{align}
 |F_\alpha^{\prime}(s)| 
 &= \left|\langle \nabla \lambda_\alpha (X(s)), Y-X \rangle\right| \\
 &= \left|\sum_{(i,j) \in B} \partial_{ij} \lambda_\alpha (X(s)) (Y_{ij} - X_{ij})\right|
 \intertext{There are at most $\nu(B)$ nonzero entries in $Y-X$, so}
 &\leq \nu(B) \left( \max_{(a,b) \in [n] \times [n]} |\partial_{ab}  \lambda_\alpha(X(s))| \right) \|Y-X \|_{\ell^{\infty}} 
 \intertext{Applying the bound on the first derivatives of $\lambda_{\alpha}$ given in \eqref{equ: main_inequality_1}, }
 &\leq \nu(B) M( X(s) )^2  \|Y-X \|_{\ell^{\infty}}.
    \label{equ: F_first_(0)}
    \end{align}
    Similarly, for the second derivative,
    \begin{align}  
  |F_\alpha^{\prime \prime}(s)| &= \left|\sum_{a,b = 1}^n (Y_{ab} - X_{ab}) \sum_{c,d = 1}^n \partial_{ab}  \partial_{cd} \lambda_\alpha (X(s))  (Y_{cd} - X_{cd})\right| \\
  &\leq \nu(B)^2 \cdot \left( \max_{(a,b), (c,d)  \in  [n] \times [n]} |\partial_{ab}  \partial_{cd}  \lambda_\alpha (X(s))| \right) \cdot \| Y-X \|_{\ell^{\infty}}^2  \intertext{Applying the bound on the second derivatives of $\lambda_{\alpha}$ given in \eqref{equ: second_deriv_1},} 
  &\leq 2\nu(B)^2  \cdot  S_\alpha(X(s)) \cdot M(X(s))^4 \cdot \|Y-X \|_{\ell^{\infty}}^2. \label{equ: F_second_(0)}
\end{align}
Finally, we note that by definition $ | \lambda_\alpha (Y) -  \lambda_\alpha (X)| = | F_\alpha(1) - F_\alpha(0)| $, and plug the estimates \eqref{equ: F_first_(0)} and \eqref{equ: F_second_(0)} into the formulas \eqref{equ: lagrange_taylor_f_1} and \eqref{equ: lagrange_taylor_f_2}, obtaining
\begin{align}
       & \left|\lambda_\alpha(Y) - \lambda_\alpha(X) \right|  \leq  \nu(B)  \| Y-X \|_{\ell^{\infty}}  \sup_{s \in [0, 1]} M( X(s) )^2,
   \\[4pt]   &  \left|\lambda_\alpha(Y) - \lambda_\alpha(X) - F_\alpha'(0)  \right|  \leq   \nu(B)^2    \| Y-X \|_{\ell^{\infty}}^2 \sup_{s \in [0, 1]}   ( S_\alpha(X(s)) \cdot M(X(s))^4 ),
    \end{align}
completing the proof.
\end{proof}

\section{Uniform eigenvector delocalization along resampling paths}\label{sec: eigenvector_delocal}

We have seen in the previous section that we may control the eigenvalue differences we are interested in by the quantities $M(X(s))$ and $S_{\alpha}(X(s))$ over $X(s)$ a resampling path.
In the next two sections we develop tools for controlling these two quantities, which have to do with eigenvector delocalization and eigenvalue spacing, respectively, uniformly over paths of matrices.
In this section we focus on the first quantity $M(X(s))$, which amounts to showing that eigenvectors delocalize uniformly over paths.

We adapt the approach in Section~5.3 of \cite{ajanki2017universality}, which has been used in several prior works as well to deduce delocalization from local laws, and extend their analysis of a single sub-Gaussian generalized Wigner matrix to the entire path $X(s) =  (1-s)X + sY$ over $s \in [0,1]$, for any two sub-Gaussian generalized Wigner matrices $X,Y$ (thus we have that each $X_{ij}$ and $Y_{ij}$ may be dependent, but pairs $(X_{ij}, Y_{ij})$ over different $1 \leq i \leq j \leq n$ are independent).
In particular, we do not make the demand that $Y$ and $X$ only differ in a single block.
We state our main result to this effect below in this general language, in case it may be of independent interest:

\begin{theorem}\label{thm: delocalization_all_t}
Let $X$, $Y$ be two sub-Gaussian generalized Wigner matrices (not necessarily independent) with parameters $(c_1, c_2, K)$, and let $X(s) \colonequals (1 - s)X + sY$.
Suppose that we have the condition, for some $c > 0$,
\begin{equation}
c^{-1} \leq \Var(X(s)_{ij}) \leq c \text{ for all } s \in [0, 1] \text{ and all } i, j \in [n].
\label{eq:var-asm}
\end{equation}
Then, for all $\epsilon > 0$ and $C > 0$, for all sufficiently large $n$,
\begin{align}\label{equ: delocalization_norm}
      \PP\left\{\sup_{s \in [0, 1]} \max_{\alpha \in [n]} \| v_{\alpha}( X(s) ) \|_\infty \leq n^{-1/2 + \epsilon} \right\} \geq 1 - n^{-C}.
  \end{align}
\end{theorem}

We now outline the proof strategy. We aim to extend the eigenvector delocalization from a single matrix to the entire path $X(s)$ by controlling $M(X(s_i))$, defined in \eqref{equ: M}, over a grid $0 = s_1 < \cdots < s_q = 1$. However, extending this control to the full interval via standard perturbative, e.g., the Davis-Kahan theorem, is problematic. Comparing eigenvectors $v_\alpha(X(s_i))$ and $v_\alpha(X(t))$ typically involve the spectral gap, $\lambda_\alpha(X(s_i)) -\lambda_\alpha (X(t))$. As discussed in Section~\ref{prel: minimal_spacing}, while we have explicit high-probability bounds on eigenvalue spacings for any fixed $s$, these bounds do not remain strong after taking a union over the grid $\{s_i\}_{i=1}^q$. As a result, the union bound becomes too weak for the proofs of main Theorems~\ref{thm: main_POU} and~\ref{thm: resample_block_1}. We therefore work with the resolvent $R_X(z)$, defined in Definition~\ref{def: resolvent}, as a proxy for eigenvector delocalization. 
\begin{definition}\label{def: resolvent}
    The \emph{resolvent} of $X \in \RR^{n \times n}_{\sym}$ is the function
    \[ R_X(z) = (X - zI)^{-1}, \]
    defined on $z \in \CC \setminus \RR$.
\end{definition}
Using this formalism, we claim that we can break up the task of bounding $M(X)$ into bounding $\|X\|$ and $\widetilde{M}(X; C, \eta)$ given in \eqref{equ: tilde_M}, and further allow us to extend the estimate to the whole path without loss of probabilistic control.

\begin{prop}\label{prop: tilde_M}
    Consider the parameter \begin{align}\label{equ: tilde_M}
        \widetilde{M}(X; C, \eta) \colonequals \sup_{w \in [-C, C]} \max_{i = 1}^n |R_{X / \sqrt{n}}  (w + \bm i \eta)_{ii}|.
    \end{align}
    If $\|X\| \leq C \sqrt{n}$, then for all $\alpha \in [n]$, we have \begin{align}
        M(X)^2 \leq \eta \cdot \widetilde{M}(X; C, \eta).
    \end{align}
\end{prop}

\begin{proof}[Proof of Proposition~\ref{prop: tilde_M}]
   Fix $\alpha \in [n]$. For simplicity, we write $v_{\alpha} \colonequals v_{\alpha}(X)$ and $\lambda_{\alpha} \colonequals \lambda_{\alpha}(X) $. For a general $z \in \CC$ with $\Im(z) > 0$,
    \begin{align*}
    \Im(R_X(z)_{ii}) 
    &= \Im\left(e_i^{*}\left(\sum_{\alpha = 1}^n \frac{1}{z - \lambda_{\alpha}}v_{\alpha} v_{\alpha}^{\top}\right)e_i\right) \\
    &= \sum_{\alpha = 1}^n \Im\left(\frac{1}{z - \lambda_{\alpha}}\right) (v_{\alpha})_i^2 \\
    &= \sum_{\alpha = 1}^n \frac{\Im(z)}{|z - \lambda_{\alpha}|^2} (v_{\alpha})_i^2 \\
    &\geq \frac{\Im(z)}{|z - \lambda_{\alpha}|^2} (v_{\alpha})_i^2
    \end{align*}
  where  the last part following since all terms in the sum are non-negative as $\Im(z) > 0$ by assumption. Taking $z = \lambda_{\alpha} + \cpxi\eta$ for $\eta>0$, we have \begin{align}\label{equ: v_alpha_bound_R}
        \|v_{\alpha}\|_{\infty}^2 \leq \eta \cdot \max_{i = 1}^n \Im(R_X(\lambda_{\alpha} + \cpxi\eta)_{ii}).
    \end{align}
   By taking the maximum over all $\alpha \in [n]$ in \eqref{equ: v_alpha_bound_R},  \begin{align}
        M(X)^2 \colonequals \max_{\alpha \in [n]} \|v_\alpha\|_\infty^2 
        &\leq \eta \max_{\alpha \in [n]} \max_{i = 1}^n |R_{X / \sqrt{n}}(\lambda_{\alpha}(X/\sqrt{n}) + \cpxi\eta)_{ii}| 
        \intertext{Whenever $\| X\| \leq C\sqrt{n}$, $\lambda_{\alpha}(X/\sqrt{n}) \in [-C, C]$, we have}
        &\leq \eta \sup_{w \in [-C, C]} \max_{i = 1}^n |R_{X / \sqrt{n}}  (w + \bm i \eta)_{ii}| \\
        &\colonequals \eta \cdot \widetilde{M}(X; C, \eta).
     \end{align}
\end{proof}

Bounding $\|X(s)\|$ uniformly over $s$ will be easy by concentration inequalities follows from Lemma~\ref{lem: norm_inequality}.
To bound $\widetilde{M}(X; C, \eta)$, we use the following tool.
This is a consequence of Theorem~1.13 of \cite{ajanki2017universality}, as used in their proof of their Corollary~1.14 in their Section~5.3.
\begin{theorem}\label{thm: tilde_M_p}
    Let $X \in \RR^{n \times n}_{\sym}$ have independent centered uniformly sub-Gaussian entries on and above the diagonal and satisfy the assumption \eqref{eq:var-asm} on the entrywise variances.
    Let $\gamma, C_1, C_2 > 0$ be arbitrary.
    Then, there exists $C_3 > 0$ depending only on $\gamma, C_1, C_2$, the sub-Gaussian variance proxy, and the constant in \eqref{eq:var-asm} such that
    \[ \PP\left( \sup_{\substack{z \in \CC \\ |\Re(z)| \leq C_1 \\ n^{-1 + \gamma} \leq \Im(z) \leq C_1}} \max_{i = 1}^n |R_{X / \sqrt{n}}(z)_{ii}| > C_3\right) \leq n^{-C_2}. \]
    In particular, in our notation, we also have \begin{align}\label{equ: tilde_M_p}
        \PP  (\widetilde{M}(X; C_1, n^{-1 + \gamma}) > C_3) \leq n^{-C_2}.
    \end{align}
\end{theorem}

\begin{proof}[Proof of Theorem~\ref{thm: delocalization_all_t}]
We consider $M(X(s_i))$ over a grid $0 = s_1 < \cdots < s_q = 1$ for $q$ to be chosen later.
    We first note that for $z = w + \bm i \eta$ for $\eta > 0$, \begin{align}
        \|R_X(z)\| = \|(X - zI)^{-1}\| \leq (\min_{\alpha = 1}^n |\lambda_{\alpha}(X) - z|)^{-1} \leq 1 / \Im(z).
    \end{align}
    Using the resolvent identity,
    \begin{align}\label{equ: R_X_Y_ineq}
        \|R_X(z) - R_Y(z)\|
        &= \|R_X(z)(Y - X)R_Y(z)\| \\
        &\leq \|R_X(z)\| \cdot \|R_Y(z)\| \cdot \|Y - X\| \\
        &\leq \frac{1}{\Im(z)^2} \|Y - X\|.
    \end{align}
    Then, fix $s_i$, for any $s$ that that $|s-s_i| \leq 1/(2q)$, we have \begin{align}
        &|\widetilde{M}(X(s); C, n^{-1 + \gamma}) - \widetilde{M}(X(s_i); C, n^{-1 + \gamma})| \\
        &= \left|
        \sup_{w \in [-C, C]} \max_{i = 1}^n |R_{X(s) / \sqrt{n}}  (w + \bm i \eta)_{ii}| - \sup_{w \in [-C, C]} \max_{i = 1}^n |R_{X(s_i) / \sqrt{n}}  (w + \bm i \eta)_{ii}| \right|\\
        &\leq \sup_{w \in [-C, C]} \max_{i = 1}^n |R_{X(s) / \sqrt{n}}  (w + \bm i \eta)_{ii} -  R_{X(s_i) / \sqrt{n}}  (w + \bm i \eta)_{ii}| \\
        &\leq \sup_{w \in [-C, C]} \| R_{X(s) / \sqrt{n}}  (w + \bm i \eta) - R_{X(s_i) / \sqrt{n}}  (w + \bm i \eta) \| \intertext{Applying the bound given in \eqref{equ: R_X_Y_ineq},}
        &\leq \sup_{w \in [-C, C]} \frac{1}{\eta^2} \frac{1}{\sqrt{n}}  \|X(s) - X(s_i)\| \intertext{With $\eta = n^{-1 + \gamma}$, and apply the triangle inequality that $\|X(s) - X(s_i)\| \leq |s_i-s|  \| X-Y \| \leq  (2q)^{-1} ( \|X\|+\|Y\|)$, we obtain}
        &\leq n^{3/2 - 2 \gamma}  ({2q})^{-1}   (\| X\| +\| Y\|). \label{equ: tilde_M_difference}
    \end{align}
Fix any $C_2 > 0$, by Lemma~\ref{lem: norm_inequality}, we can choose a $C_1$ sufficiently large such that, defining the event
\[ \sE_{\mathrm{op}} = \{\|X\| + \|Y\| \geq C_1 \sqrt{n} \}, \]
we have
\[ \PP(\sE_{\mathrm{op}}) \leq n^{-C_2}. \]
Then, we have
\begin{align}
    \PP(\sup_{s \in [0, 1]} \|X(s)\| \geq C_1 \sqrt{n}) &\leq \PP(\|X\| + \|Y\| \geq C_1 \sqrt{n})  = \PP(\mathcal{E}_{\mathrm{op}}^c ) \leq n^{-C_2}.
\end{align}
Then, on the event $\sE_{\mathrm{op}}$, we choose $q = C_1 n^{2 - 2 \gamma}$ such that from \eqref{equ: tilde_M_difference}, we have \begin{align}
    \sup_{s \in [0,1]} \min_{1 \leq i \leq q} |\widetilde{M}(X(s); C_1, n^{-1 + \gamma}) - \widetilde{M}(X(s_i); C_1, n^{-1 + \gamma}) | \leq \frac{1}{2},
\end{align}
and thus, \begin{align}\label{equ: sup_s_m}
    \sup_{s \in [0,1]} \widetilde{M}(X(s); C_1, n^{-1 + \gamma}) \leq \max_{1 \leq i \leq q} \widetilde{M}(X(s_i); C_1, n^{-1 + \gamma}) + \frac{1}{2}.
\end{align}Fix such choice of $C_1, C_2$, since \eqref{eq:var-asm} holds, we can apply \eqref{equ: tilde_M_p} in Theorem~\ref{thm: tilde_M_p} that for any $\gamma > 0$, there exists $C_3^i>0$ such that \begin{align}\label{equ: tilde_M_si}
        \PP \left( \widetilde{M}(X(s_i); C_1, n^{-1 + \gamma}) \geq C_3^i \right) \leq n^{-(C_2 - 2\gamma +2)}.
    \end{align}
  We choose $C_3 = \max_{i \in [q]} C_3^i$, then from \eqref{equ: sup_s_m}, we have \begin{align}
      \PP \left( \sup_{s\in[0,1]}\widetilde M(s; C_1, n^{-1 + \gamma})\ge C_3+\frac{1}{2}\right) &\leq \PP \left( (\mathcal{E}_{op}(C_2) \cap \{ \max_{1 \leq i \leq q} \widetilde M(s_i; C_1, n^{-1 + \gamma}) \leq C_3 \})^c \right) \\
      &\leq \PP \left( \mathcal{E}_{op}^c(C_2) \right) + \PP \left(  \max_{1 \leq i \leq q} \widetilde M(s_i; C_1, n^{-1 + \gamma}) \geq C_3  \right) \intertext{By taking the union bound over all $i \in \{ 1, \dots, q\}$ in \eqref{equ: tilde_M_si},}
      &\leq n^{-C_2} + C_1 n^{2 - 2\gamma} n^{-(C_2-2 \gamma+2)} \\
      &\leq (1+C_1) n^{-C_2} \\
      &\leq n^{-C}, \label{equ: P_tilde_M}
  \end{align}
    for some constant $C$ for sufficiently large $n$. With Proposition~\ref{prop: tilde_M}, \eqref{equ: P_tilde_M} implies for any $\gamma>0$, we have \begin{align}
         \PP \left( \sup_{s\in[0,1]}M(s)\ge (C_3+\frac{1}{2}) n^{-1/2+\gamma/2}\right) \leq n^{-C}.
    \end{align}
    Therefore, by adjusting the constant for sufficiently large $n$, we have the result follows.
\end{proof}

\section{Uniform eigenvalue spacing along resampling paths}\label{sec: eigenvalue_spacing}

We now continue to the second part of our estimates for working with the Taylor expansion of $\lambda_{\alpha}(X(s))$, the control of $S_{\alpha}(X(s))$ uniformly over the sampling path $X(s)$ for $s \in [0, 1]$.
Recall that the quantity $S_{\alpha}(X(s))$ depends on the spacing of the eigenvalues of $X(s)$ for eigenvalues of indices near $\alpha$; thus, this amounts to uniform control of eigenvalue spacing over the resampling path.

\begin{theorem}\label{thm: sum_of_inverse_spacing}
    Let $X, Y$ be two sub-Gaussian generalized Wigner matrices (not necessarily independent) with parameters $(c_1, c_2, K)$ such that $Y$ almost surely differs from $X$ only on entries $(i, j) \in B$ for a block $B$.
    Suppose that the vector formed from $(X_{ij})_{1 \leq i \leq j \leq n}$ and $(Y_{ij})_{1 \leq i \leq j \leq n, (i, j) \in B}$ has a joint density with respect to Lebesgue measure.
    Let $X(s) \colonequals (1 - s)X + sY$ and suppose these matrices satisfy the condition \eqref{eq:var-asm} on their entrywise variances.
     There exists $\delta_0 \in (0,1)$ such that, for all $\delta \in (0, \delta_0)$, there exist constants $C, c > 0$ such that, for every $\alpha \in [n]$, if 
     \begin{align}
         \nu(B) \leq \frac{n^{5/6 - \delta}}{\log n} \hat{\alpha}^{-1/3},
     \end{align}
     then
    \begin{align}\label{equ: thm_S_bound}
        \PP\left[\sup_{s \in [0, 1]} S_\alpha(X(s) ) \leq C  n^{1/2 + \delta}\right] \geq 1 - c n^{-\delta / 2}.
    \end{align}
    Furthermore, if $X$ is a Wigner matrix, i.e., $\sigma_{ij}^2 = 1$ for all $i, j \in [n]$, then the same holds with any $\delta_0>0$. 
\end{theorem}

\begin{remark}
By a more careful use of eigenvalue rigidity, one can obtain a sharper bound on  $S_\alpha(X(s) )$ for each fixed $\alpha \in [n]$; see \eqref{equ: sum_a}. In particular, $S_\alpha(X(s) )$ is sharper in the edge regime then in the bulk. 
For the purpose of a uniform statement, we state here only the bound \eqref{equ: thm_S_bound}.
\end{remark}
\noindent
We prove  Theorem~\ref{thm: sum_of_inverse_spacing} based on Corollary~\ref{cor: lambda_i-lambda_j} and Lemma~\ref{lem: min_eigenvalue_spacing}.

\begin{proof}[Proof of Theorem~\ref{thm: sum_of_inverse_spacing}]
Fix $1 \leq \alpha \leq \lfloor n/2 \rfloor$, so that we are working with an eigenvalue in the ``right half'' of the spectrum.
In this case, $\hat{\alpha} = \alpha$.
A symmetric argument applies to the left half.

Define constants $\delta, C, c^{\prime}, C^{\prime}, C_{X, Y} > 0$ to be chosen later.
In terms of these constants, we define the following events:
\begin{align*}
    \sE_{\mathrm{deloc}} &= \left\{\sup_{s \in [0, 1]} M(X(s)) \leq Cn^{-1/2 + \delta/4}\right\}, \\
    \sE_{\mathrm{space}, \alpha, 1} &= \left\{\Delta_{\alpha}(X) \geq n^{-1/6 - \delta/2} \alpha^{-1/3}\right\}, \\
    \sE_{\mathrm{space}, \alpha, 2} &= \left\{\text{for all } \beta \in [n] \text{ with } |\alpha - \beta| \geq C^{\prime}n^{\delta / 2},  |\lambda_{\alpha}(X) - \lambda_{\beta}(X)| \geq c^{\prime}|\beta - \alpha| n^{-1/2}\right\}, \\
    \sE_{\mathrm{norm}} &= \left\{\|Y  - X\|_{\ell^{\infty}} \leq C_{X, Y} \sqrt{\log n}\right\}, \\
    \sE_{\alpha} &\colonequals \sE_{\mathrm{deloc}} \cap \sE_{\mathrm{space}, \alpha, 1} \cap \sE_{\mathrm{space}, \alpha, 2} \cap \sE_{\mathrm{norm}}.
\end{align*}
Here $\Delta_\alpha$ is as defined in \eqref{equ: Delta_1}, \eqref{equ: Delta_alpha}, and \eqref{equ: Delta_n}.
The interpretations of the events are that $\sE_{\mathrm{deloc}}$ asks for uniform delocalization in the sense of Theorem~\ref{thm: delocalization_all_t}, $\sE_{\mathrm{space}, \alpha, 1}$ asks for the eigenvalue spacing of eigenvalues adjacent to $\lambda_{\alpha}$ to be large, $\sE_{\mathrm{space}, \alpha, 2}$ asks for the eigenvalue spacing of eigenvalues sufficiently far from $\lambda_{\alpha}$ to be large, and $\sE_{\mathrm{norm}}$ bounds the entries of $X - Y$.

We will first show that $\sup_{s \in [0, 1]} S_{\alpha}(X(s))$ is small on the event $\sE_{\alpha}$, and then will show that $\sE_{\alpha}$ occurs with high probability.
Note that, by Lemma~\ref{lem: smooth_lambda}, almost surely $\lambda_{\alpha}$ is a smooth function at $X(s)$ for all $s \in [0, 1]$.
Therefore, we may bound by the fundamental theorem of calculus
\begin{align*}
    |\lambda_{\alpha}(X(s)) - \lambda_{\alpha}(X)|
    &= \left| \int_0^s \frac{d}{ds} \lambda_{\alpha}(X(s))ds \right| \\
    &= \left| \int_0^s \langle v_{\alpha}(X(s))v_{\alpha}(X(s))^{\top}, Y - X\rangle \right| \\
    &\leq \nu(B) \|Y - X\|_{\ell^{\infty}} \sup_{s \in [0, 1]} M(X(s))^2,
\end{align*}
arguing at the end as in Corollary~\ref{cor: derivative_bound}.
Thus, on the event $\sE_{\alpha}$, for any $\alpha \neq \beta$, by triangle inequality we have
\begin{align*}
|\lambda_{\alpha}(X(s)) - \lambda_{\beta}(X(s))| 
&\geq |\lambda_{\alpha}(X) - \lambda_{\beta}(X)| - 2C^2 \nu(B) \|Y - X\|_{\ell^{\infty}} n^{-1 + \delta / 2} \\
&\geq |\lambda_{\alpha}(X) - \lambda_{\beta}(X)| - 2C^2 \nu(B) \|Y - X\|_{\ell^{\infty}} n^{-1 + \delta/2} \\
&\geq |\lambda_{\alpha}(X) - \lambda_{\beta}(X)| - 2C_{X,Y}C^2 \nu(B) (\log n)^{1/2} n^{-1 + \delta/2}
\intertext{and plugging in the bound we assume on $\nu(B)$ as well,}
&\geq |\lambda_{\alpha}(X) - \lambda_{\beta}(X)| - 2C_{X,Y}C^2 n^{-1/6-\delta / 2}\alpha^{-1/3} (\log n)^{-1/2}
\intertext{Here, since we have $|\lambda_{\alpha}(X) - \lambda_{\beta}(X)| \geq \Delta_{\alpha}(X) \geq n^{-1/6-\delta/2}\alpha^{-1/3}$ on the event $\sE_{\alpha}$, we have}
&\geq \left(1 - \frac{2C_{X,Y}C^2}{\sqrt{\log n}}\right) |\lambda_{\alpha}(X) - \lambda_{\beta}(X)| \\
&\geq \frac{1}{2}|\lambda_{\alpha}(X) - \lambda_{\beta}(X)|
\end{align*}
for $n$ sufficiently large depending only on the constants we have defined.
So, for $n$ this large, on the event $\sE_{\alpha}$, we have
\[ \sup_{s \in [0, 1]} S_{\alpha}(X(s)) \leq 2S_{\alpha}(X). \]

Thus, to bound $\sup_{s \in [0, 1]} S_{\alpha}(X(s))$ on the event $\sE_{\alpha}$, it suffices to bound $S_{\alpha}(X)$ on this event.
We have on this event that, letting $h \colonequals C^{\prime}n^{\delta / 2}$,
\begin{align}\label{equ: S_alpha_two_parts}
  S_\alpha(X)
  &=  \sum_{\beta: |\beta-\alpha| \leq h} \frac{1}{|\lambda_\alpha(X) -\lambda_\beta(X)|} + \sum_{\beta: |\beta-\alpha| > h} \frac{1}{|\lambda_\alpha(X) -\lambda_\beta(X)|} \\
  &\leq  \frac{2h}{\Delta_\alpha(X)} + \mathbbm{1}\{\alpha \geq h \} \cdot \frac{1}{c^{\prime}}n^{1/2} \sum_{\beta=1}^{\alpha-h} \frac{1}{|\beta - \alpha|} + \frac{1}{c^{\prime}}n^{1/2} \sum_{\beta=\alpha+h}^n  \frac{1}{|\beta - \alpha|} \\
  &\leq  h \cdot n^{1/6 + \delta/2} \cdot \alpha^{1/3} + \frac{1}{c^{\prime}}  n^{1/2} \left( \log \left(\frac{\alpha}{h-1} \right) + \log\left(\frac{n - \alpha}{h-1}\right) + 2 \right)
  \intertext{The bound is maximized when $\alpha = n/2$, and so we see that, for another constant $C^{\prime\prime} > 0$, we may bound}
  &\leq C^{\prime\prime} n^{1/2 + \delta}.
  \label{equ: sum_a} 
\end{align}
Thus, on the event $\sE_{\alpha}$, we also have
\[ \sup_{s \in [0, 1]} S_{\alpha}(X(s)) \leq 2C^{\prime\prime} n^{1/2 + \delta}, \]
the condition that we are trying to show holds with high probability.

It remains to show that the probability of $\sE_{\alpha}$ is large.
From the definition, it suffices to show that each of $\sE_{\mathrm{deloc}}$, $\sE_{\mathrm{space}, \alpha, 1}$, $\sE_{\mathrm{space}, \alpha, 2}$, and $\sE_{\mathrm{norm}}$ have large probability.

For $\sE_{\mathrm{deloc}}$, Theorem~\ref{thm: delocalization_all_t} gives that $\PP[\sE_{\mathrm{deloc}}^c] \leq n^{-K}$ for any $K > 0$, provided we take $n$ sufficiently large, or equivalently our constant $C$ sufficiently large.

For $\sE_{\mathrm{space}, \alpha, 1}$, Lemma~\ref{lem: min_eigenvalue_spacing} gives that $\PP[\sE_{\mathrm{space}, \alpha, 1}^c] \leq C(\delta) n^{-\delta/2}$ for any $\delta \in (0, \delta_0)$ for $\delta_0$ a constant implicit in the proof techniques of \cite{benigni2022optimal}.

For $\sE_{\mathrm{space}, \alpha, 2}$, Corollary~\ref{cor: lambda_i-lambda_j} gives that $\PP[\sE_{\mathrm{space}, \alpha, 2}^c] \leq n^{-K}$ for any $K > 0$ provided we take $n$ sufficiently large, or equivalently our constant $c^{\prime}$ sufficiently large.

Finally, for $\sE_{\mathrm{norm}}$, by sub-Gaussianity of the entries in $X$ and $Y$ we find that $\PP[\sE_{\mathrm{norm}}^c] \leq n^{-K}$ for any $K > 0$ provided we take $C_{X, Y}$ sufficiently large.

We see that the ``bottleneck'' in these bounds is in the estimate of $\PP[\sE_{\mathrm{space}, \alpha, 1}^c]$, whose bound only gives a rate of $n^{-\delta/2}$ rather than an arbitrarily fast rate of polynomial decay.
Thus, choosing the constant $K$ in the other bounds sufficiently large, we may ensure that, say, $\PP[\sE_{\alpha}^c] \leq 2C(\delta)n^{-\delta / 2}$, completing the proof.

For $X$ a Wigner matrix, the improved result follows by carrying out exactly the same argument but using Lemma~\ref{lem: min_eigenvalue_spacing_2} instead of Lemma~\ref{lem: min_eigenvalue_spacing} to control $\PP[\sE_{\mathrm{space}, \alpha, 1}^{c}]$.
\end{proof}

\section{Sensitivity under discrete dynamics}\label{sec: resampling_dynamics}

We now give the proofs of our remaining two main results, Theorems~\ref{thm: main_POU} and~\ref{thm: resample_block_1}, which both concern sensitivity of eigenvectors under entrywise dynamics involving changing blocks of entries at discrete times.

Before proceeding to the proofs, let us recall the general proof technique, as we have discussed earlier in Section~\ref{sec: proof_technique}.
The basic issue we encounter with these kinds of dynamics is that the variance identity relates $\Var(\lambda_{\alpha}(X))$ to sums of expressions of the form
\begin{equation}
\label{eq:resampling-pair-exp}
\EE[(\lambda_{\alpha}(W) - \lambda_{\alpha}(X))(\lambda_{\alpha}(Y) - \lambda_{\alpha}(Z))],
\end{equation}
where $W, X, Y, Z \in \RR^{n \times n}_{\sym}$ are matrices such that the pair $W, X$ and the pair $Y, Z$ each only differ in a block of entries.
To extract information about the eigenvectors from such expressions, we seek to approximate these discrete differences by derivatives of the function $\lambda_{\alpha}$, which per the differential identities in Section~\ref{sec:diff-eig} indeed relate to the eigenvectors.
To do this, we consider the discrete differences along a \emph{resampling path}, for instance $W(s) = (1 - s)W + sX$, take a Taylor expansion of $\lambda_{\alpha}(W(s))$, and then study the result using the general analysis from the previous two sections.

Below we give the details of how this idea is applied to each of the two specific discrete dynamics settings we consider.

\subsection{Poisson-driven block Ornstein-Uhlenbeck dynamics}

\subsubsection{Approximation of discrete differences by derivatives}

Let $\sB$ be an admissible partition of $[n] \times [n]$ and $\widetilde{G} = \widetilde{G}^{(n)} \sim \PDBOU(\sB, \eta, \tau)$. 
We recall that the variance identity associated to this process, which is given in Lemma~\ref{thm: var_identity_poi_OU}, relates $\Var(\lambda_{\alpha}(X))$ to expectations of the form
\begin{equation}
\label{eq:resampling-pair-exp-PDBOU}
\EE[\Delta_{B}\lambda_{\alpha} \Delta_{B} \lambda_{\alpha}^K],
\end{equation}
where we use the shorthand
\begin{align*}
    \Delta_{B}\lambda_{\alpha} &= \lambda_{\alpha}(G(0)) -\lambda_{\alpha}(G(e_B)), \\
    \Delta_{B}\lambda_{\alpha}^K &= \lambda_{\alpha}(G(K)) - \lambda_{\alpha}(G(K + e_B)),
\end{align*}
where $K \in \ZZ^{\sB}_{\geq 0}$, $e_B$ is the indicator vector of a given block $B \in \sB$, and $G(K)$ is as in \eqref{equ: G_ij_K}, defined upon expanding the definition for each such $K$ as
\begin{align}
    G(K)_{ij} = e^{-\tau 
    \Bar{K}_{ij}}  G_{ij} + e^{-\tau \Bar{K}_{ij}} W_{ij}(e^{2\tau \Bar{K}_{ij}} -1),
\end{align}
for $G \sim \GOE(n)$ and $W_{ij}$ a Brownian motion.
In particular, we see that \eqref{eq:resampling-pair-exp-PDBOU} is indeed of the general form \eqref{eq:resampling-pair-exp} described above, so we are justified in applying our resampling path approach.
We will obtain the following approximation:

\begin{lemma}\label{lem: delta_partial_o1}
For any $\delta, \epsilon > 0$, there exists $C = C(\delta, \epsilon) > 0$ such that the following holds for all $n$ sufficiently large.
Let $\sB$ be an admissible partition of $[n] \times [n]$ with size parameter $\nu$. Let $\widetilde{G} = \widetilde{G}^{(n)} \sim \PDBOU(\sB, \eta, \tau)$.
Then, for all $K \in \ZZ^{\sB}_{\geq 0}$ and all $\alpha \in [n]$, we have that, if
\begin{align}\label{equ:nuB_0_1}
     \nu \leq \frac{n^{5/6 - \delta}}{\log n} \hat{\alpha}^{-1/3},
\end{align}
then
    \begin{align}
      \sR(K) &\colonequals \sum_{B \in \mathcal{B}}  \left|\EE \left[ \Delta_{B} \lambda_\alpha  \Delta_{B} \lambda_\alpha^K \right] -  \sum_{(i,j) \in B } (1+ \mathbbm{1}\{i\neq j\})  \EE[\Delta_B G_{ij} \Delta_B G(K)_{ij}  ] \EE [   \partial_{ij} \lambda_\alpha  \partial_{ij}  \lambda_\alpha^K ] \right|  \label{equ: sum_expectation} \\
      &\leq C (1 - e^{-\tau}) \left(   \nu^{2+\epsilon}   n^{-1/2 + \delta+\epsilon} +  \nu^{1+\epsilon}  n^{-\delta/4+\epsilon}   \right), \label{equ: delta_partial_tau}
    \end{align}
    where we use the notations
    \begin{align}
        \Delta_B G &= G(0) - G(e_B), \\
        \Delta_B G(K) &= G(K) - G(K+e_B), \\
        \partial_{ij}\lambda_{\alpha} &= \partial_{ij}\lambda_{\alpha}(G(0)), \\
        \partial_{ij}\lambda_{\alpha}^K &= \partial_{ij}\lambda_{\alpha}(G(K)).
    \end{align}
\end{lemma}
\begin{remark}
    Note that this bound as stated here depends on the index $\alpha$: the condition on $\nu$ in \eqref{equ:nuB_0_1} involves $\alpha$, and $\nu$ also appears in the bound in \eqref{equ: delta_partial_tau}.
    When we later apply this bound, we will choose a parameter $\gamma > 0$ such that, if $\nu = O(n^{\gamma})$, then the associated error term $\sR$ is small for \emph{any} choice of $\alpha \in [n]$, to simplify the final presentation of our result.
    But, we emphasize here that one may select $\gamma$ depending on $\alpha$ to obtain slightly more precise estimates and looser conditions on $\nu$ in some cases.
\end{remark}

To prove Lemma~\ref{lem: delta_partial_o1}, we begin by establishing a decomposition of the difference we are trying to control in Proposition~\ref{prop: POU_inequality} below. This first step addresses two technical issues.

First, because there are two factors inside the expectation in the expression $\EE[\Delta_{B}\lambda_\alpha\,\Delta_{B}\lambda_\alpha^K]$, we cannot just directly apply the estimates of Corollary~\ref{cor: derivative_bound}. 
Instead, we introduce the interpolations $F_\alpha(s)\coloneqq \lambda_\alpha(G(s))$ and $F_{\alpha, K}(s)\coloneqq \lambda_\alpha(G(K,s))$, defined below in \eqref{equ: F_s} and \eqref{equ: F_K_s}, and approximate
$\EE[\Delta_{B}\lambda_\alpha\,\Delta_{B}\lambda_\alpha^K] \approx \EE[F_\alpha'(0)\,F_{\alpha, K}^{\prime}(0)]$, with the error controlled by corresponding bounds on remainders in Taylor expansions.

Second, each of these derivatives expands into a sum over matrix entries.
In particular, if the first expands into a sum over indices $(i, j)$ and the second into one over indices $(a, b)$, then the product is a summation over all four indices, while the result of Lemma~\ref{lem: delta_partial_o1} only involves one summation, associated to the ``diagonal'' terms $\{i, j\} = \{a, b\}$.
In particular, the terms of our expansion contains factors of the form $(G_{ij}-G(e_B)_{ij})(G(K)_{ab}-G(K+e_B)_{ab})$.
The expectation of such an expression is zero unless $\{i, j\} = \{a, b\}$, but these factors appear in expectations together with other factors, stopping us from reducing in this way to only the diagonal terms of the summation.
Therefore, we further compare to a \emph{decoupled} version of this summation, where we can in fact use the above reasoning, at the cost of introducing another error term.

\begin{prop}\label{prop: POU_inequality}
 Let $B \in \mathcal{B}$.
 Let $(Z, (\widetilde{W}(t))_{t \in \RR_+})$ be identically distributed to $(G, ({W}(t))_{t \in \RR_+})$, coupled such that the entries of these matrices indexed by positions in $B$ are independent while the other entries are equal.
 For each $K \in \ZZ^{\sB}_{\geq 0}$, we define
 \begin{align}
    Z(K)_{ij} = e^{-\tau 
    \Bar{K}_{ij}}  Z_{ij} + e^{-\tau \Bar{K}_{ij}} \widetilde{W}(e^{2\tau \Bar{K}_{ij}} -1),
\end{align} 
where the notations follow \eqref{equ: bar_K} and \eqref{equ: G_ij_K}.
Note that, for any $K$, $Z(K)$ is identically distributed to $G(K)$, and the two are dependent according to the above coupling.
Then, for any $\alpha \in [n]$ (see below for the dependence of these notations on $\alpha$), we have
         \begin{align}\label{equ: delta_partial_bound_g}
    &\Bigl|  \EE \left[ \Delta_{B} \lambda \Delta_{B} \lambda^K \right] -   \sum_{(i,j)\in B } (1+ \mathbbm{1}\{i\neq j\}) \EE[\Delta_B G_{ij} \Delta_B G(K)_{ij} ] \EE \left[  \partial_{ij} \lambda  \partial_{ij} \lambda^K  \right] \Bigr|  \\ 
    & \leq |  \EE [  ( {\Delta_{B} \lambda - F'(0)} )  \Delta_{B}  \lambda^K ]  | 
         +  |  \EE [ F'(0)    (\Delta_{B} \lambda^K - F_K'(0)  )  ]  | \label{equ: delta_partial_bound_0} \\
   &\hspace{0.5cm} + \left|   \sum_{(i,j), (a,b)\in B} \EE\left[ \Delta_B G_{ij} \Delta_B G(K)_{ab}  ( \partial_{ij} \lambda  \partial_{ab} \lambda^K -  \partial_{ij}  \widetilde \lambda \partial_{ab} \widetilde{\lambda}^K  )  \right] \right| \label{equ: pdbou_prop_3}
    \end{align}
    where we use the following notations, here and in the proof omitting the dependence on $\alpha$.
    \begin{align}
      G(K, s) &\colonequals (1 - s)G(K) + sG(K + e_B), \label{equ: G_K_s} \\ 
      G(s) &\colonequals G(0, s), \\
      \lambda = \lambda_{\alpha}
      &\colonequals  \lambda_{\alpha}(G), \\
      F(s) = F_{\alpha}(s) &\colonequals \lambda_{\alpha}(G(s)), \label{equ: F_s} \\
      \lambda^K = \lambda_{\alpha}^K
      &\colonequals  \lambda_{\alpha}(G(K)), \\
      F_K(s) = F_{\alpha, K}(s) &\colonequals \lambda_{\alpha}(G(K, s)), \label{equ: F_K_s} \\
      \widetilde{\lambda} = \widetilde{\lambda}_{\alpha} 
      &\colonequals \lambda_{\alpha}(Z), \\  \widetilde{\lambda}^K = \widetilde{\lambda}_{\alpha}^K
      &\colonequals \lambda_{\alpha}(Z(K)).
    \end{align}
\end{prop}

\begin{proof}
As mentioned above, we fix $\alpha \in [n]$, and omit the dependence of all quantities involved on $\alpha$ as its value does not affect the proof.

 By \eqref{equ: R_ij_S_a_M^d} and the triangle inequality, we have
\begin{align}\label{equ: delta_partial_trian_1}
    &\left| \EE[ \Delta_{B} \lambda \Delta_{B}  \lambda^K] - \EE[ F'(0)  F_K^{\prime}(0)  ] \right| \\
       &\quad \leq |  \EE [  ( {\Delta_{B} \lambda - F'(0)} )  \Delta_{B}  \lambda^K ]  | 
         +  |  \EE [ F'(0)    (\Delta_{B} \lambda^K - F_K^{\prime}(0)  )  ]  | .
\end{align}
We further establish a relationship between $\EE  [  F'(0)  F_K^{\prime}(0)  ]$ and $ \sum_{(i,j)\in B }  \EE \left[   \partial_{ij} \lambda  \partial_{ij} \lambda^K  \right] $.  We first note that, by the chain rule,
\begin{equation}\label{equ: F_FK}
   \EE  [  F'(0) F_K^{\prime}(0)  ] =  \sum_{(i,j),(a,b)\in B} \mathbb{E} \left[ \Delta_B G_{ij} \Delta_B G(K)_{ab} \partial_{ij}\lambda \partial_{ab}\lambda^K \right].
\end{equation}
Recall that $\lambda$ and $\lambda^K$ depend on $G$ and $G^{\prime}$.
We now show that, if we instead replace these by the corresponding quantities $\widetilde{\lambda}$ and $\widetilde{\lambda}^K$ depending on $Z$ and $Z^{\prime}$, then this expression reduces to precisely the one appearing in the claim.
Since $\Delta_B G_{ij} \Delta_B G(K)_{ab}$ and $\partial_{ij} 
 \widetilde{\lambda}  \partial_{ab} \widetilde{\lambda}^K $ are independent for any $(i,j), (a,b) \in B$, we have
 \begin{align}
    &\sum_{(i,j), (a,b)\in B}  \EE \left[ \Delta_B G_{ij} \Delta_B G(K)_{ab} \partial_{ij}\widetilde{ \lambda} \partial_{ab}\widetilde{\lambda}^K   \right] \\
    &= \sum_{(i,j), (a,b) \in B }  \EE [ \Delta_B G_{ij} \Delta_B G(K)_{ab} ]   \EE [ \partial_{ij} 
 \widetilde{\lambda}  \partial_{ab} \widetilde{\lambda}^K  ] 
 \intertext{By the independence of the entries, $\EE [   \Delta_B G_{ij} \Delta_B G(K)_{ab} ] = 0$ for all $\{i, j\} \neq \{a, b\}$, thus}
 &= \sum_{(i,j) \in B} (1+ \mathbbm{1}\{i\neq j\}) \EE [ \Delta_B G_{ij} \Delta_B G(K)_{ij} ]   \EE [ \partial_{ij} 
 \widetilde{\lambda}  \partial_{ij} \widetilde{ \lambda}^K  ]
 \intertext{and by the identical marginal distributions of the $Z$ and $G$ variables, we thus also have the following, removing the $Z$ variables after using them in the above manipulation:}
 &= \sum_{(i,j) \in B} (1+ \mathbbm{1}\{i\neq j\}) \EE [ \Delta_B G_{ij} \Delta_B G(K)_{ij} ]   \EE [ \partial_{ij} 
 \lambda \partial_{ij} \lambda^K  ]
\end{align}
Rearranging this calculation, we have:
\begin{align}
    &\left|   \EE  [  F'(0)  {F'}^K(0)  ]  -    \sum_{(i,j)\in B } (1+ \mathbbm{1}\{i\neq j\})  \EE [ \Delta_B G_{ij} \Delta_B G(K)_{ij} ]   \EE \left[  \partial_{ij} \lambda  \partial_{ij} \lambda^K  \right]  \right| \\
    &=  \left|   \sum_{(i,j), (a,b)\in B} \EE\left[ \Delta_B G_{ij} \Delta_B G(K)_{ab}  ( \partial_{ij} \lambda  \partial_{ab} \lambda^K -  \partial_{ij} \widetilde{ \lambda} \partial_{ab}\widetilde{ \lambda}^K  )  \right]  \right|. \label{equ: delta_partial_trian_2}
\end{align} 
Thus, applying this observation and the triangle inequality to our original expression, we find
     \begin{align}
    &\left|  \EE \left[ \Delta_{B} \lambda  \Delta_{B} \lambda^K \right] -   \sum_{(i,j)\in B }  (1+ \mathbbm{1}\{i\neq j\}) \EE [ \Delta_B G_{ij} \Delta_B G(K)_{ij} ]   \EE \left[  \partial_{ij} \lambda  \partial_{ij} \lambda^K  \right] \right| \\ 
    &\leq \left| \EE[ \Delta_{B} \lambda \Delta_{B}  \lambda^K] - \EE[ F'(0)  {F'}^K(0)  ] \right| \\&\quad + \left|   \EE  [  F'(0)  {F'}^K(0)  ]  -   \sum_{(i,j)\in B }  \EE [ \Delta_B G_{ij} \Delta_B G(K)_{ij} ]   \EE \left[  \partial_{ij} \lambda  \partial_{ij} \lambda^K  \right] \right| \\
    &\leq |  \EE [  ( {\Delta_{B} \lambda - F'(0)} )  \Delta_{B}  \lambda^K ]  | 
         +  |  \EE [ F'(0)    (\Delta_{B} \lambda^K -{F'}^K(0)  )  ]  |   \\
    &\quad + \left|   \sum_{(i,j), (a,b)\in B} \EE\left[ \Delta_B G_{ij} \Delta_B G(K)_{ab}  ( \partial_{ij} \lambda  \partial_{ab} \lambda^K -  \partial_{ij} \widetilde{ \lambda} \partial_{ab}\widetilde{\lambda}^K  )  \right]
     \right|.
\end{align}
as desired.
\end{proof}

For any $K \in \ZZ^{\sB}_{\geq 0}$ and $B \in \mathcal{B}$, we denote  \begin{align}
    \eqref{equ: pdbou_prop_3} \equalscolon \sR_1(K, B) + \sR_2(K, B) + \sR_3(K, B)  \equalscolon \sR(K, B),
\end{align} where, to recall, \begin{align}
     \sR_1(K, B) &= \left|  \EE [  ( {\Delta_{B} \lambda - F'(0)} )  \Delta_{B}  \lambda^K ]  \right| , \label{equ: R_1B} \\
     \sR_2(K, B) &= \left|  \EE [ F'(0)    (\Delta_{B} \lambda^K - F_K'(0)  )  ]  \right|, \label{equ: R_2B} \\
     \sR_3(K, B) &= \left|   \sum_{(i,j), (a,b)\in B} \EE\left[ \Delta_B G_{ij} \Delta_B G(K)_{ab}  ( \partial_{ij} \lambda  \partial_{ab} \lambda^K -  \partial_{ij}  \widetilde \lambda \partial_{ab} \widetilde{\lambda}^K  )  \right] \right| \label{equ: R_3B} .
\end{align}

Our strategy then relies on the bounds provided by Corollary~\ref{cor: derivative_bound} on derivatives and Taylor approximations of the derivatives of eigenvalues. 
All of these bounds are in terms of further derivatives of the eigenvalues.
To control these derivatives, we invoke Theorems~\ref{thm: delocalization_all_t} and~\ref{thm: sum_of_inverse_spacing}, which control with high probability the various spectral quantities appearing in these derivatives, uniformly along resampling paths.
The similar terms $\sR_1(K, B)$ and $\sR_2(K, B)$ can be estimated by directly following this plan.
For $\sR_3(K, B)$, which involves the coupling between the $G$ and $Z$ variables, we will need to work with a slightly different event that concerns both of these.

Before establishing these bounds, we first verify that our resampling path, $G(K, s)$ as defined in \eqref{equ: G_K_s}, satisfies the variance assumptions \eqref{eq:var-asm} required to apply the uniform delocalization results of Theorem~\ref{thm: delocalization_all_t}.

\begin{prop}\label{prop: delocalization_cond}
  For any $K \in \ZZ^{\sB}_{\geq 0}$, we have for all $\epsilon>0$ and $c>0$, for all sufficiently large $n$, \begin{align}
    \PP  \left\{\sup_{s \in [0, 1]} M(G(K,s)) \leq Cn^{-1/2 + \epsilon}\right\} \geq 1 - n^{-c}.
  \end{align}
\end{prop}
\begin{proof}
 We denote \begin{align}\label{equ: sigma_ij}
    \sigma^2_{ij} = 1 + \ind\{i = j\},
\end{align}
so that $G_{ij} \sim N(0, \sigma^2_{ij})$.
Recall that, by the definition of an admissible partition (Definition~\ref{def:blocks}), for any block $B \in \sB$ and any $(i, j) \in B$, we have $K_{ij} = \Bar{K}_{ij} = K_B$.
Using this, for any $K \in \ZZ^{\sB}_{\geq 0}$, we calculate for $(i,j) \in B$ that
\begin{align}
    &\Cov (G(K)_{ij},G(K+e_B)_{ij}) \\
    &=\EE [ (e^{-\tau K_B} G_{ij} + e^{-\tau K_B} W(e^{2\tau K_B} -1)_{ij}) ( e^{-\tau (K_B+1)} G_{ij} + e^{-\tau (K_B+1)}  W(e^{2\tau( K_B+1)} -1)_{ij})] \\
    &= e^{-\tau(2K_{B}+1)} \sigma^2_{ij} + e^{-\tau(2K_{B}+1)} (e^{2\tau K_B} -1) \sigma^2_{ij} \\
    &= e^{-\tau} \sigma^2_{ij}.
\end{align}
Then, expanding the variance, we have
\begin{align}
    &\Var (G(K, s)_{ij} ) \\
    &= \Var ((1-s)G(K)+sG(K+e_B) ) \\
    &= (1-s)^2  \Var(G(K)_{ij}) + s^2  \Var(G(K+e_B)_{ij}) + 2s (1-s)\Cov (G(K)_{ij},G(K+e_B)_{ij})\\
    &= (1-2s(1-s)( 1-e^{-\tau} )) \sigma_{ij}^2  \\
    &\in \left[  \frac{1+e^{-\tau}}{2}\sigma^2_{ij},  \sigma^2_{ij}\right].\label{equ: var_G_s}
\end{align}
On the other hand, for $(i,j) \notin B$, we simply have
\begin{align}
    \Var (G(K,s)_{ij} ) = \Var ( G(K)_{ij})  = \sigma^2_{ij}.\label{equ: var_G_s_2}
\end{align}
In all cases, the condition \eqref{eq:var-asm} holds for $G(K,s)$ uniformly over $s\in[0,1]$, and the conclusion then follows immediately from Theorem~\ref{thm: delocalization_all_t}.
\end{proof}

We further record the distribution of $\Delta_B G(K)_{ij}$ for $(i,j) \in B$, which will be used repeatedly for the calculation in the proof.
\begin{lemma}\label{lem: delta_variance}
    For any $K \in \ZZ^{\sB}_{\geq 0}$ and $(i,j) \in B$, we have that $\Delta_B G(K)_{ij}$ is a Gaussian random variable with mean zero and with
    \begin{align}
       \Var(\Delta_B G(K)_{ij}) = 2 (1-e^{-\tau}) \sigma_{ij}^2.
    \end{align}
\end{lemma}

\begin{proof}
That $\Delta_B G(K)_{ij}$ is a mean-zero Gaussian directly follows from the fact that it is a linear combination of the mean-zero Gaussian variables $G_{ij}$ and $W(t)$ for a suitable $t \geq 0$. Following the notation $\sigma_{ij}^2 = 1 + \One\{i = j\}$ as above, we have
    \begin{align}\label{equ: G_Geb_K_var}
    &\Var( \Delta_B G(K)_{ij}) \\
    &= \EE [(\Delta_B G(K)_{ij})^2] \\
   &= \EE [(e^{-\tau K_B} G_{ij} + e^{-\tau K_B} W(e^{2\tau K_B} -1)_{ij} - e^{-\tau (K_B+1)} G_{ij} - e^{-\tau (K_B+1)} W(e^{2\tau (K_B+1)} -1)_{ij} )^2] \\
  &= e^{-2\tau K_B} \left(  (1-e^{-\tau})^2 \Var(G_{ij}) + \Var(W(e^{2\tau K_B} -1)_{ij} - e^{-\tau} W(e^{2\tau (K_B+1)} -1) )_{ij} \right) \label{equ: var_delta_G_e} \intertext{Using the standard Brownian motion covariance kernel $\Cov(W_{ij}(s), W_{ij}(t) ) = \min(s, t) \sigma_{ij}^2 $, we then compute}
  &=e^{-2\tau K_B} \sigma_{ij}^2 \left( (1-e^{-\tau})^2  + (e^{2\tau K_B} -1) + e^{-2\tau} (e^{2\tau (K_B+1)} -1) - 2e^{-\tau}(e^{2\tau K_B} -1) 
  \right) \\
  &= 2 (1-e^{-\tau}) \sigma_{ij}^2,
  \end{align}
  as claimed.
\end{proof}

We now start moving towards the proof of Lemma~\ref{lem: delta_partial_o1} by bounding the $\sR_i$ for $i \in \{1, 2, 3\}$.
 \begin{prop}[Bounds for $\sR_1$ and $\sR_2$]\label{prop: R_bound_12}
Under the condition on $\nu(B)$ in \eqref{equ:nuB_0_1}, for any $\delta, \epsilon > 0$, there exists $C_1 = C(\epsilon, \delta) > 0$ such that
\begin{align}
      \sR_i(K, B) &\leq  C_1 (1-e^{-\tau})   \left(  (1 - e^{-\tau})^{1/2}  \nu^{3+\epsilon}  n^{-5/2 +\delta +  \epsilon}     +    \nu^{2+\epsilon}  n^{-2-\delta/4 + \epsilon}   \right) \text{ for } i \in \{1, 2\}. \label{equ: R12_bound} 
\end{align}
\end{prop}

\begin{proof}\label{proof: prop_r12}
To keep the notation concise, we write $\sR_i \colonequals \sR_i(K,B)$ for a fixed $K \in \ZZ^{\sB}_{\geq 0}$ and $B \in \sB$ in the proof. 
Note that while there is a tradeoff in the role of $\delta$ in the two terms, for the role of $\epsilon$ it suffices to show that the bound holds for $\epsilon > 0$ arbitrarily small. Fix any $\delta, \epsilon>0$.
Let $C, C' > 0$ be constants to be fixed later.
In terms of these constants, we define the following events:
\begin{align}\label{equ: R_1_events}
     \sE_{\mathrm{deloc}} &= \left\{\sup_{s \in [0, 1]} M(G(s)) \leq Cn^{-1/2 + \epsilon}\right\}\cap \left\{\sup_{s \in [0, 1]} M(G(K,s)) \leq Cn^{-1/2 + \epsilon}\right\}, \\ 
    \sE_{\mathrm{space}} &= \left\{\sup_{s \in [0,1]}  S_\alpha ( G(s) )\leq C'  n^{1/2 + \delta} \right\} \cap  \left\{\sup_{s \in [0,1]}  S_\alpha ( G(K,s) )\leq C'  n^{1/2 + \delta} \right\}.
 \end{align}
Here we note that $\sE_{\mathrm{space}}$ depends on $\alpha$, but we omit this dependence for the sake of simplicity.

We first show that both $\sE_{\mathrm{deloc}}$ and $\sE_{\mathrm{space}}$ happen with somewhat high probability.
For $\sE_{\mathrm{deloc}}$, Theorem~\ref{thm: delocalization_all_t} applies, with its condition verified above in Proposition~\ref{prop: delocalization_cond}: for any $c > 0$, we may choose $C > 0$ sufficiently large that
\[ \PP[\sE_{\mathrm{deloc}}^c] \leq 2n^{-c}. \]

For $\sE_{\mathrm{space}}$, we use Theorem~\ref{thm: sum_of_inverse_spacing} on uniform eigenvalue spacing over paths.
Since our assumption \eqref{equ:nuB_0_1} on the size $\nu$ of $B$ holds, there exists $C' = C'(\delta)>0$ such that  \begin{align}\label{equ: prob_c_1}
    \PP[\sE_{\mathrm{space}}^c ] \leq  C^{\prime\prime} n^{-\delta / 2}
\end{align}
provided $C^{\prime\prime} = C(\delta) > 0$, as detailed in Theorem~\ref{thm: sum_of_inverse_spacing}.

We begin by bounding $\sR_1$ using these estimates.
We first partition according to whether $\sE_{\mathrm{deloc}}$ and $\sE_{\mathrm{space}}$ both happen, only $\sE_{\mathrm{deloc}}$ happens, or $\sE_{\mathrm{deloc}}$ does not happen.\footnote{It may seem at first that a more natural decomposition is just into the events $\{\sE_{\mathrm{deloc}} \cap \sE_{\mathrm{space}} \} $ and $\{\sE_{\mathrm{deloc}} \cap \sE_{\mathrm{space}} \}^c$, i.e., whether both events happen or at least one does not. However, one may check that this does not give sufficiently precise control over the error terms involved, essentially because our bound on the probability with which $\sE_{\mathrm{space}}$ happens is so much looser than that for $\sE_{\mathrm{deloc}}$. For similar reasons we also cannot use the Cauchy-Schwarz inequality to control the subsequent expectations involving indicators of rare events.}
This gives, by triangle inequality,
    \begin{align}\label{equ: R_B_delta_A}
  \sR_1  
  &= |  \EE [  ( {\Delta_{B} \lambda - F'(0)} )  \Delta_{B}  \lambda^K ]  | \\ 
   &\leq \left| \EE [  ( {\Delta_{B} \lambda - F'(0)} )  \Delta_{B}  \lambda^K   \mathbbm{1}\{\sE_{\mathrm{deloc}} \cap \sE_{\mathrm{space}} \} ]  \right| \\
   &\quad + \left| \EE [  ( {\Delta_{B} \lambda - F'(0)} )  \Delta_{B}  \lambda^K   \mathbbm{1}\{\sE_{\mathrm{deloc}} \cap \sE_{\mathrm{space}}^c \} ]  \right| \\
   &\quad + \left| \EE [  ( {\Delta_{B} \lambda - F'(0)} )  \Delta_{B}  \lambda^K   \mathbbm{1}\{\sE_{\mathrm{deloc}}^c \} ]  \right| \\
   &\leq  \EE \left| {\Delta_{B} \lambda - F'(0)}\right| \cdot |\Delta_{B}  \lambda^K| \cdot  \mathbbm{1}\{\sE_{\mathrm{deloc}} \cap \sE_{\mathrm{space}} \} \\
   &\quad +  \EE \left| \Delta_{B} \lambda - F'(0)\right| \cdot |\Delta_{B}  \lambda^K| \cdot  \mathbbm{1}\{\sE_{\mathrm{deloc}} \cap \sE_{\mathrm{space}}^c \}   \\
   &\quad +  \EE \left| \Delta_{B} \lambda - F'(0) \right| \cdot |\Delta_{B}  \lambda^K| \cdot \mathbbm{1}\{\sE_{\mathrm{deloc}}^c \}. \label{equ: eplison_1_3_terms} 
   \end{align}

We then bound the three terms in \eqref{equ: eplison_1_3_terms} separately.
Below, $C_1$ is a parameter depending only on $\delta$ and $\epsilon$ (as appears in the statement of the Proposition), which we allow to vary from line to line for the sake of concisely absorbing various constants appearing in these inequalities.
For the first term, we use that $|{\Delta_{B} \lambda - F'(0)}|$ and $|\Delta_{B}  \lambda^K|$ are bounded by the results \eqref{equ: R_ij_S_a_M^d} and \eqref{equ: lambda_M_2} of Corollary~\ref{cor: derivative_bound}, respectively, which gives that
\begin{align}
      & \EE \left| {\Delta_{B} \lambda - F'(0)}\right| \cdot |\Delta_{B}  \lambda^K| \cdot  \mathbbm{1}\{\sE_{\mathrm{deloc}} \cap \sE_{\mathrm{space}} \}    \\
       &\leq \nu^{3}   \EE  \|\Delta_B G \|_{\ell^{\infty}}^2 \|\Delta_B G(K) \|_{\ell^{\infty}}  \left(\sup_{s \in [0,1]}   S_\alpha(G(s))  M(G(s))^4 \right)  \\
       &\hspace{1.5cm} \left(\sup_{s \in [0,1]} M(G(K,s))^2\right)  \mathbbm{1}\{\sE_{\mathrm{deloc}} \cap \sE_{\mathrm{space}}\}  \label{equ:E_1} \intertext{and here on the event $\sE_{\mathrm{deloc}} \cap \sE_{\mathrm{space}}$, the two supremum factors may be bounded as} 
       &\leq C_1 \nu^{3}  \cdot n^{1/2+\delta}  n^{-3 + 6 \epsilon} \cdot  \EE [ \|\Delta_B G  \|_{\ell^{\infty}}^2  \|\Delta_B G(K) \|_{\ell^{\infty}} ].   \label{equ: exp_2to3} 
       \end{align}
We then bound  $\EE [ \|\Delta_B G  \|_{\ell^{\infty}}^2  \|\Delta_B G(K) \|_{\ell^{\infty}} ]$ by Cauchy-Schwarz, \begin{align}
     \EE  \|\Delta_B G  \|_{\ell^{\infty}}^2  \|\Delta_B G(K) \|_{\ell^{\infty}}
     &\leq (\EE  \|\Delta_B G  \|_{\ell^{\infty}}^4 )^{1/2} \cdot (\EE \|\Delta_B G(K) \|_{\ell^{\infty}}^2 )^{1/2} \label{equ: G_G_K_4} \intertext{By construction, $ \Delta_B G$ and $\Delta_B G(K)$ are supported on at most $\nu$ entries. Using the calculation in Lemma~\ref{lem: delta_variance}, we can bound the expectations by Corollary~\ref{cor: moment_bounds} and obtain}
     &\leq C_1 \left( (1 - e^{-\tau})^2 ( \log(\nu)+1 )^2 \right)^{1/2} \left( (1 - e^{-\tau}) ( \log(\nu) + 1 ) \right)^{1/2} \intertext{Since $\nu \geq 2$,}
     &\leq C_1 (1 - e^{-\tau})^{3/2}  \log^{3/2}(\nu).
     \end{align}
Substituting this into \eqref{equ: exp_2to3}, \begin{align}
      \EE  \left|\Delta_{B} \lambda - F'(0)\right| \cdot |\Delta_{B}  \lambda^K| \cdot  \mathbbm{1}\{\sE_{\mathrm{deloc}} \cap \sE_{\mathrm{space}} \}  &\leq  C_1  \nu^{3}    n^{1/2+\delta}  n^{-3+6 \epsilon} \cdot (1 - e^{-\tau})^{3/2}  \log^{3/2}(\nu) \\
      &= C_1  \nu^{3}   n^{-5/2 +\delta + 6 \epsilon} \cdot (1 - e^{-\tau})^{3/2}  \log^{3/2}(\nu).\label{equ: R_delta_A_1} 
   \end{align}

 For the second term $\EE|\Delta_{B} \lambda - F'(0)| \cdot |\Delta_{B}  \lambda^K| \cdot \mathbbm{1}\{\sE_{\mathrm{deloc}} \cap \sE_{\mathrm{space}}^c \}$ in \eqref{equ: eplison_1_3_terms}, working over the event $\sE_{\mathrm{space}}^c$ now prevents us from bounding the terms involving $S_{\alpha}$ in the way we did above. Instead, we apply triangle inequality and by \eqref{equ: lambda_M_2}, \begin{align}
   |\Delta_{B} \lambda - F'(0) | \leq  |\Delta_{B} \lambda| + |F'(0)| \leq 2 \nu  \| \Delta_B G \|_{\ell^{\infty}}   \sup_{s \in [0, 1]} M( X(s) )^2 .
   \end{align}
  As before, we also have $ | \Delta_{B}  \lambda^K| \leq 2 \nu  \| \Delta_B G(K) \|_{\ell^{\infty}}  \sup_{s \in [0, 1]} M(G(K,s))^2 $, and all together we have
  \begin{align}
    &\EE \left|\Delta_{B} \lambda - F'(0)\right| \cdot |\Delta_{B}  \lambda^K| \cdot  \mathbbm{1}\{\sE_{\mathrm{deloc}} \cap \sE_{\mathrm{space}}^c \}  \\ 
    &\leq   \nu^2 \cdot  \EE   \| \Delta_B G  \|_{\ell^{\infty}}  \| \Delta_B G(K) \|_{\ell^{\infty}}  \left(\sup_{s\in [0,1]} M (G(s) )^2 \right)  \\
    &\quad \left( \sup_{s\in [0,1]} M (G(K,s) )^2 \right)  \mathbbm{1} \{\sE_{\mathrm{deloc}} \cap \sE_{\mathrm{space}}^c \} \label{equ: Rij_1_1} \intertext{As before, we do have a deterministic bound on these suprema on the event $\sE_{\mathrm{deloc}}$,} 
    &\leq C_1  \nu^2  \cdot n^{-2+4\epsilon} \cdot \EE  \| \Delta_B G  \|_{\ell^{\infty}}  \| \Delta_B G(K) \|_{\ell^{\infty}} \mathbbm{1} \{\sE_{\mathrm{deloc}} \cap \sE_{\mathrm{space}}^c \}.
\end{align}
For the remaining expectation, we apply H\"older's inequality, \begin{align}     &\EE  \| \Delta_B G  \|_{\ell^{\infty}}  \| \Delta_B G(K) \|_{\ell^{\infty}} \mathbbm{1} \{\sE_{\mathrm{deloc}} \cap \sE_{\mathrm{space}}^c \}  \\     &\leq   (\EE \| \Delta_B G  \|_{\ell^{\infty}}^4 )^{1/4} (\EE \| \Delta_B G(K) \|_{\ell^{\infty}}^4  )^{1/4} \PP( \sE_{\mathrm{deloc}} \cap \sE_{\mathrm{space}}^c )^{1/2} \intertext{Now, as before, Lemma~\ref{lem: delta_variance} and Corollary~\ref{cor: moment_bounds} give estimates on the expectations $\EE \| \Delta_B G  \|_{\ell^{\infty}}^4$ and $\EE \| \Delta_B G(K) \|_{\ell^{\infty}}^4$. The probability is bounded by \eqref{equ: prob_c_1}, and combining these we find:}
&\leq C_1 (1-e^{-\tau})  \log(\nu) n^{-\delta/4}. \label{equ: second_G_4}\end{align}
Substituting this into \eqref{equ: Rij_1_1}, we find that this term is bounded by
\begin{align}
   \EE \left|\Delta_{B} \lambda - F'(0)\right| \cdot |\Delta_{B}  \lambda^K| \cdot  \mathbbm{1}\{\sE_{\mathrm{deloc}} \cap \sE_{\mathrm{space}}^c \} \leq  C_1  \nu^2   {n^{-2-\delta/4 + 4\epsilon}} \cdot (1 - e^{-\tau})  \log(\nu). \label{equ: R_delta_A_2}
\end{align}

For the third term in \eqref{equ: eplison_1_3_terms}, we are able to bound neither the suprema over $S_{\alpha}$ nor over $M$, leaving us with only trivial bounds $\sup_{s\in [0,1]} M (G(s) )^2 \leq 1 $ and $ \sup_{s\in [0,1]} M(G(K,s) )^2 \leq 1$ if we follow the approach for the second term above.
Fortunately, the event $\sE_{\mathrm{deloc}}^c$ has sufficiently small probability to offset these poor bounds.
Following the same steps to reach \eqref{equ: Rij_1_1} and modifying appropriately, we get
\begin{align}
\EE  \left|\Delta_{B} \lambda - F'(0)\right| \cdot |\Delta_{B}  \lambda^K| \cdot  \mathbbm{1}\{\sE_{\mathrm{deloc}}^c \}  \leq C_1  \nu^2  n^{-c/2} \cdot (1 - e^{-\tau})  \log(\nu).  \label{equ: R_delta_A_3}
\end{align}

Finally, we substitute the estimates \eqref{equ: R_delta_A_1},  \eqref{equ: R_delta_A_2}, and \eqref{equ: R_delta_A_3} into \eqref{equ: eplison_1_3_terms}, which gives
\begin{align}
      \sR_1  &\leq C_1  (1-e^{-\tau})   \log(\nu) \biggl(  (1-e^{-\tau})^{1/2} \log^{1/2}(\nu)    \nu^{3}  n^{-5/2 +\delta + 6 \epsilon}     +    \nu^2  n^{-2-\delta/4 + 4\epsilon}  +     \nu^2  n^{-c/2}  \biggr)
      \intertext{To simplify, we use that $\log(\nu) = O_{\epsilon}(\nu^{\epsilon})$ for any $\epsilon > 0$. Choosing $c$ sufficiently large, we may also absorb the last summand above into the others, finding}
    &\leq C_1 (1-e^{-\tau})   \left(  (1 - e^{-\tau})^{1/2}  \nu^{3+\epsilon}  n^{-5/2 +\delta + 6 \epsilon}     +    \nu^{2+\epsilon}  n^{-2-\delta/4 + 4\epsilon}   \right). \label{equ: first_triangle_bound_1}
\end{align}
Lastly, an identical argument applies to $\sR_2$, giving the same bound and completing the proof.
\end{proof}

\begin{prop}[Bound for $\sR_3$]\label{prop: R_bound_3}
Under the condition on $\nu(B)$ in \eqref{equ:nuB_0_1}, for any $\delta, \epsilon > 0$ there exists $C_1 = C(\epsilon, \delta) > 0$ such that for arbitrarily small $\epsilon$,
\begin{align}
      \sR_3(K, B) &\leq C_1 (1-e^{-\tau})   \left(      \nu^{3+\epsilon}   n^{-5/2 +\delta +  \epsilon}  +     \nu^{2+\epsilon}  n^{-2-\delta/4 + \epsilon}   \right). \label{equ: R3_bound} 
\end{align}
\end{prop}

\begin{proof}
    We start by bounding a term of the form
    \begin{align}
     \EE\left[ \Delta_B G_{ij} \Delta_B G(K)_{ab}  ( \partial_{ij} \lambda  \partial_{ab} \lambda^K - \partial_{ij} \widetilde{ \lambda} \partial_{ab}\widetilde{ \lambda}^K  )  \right] \colonequals \sR_3^{(i,j,a,b)}
\end{align} 
    for some fixed $(i, j), (a, b) \in B$.
    We will obtain a uniform bound not depending on these indices, and at the end will use that $\sR_3$ itself is a sum of at most $\nu^2$ such terms.

    We use a similar decomposition to \eqref{equ: eplison_1_3_terms} from the previous proof:
    \begin{align}\label{equ: r_3 two_detivative}
      \left|   \partial_{ij} \lambda  \partial_{ab} \lambda^K -  \partial_{ij} \widetilde{ \lambda} \partial_{ab} \widetilde{\lambda}^K\right|  \leq | \partial_{ij} \lambda | \cdot |\partial_{ab} \lambda^K -  \partial_{ab} \widetilde{\lambda}^K| + |\partial_{ij} \lambda - \partial_{ij} \widetilde{ \lambda} | \cdot | \partial_{ab} \widetilde{\lambda}^K|.
\end{align} 
To control $\partial_{ab} \lambda^K -  \partial_{ab} \widetilde{\lambda}^K$ and $ \partial_{ij} \lambda - \partial_{ij} \widetilde{ \lambda} $ by Taylor expansion, we introduce resampling paths $ \widetilde{G}(K, s)$ and $\widetilde{G}(s)$ given below, and then define corresponding events $\widetilde \sE_{\mathrm{deloc}}$ and $\widetilde \sE_{\mathrm{space}} $ along these new paths that ensure the same type of uniform delocalization and the spacing bounds as used previously.
Given $B \in \sB$, we define  
\begin{align}\label{equ: new_path}
   \widetilde{G}(K, s) &\colonequals (1 - s)G(K) + sZ(K), \\
   \widetilde{G}(s) &\colonequals \widetilde{G}(0, s).
 \end{align}
 Fix any $\delta, \epsilon > 0$ and let constants $ C, C' > 0$ to be chosen later, we define the following events:
 \begin{align}\label{equ: R_3_event_def}
     \widetilde{\sE}_{\mathrm{deloc}} &= \left\{\sup_{s \in [0, 1]} M(\widetilde{G}(s)) \leq Cn^{-1/2 + \epsilon}\right\} \cap \left\{\sup_{s \in [0, 1]} M(\widetilde{G}(K, s)) \leq Cn^{-1/2 + \epsilon}\right\}, \\
    \widetilde{\sE}_{\mathrm{space}} &= \left\{\sup_{s \in [0,1]}  S_\alpha ( \widetilde{G}(s) )\leq C'  n^{1/2 + \delta} \right\} \cap  \left\{\sup_{s \in [0,1]}  S_\alpha ( \widetilde{G}(K, s) )\leq C'  n^{1/2 + \delta} \right\},
 \end{align}
 where we omit the dependence on $\alpha$ for $\widetilde{\sE}_{\mathrm{space}}$  for the sake of simplicity.
 Since $\widetilde{G}(s), \widetilde{G}(K, s)$ have the same distribution as $G(s), G(K,s)$, we choose $
 C, C' > 0$ consistent with those in ${\sE}_{\mathrm{deloc}}, {\sE}_{\mathrm{space}}$ and thereby apply the corresponding probability estimates directly.

 We now apply these events to $\sR_3^{(i,j,a,b)}$ to control the expectations: we partition according to whether $ \widetilde \sE_{\mathrm{deloc}}$ and $\widetilde \sE_{\mathrm{space}}$ both happen, only $\widetilde \sE_{\mathrm{deloc}}$ happens, or $\widetilde \sE_{\mathrm{deloc}}$ does not happen. By triangle inequality,
\begin{align}
 \sR_3^{(i,j,a,b)}  
  &\leq  \EE | \Delta_B G_{ij} \Delta_B G(K)_{ab} | \cdot | \partial_{ij} \lambda  \partial_{ab} \lambda^K - \partial_{ij} \widetilde{ \lambda} \partial_{ab}\widetilde{ \lambda}^K  |   \label{equ: eplison_3_absolute} \\
   &\leq  \EE  | \Delta_B G_{ij} \Delta_B G(K)_{ab} | \cdot  | \partial_{ij} \lambda  \partial_{ab} \lambda^K - \partial_{ij} \widetilde{ \lambda} \partial_{ab}\widetilde{ \lambda}^K  | \cdot \mathbbm{1}\{\widetilde{\mathcal{E}}_{\mathrm{deloc}} \cap \widetilde{\mathcal{E}}_{\mathrm{space}}\} \\ 
   &\quad  +\EE  | \Delta_B G_{ij} \Delta_B G(K)_{ab} | \cdot  | \partial_{ij} \lambda  \partial_{ab} \lambda^K - \partial_{ij} \widetilde{ \lambda} \partial_{ab}\widetilde{ \lambda}^K  | \cdot  \mathbbm{1}\{\widetilde{\mathcal{E}}_{\mathrm{deloc}} \cap \widetilde{\mathcal{E}}_{\mathrm{space}}^c \} \\&\quad + \EE  | \Delta_B G_{ij} \Delta_B G(K)_{ab} | \cdot  | \partial_{ij} \lambda  \partial_{ab} \lambda^K - \partial_{ij} \widetilde{ \lambda} \partial_{ab}\widetilde{ \lambda}^K  | \cdot  \mathbbm{1}\{\widetilde{\mathcal{E}}_{\mathrm{deloc}}^c  \}.  \label{equ: R_3_events}
\end{align}
For the first term that is on the event $\widetilde{\mathcal{E}}_{\mathrm{deloc}} \cap \widetilde{\mathcal{E}}_{\mathrm{space}}$, we first apply the bound in \eqref{equ: r_3 two_detivative}.
Then, we use the mean value theorem on $ |\partial_{ab} \lambda^K -  \partial_{ab} \widetilde{\lambda}^K|$ and $|\partial_{ij} \lambda - \partial_{ij} \widetilde{ \lambda} | $. 
In addition, by the coupling in Proposition~\ref{prop: POU_inequality}, the pairs $(G,Z)$ and $(G(K), Z(K))$ differ in at most $\nu$ entries.
Then, on the event $\widetilde{\mathcal{E}}_{\mathrm{deloc}} \cap \widetilde{\mathcal{E}}_{\mathrm{space}}$,  
\begin{align}
    & \left|   \partial_{ij} \lambda  \partial_{ab} \lambda^K -  \partial_{ij} \widetilde{ \lambda} \partial_{ab} \widetilde{\lambda}^K\right| \\
    &\leq  | \partial_{ij} \lambda |  \cdot \nu  \|G(K)- Z(K)\|_{\ell^{\infty}} \cdot \sup_{s \in [0,1]}  \max_{(c,d) \in [n] \times [n]} | \partial_{ab} \partial_{cd} \lambda ( \widetilde{G}(K, s) )| \\
 &\quad +  | \partial_{ab} \widetilde{\lambda}^K| \cdot \nu  \|G-Z\|_{\ell^{\infty}} \cdot \sup_{s \in [0,1]}  \max_{(c,d) \in [n] \times [n]} | \partial_{ij} \partial_{cd}   \lambda ( \widetilde{G}(s) ) | \intertext{Now, we again apply the bound on these derivatives from \eqref{equ: derivative_bound_2} in Corollary~\ref{cor: derivative_bound}, which gives}
 &\leq M(G)^2 \cdot  2\nu  \|G(K)- Z(K)\|_{\ell^{\infty}} \cdot \sup_{s \in [0,1]} S_\alpha ( \widetilde{G}(K, s) )  M( \widetilde{G}(K, s) )^4  \\
 &\quad + M(Z(K))^2 \cdot  2\nu  \|G- Z\|_{\ell^{\infty}} \cdot \sup_{s \in [0,1]} S_\alpha( \widetilde{G}(s) )  M( \widetilde{G}(s) )^4 \label{equ: mean_value_derivative}.
\end{align}
Substituting in the bounds that hold on the event $\widetilde{\mathcal{E}}_{\mathrm{deloc}} \cap \widetilde{\mathcal{E}}_{\mathrm{space}}$, \begin{align}\label{equ: first_R3}
    &  \EE  | \Delta_B G_{ij} \Delta_B G(K)_{ab} | \cdot | \partial_{ij} \lambda  \partial_{ab} \lambda^K - \partial_{ij} \widetilde{ \lambda} \partial_{ab}\widetilde{ \lambda}^K  | \cdot \mathbbm{1}\{\widetilde{\mathcal{E}}_{\mathrm{deloc}} \cap \widetilde{\mathcal{E}}_{\mathrm{space}}\}     \\
    &\leq C_1  \nu  n^{1/2+\delta} n^{-3+6\epsilon} \cdot \EE  | \Delta_B G_{ij} \Delta_B G(K)_{ab} | (\|G(K)- Z(K)\|_{\ell^{\infty}} + \|G- Z\|_{\ell^{\infty}} ).  \label{equ: G_G_entrywise}
\end{align}
We bound $ \EE  | \Delta_B G_{ij} \Delta_B G(K)_{ab} | \cdot \|G(K)- Z(K)\|_{\ell^{\infty}} $, and bounding $ \EE  | \Delta_B G_{ij} \Delta_B G(K)_{ab} | \cdot \|G- Z\|_{\ell^{\infty}}  $ can be done by an identical argument.
By H\"older's inequality, \begin{align}
  &\EE  | \Delta_B G_{ij} \Delta_B G(K)_{ab} | \cdot \|G(K)- Z(K)\|_{\ell^{\infty}} \\&\leq (\EE [(\Delta_B G_{ij})^4 ])^{1/4} ( \EE   [(\Delta_B G(K))_{ab}]^4 )^{1/4}  (\EE  \|G(K)- Z(K)\|_{\ell^{\infty}}^2)^{1/2} \intertext{By Lemma~\ref{lem: delta_variance} and Corollary~\ref{cor: moment_bounds}, we may bound these as}
  &\leq C_1 (1-e^{-\tau}) \log(\nu)^{3/2}.
\end{align}
Applying this to \eqref{equ: first_R3},
\begin{align}
&\EE  | \Delta_B G_{ij} \Delta_B G(K)_{ab} | \cdot | \partial_{ij} \lambda  \partial_{ab} \lambda^K - \partial_{ij} \widetilde{ \lambda} \partial_{ab}\widetilde{ \lambda}^K  | \cdot \mathbbm{1}\{\widetilde{\mathcal{E}}_{\mathrm{deloc}} \cap \widetilde{\mathcal{E}}_{\mathrm{space}}\} \\
&\leq    C_1  (1-e^{-\tau})  \nu^{1 + \epsilon}  n^{-5/2 +\delta + 6 \epsilon}. \label{equ: Z_first_term}  
\end{align}

For the second term in \eqref{equ: R_3_events}, instead of using the mean value theorem on differences of derivatives, we apply the triangle inequality and the bound for the first derivative given in \eqref{equ: derivative_bound_1} directly. This gives
\begin{align}
    |\partial_{ab} \lambda^K -  \partial_{ab} \widetilde{\lambda}^K| \leq  |\partial_{ab} \lambda^K| +|  \partial_{ab} \widetilde{\lambda}^K| \leq M(G(K))^2 + M(Z(K))^2 \leq   2   \sup_{s \in [0,1]} M( \widetilde G(K,s))^2 ,
\end{align}
and similarly,
\begin{align}
    |\partial_{ij} \lambda - \partial_{ij} \widetilde{ \lambda} |  \leq   2 \sup_{s \in [0,1]} M(\widetilde G(s))^2 .
\end{align}
We have \begin{align}
   &\left|   \partial_{ij} \lambda  \partial_{ab} \lambda^K -  \partial_{ij} \widetilde{ \lambda} \partial_{ab} \widetilde{\lambda}^K\right| \\
   &\leq | \partial_{ij} \lambda | \cdot |\partial_{ab} \lambda^K -  \partial_{ab} \widetilde{\lambda}^K| + |\partial_{ij} \lambda - \partial_{ij} \widetilde{ \lambda} | \cdot | \partial_{ab} \widetilde{\lambda}^K| \\
    &\leq 2 \left( \max \left\{ \sup_{s \in [0,1]} M(\widetilde G(s)), \sup_{s \in [0,1]} M(\widetilde G(K,s))\right\} \right)^4 \label{equ: tilde_difference_1}.
\end{align}
The above term is at most $C_1  n^{-2+4\epsilon}$ on the event $ \widetilde{\mathcal{E}}_{\mathrm{deloc}} \cap \widetilde{\mathcal{E}}_{\mathrm{space}}^c$, and so we have
\begin{align}\label{equ: Z_second_term}
    &  \EE  | \Delta_B G_{ij} \Delta_B G(K)_{ab} | \cdot | \partial_{ij} \lambda  \partial_{ab} \lambda^K - \partial_{ij} \widetilde{ \lambda} \partial_{ab}\widetilde{ \lambda}^K  | \cdot \mathbbm{1}\{\widetilde{\mathcal{E}}_{\mathrm{deloc}} \cap \widetilde{\mathcal{E}}_{\mathrm{space}}^c \}   \\
    &\leq C_1    n^{-2+4\epsilon} (\EE [(\Delta_B G_{ij})^4 ])^{1/4} ( \EE   [(\Delta_B G(K)_{ab})^4] )^{1/4}  \PP(\widetilde{\mathcal{E}}_{\mathrm{deloc}} \cap \widetilde{\mathcal{E}}_{\mathrm{space}}^c  )^{1/2} \\
     &\leq C_1    n^{-2+4\epsilon} n^{-\delta/4} \cdot (1-e^{-\tau}) \log(\nu) \\
     &\leq C_1     n^{-2-\delta/4 + 4\epsilon} \cdot (1-e^{-\tau})  \nu^\epsilon
\end{align}

For the last term in \eqref{equ: R_3_events}, we follow the same argument as we used to obtain \eqref{equ: R_delta_A_3}, which here gives
\begin{align}\label{equ: Z_third_term}
    \EE  | \Delta_B G_{ij} \Delta_B G(K)_{ab} | \cdot  | \partial_{ij} \lambda  \partial_{ab} \lambda^K - \partial_{ij} \widetilde{ \lambda} \partial_{ab}\widetilde{ \lambda}^K  | \cdot  \mathbbm{1}\{\widetilde{\mathcal{E}}_{\mathrm{deloc}}^c  \} \leq C_1   n^{-c/2} \cdot (1-e^{-\tau}) \nu^\epsilon.
\end{align}

As mentioned before, we may bound the expression we were originally interested in as
\begin{align*}
    \sR_3 
    &\leq \nu^2 \cdot \max_{(i, j), (a, b) \in B} R_3^{(i,j,a,b)}
    \intertext{and now by putting together \eqref{equ: Z_first_term}, \eqref{equ: Z_second_term} and \eqref{equ: Z_third_term} we have a uniform bound on every term in the maximum, which gives}
    &\leq C_1 (1-e^{-\tau})  \left(      \nu^{3+\epsilon}   n^{-5/2 +\delta + 6 \epsilon}  +     \nu^{2+\epsilon}  n^{-2-\delta/4 + 4\epsilon}  +     \nu^{2+\epsilon} n^{-c/2} \right)
    \intertext{and, by choosing $c$ sufficiently large,}
    &\leq C_1 (1-e^{-\tau})   \left(      \nu^{3+\epsilon}   n^{-5/2 +\delta + 6 \epsilon}  +     \nu^{2+\epsilon}  n^{-2-\delta/4 + 4\epsilon}   \right),
\end{align*}
completing the proof.
\end{proof}

Lemma~\ref{lem: delta_partial_o1} then directly follows by combining Propositions~\ref{prop: POU_inequality}, \ref{prop: R_bound_12}, and~\ref{prop: R_bound_3}.

\begin{proof}[Proof of Lemma~\ref{lem: delta_partial_o1}]
    By Proposition~\ref{prop: POU_inequality}, we have
    \begin{align*}
        \sR(K)
        &\leq \sum_{B \in \sB} \sR(K, B) \\
        &\leq \sum_{B \in \sB} (\sR_1(K, B) + \sR_2(K, B) + \sR_3(K, B))
        \intertext{In Propositions~\ref{prop: R_bound_12} and~\ref{prop: R_bound_3}, we give bounds on the summands that are uniform over $B \in \sB$, so substituting these gives}
        &\leq |\sB| \cdot C (1-e^{-\tau})   \left(      \nu^{3+\epsilon}   n^{-5/2 +\delta +  \epsilon}  +     \nu^{2+\epsilon}  n^{-2-\delta/4 + \epsilon}   \right)
        \intertext{for a suitable $C = C(\delta, \epsilon)$. For an admissible partition $\sB$, we have $|\sB| \nu = n^2 + n$, so $|\sB| \leq 2n^2 / \nu$, and applying this gives}
        &\leq 2C (1 - e^{-\tau})  \left(   \nu^{2+\epsilon}   n^{-1/2 + \delta+\epsilon} +  \nu^{1+\epsilon}  n^{-\delta/4+\epsilon}   \right),
    \end{align*}
    completing the proof.
\end{proof}

\subsubsection{Proof of Theorem~\ref{thm: main_POU}}

In this section, we again omit the dependence on $\alpha$ in the proof for simplicity and note that the argument applies identically to all $\alpha \in [n]$.

Let us recall our situation now that we are equipped with Lemma~\ref{lem: delta_partial_o1}.
Together with the variance identity, the Lemma relates the variance of $\lambda_{\alpha}$ to the expression
\begin{align}\label{equ: }
    \sum_{B \in \sB} \sum_{(i,j)\in B}   (1+ \mathbbm{1}\{i\neq j\}) \EE [ \Delta_B G_{ij} \Delta_B G(K)_{ij}]     \EE [  \partial_{ij} \lambda  \partial_{ij} \lambda^K  ].
\end{align}
On the other hand, we are interested in relating the variance to $ \EE [ \langle v(\widetilde{G}(0)), v(\widetilde{G}(t)) \rangle^2 ]$, which, upon expanding and conditioning on $K$ is a sum like the above, but only involving the terms $\EE [  \partial_{ij} \lambda  \partial_{ij} \lambda^K  ]$.
In particular, we have:
\begin{prop}
    \label{prop:inner_prod_K_sum}
    For any $t \geq 0$, we have
    \begin{align}
        \EE [ \langle v(\widetilde{G}(0)), v(\widetilde{G}(t))  \rangle^2 ] 
        &= \sum_{K \in \ZZ^{\sB}_{\ge 0}} \PP(K(t) = K) \sum_{(i, j) \in [n]^2} \EE [  \partial_{ij} \lambda  \partial_{ij} \lambda^K  ]
        \intertext{and so, for any admissible partition $\sB$,}
        &= \sum_{K \in \ZZ^{\sB}_{\ge 0}} \PP(K(t) = K)  \sum_{B \in \sB} \sum_{(i,j)\in B}     \EE [  \partial_{ij} \lambda  \partial_{ij} \lambda^K  ].
    \end{align}
\end{prop}
Thus, we would like to get rid of the terms $\EE [ \Delta_B G_{ij} \Delta_B G(K)_{ij}]$.
While these can be computed in closed form, they depend non-trivially on $K$, and further we will see below that they can have different signs.
In particular, if $K = 0$ then they are clearly positive, while otherwise they will turn out to be negative (one may see this as a consequence of the mean-reverting behavior of the OU process: if $G$ moves up at an earlier time, it tends to move down at a later time and vice-versa).
So, we will have to rather carefully control these two contributions.

We begin by establishing the non-negativity of $ \EE [  \partial_{ij} \lambda  \partial_{ij} \lambda^K  ]$.

\begin{lemma}\label{lem: pos_E_derivative}
For all $\alpha \in [n]$,  $K \in \ZZ^{\sB}_{\geq 0}$ and $(i,j) \in [n] \times [n]$, we have $  \EE [  \partial_{ij} \lambda  \partial_{ij} \lambda^K  ] \geq 0$. Moreover, for any fixed $B \in \sB$, given $K_C $ fixed for all ${C \neq B}$, we have $ \EE [  \partial_{ij} \lambda  \partial_{ij} \lambda^K  ]  $ is a non-increasing function of $K_B$.
\end{lemma}

\begin{proof}

We begin by proving the non-negativity. Given the partition $\sB$, where $|\sB| = m$, we decompose the matrix $G$ into a sequence of blocks $G = (G_1, G_2, \dots G_m)$. 
Similarly, given $K$, we write \begin{align}
    G(K) =( G_1(t_1), G_2(t_2), \dots, G_m(t_m)),
\end{align} where each $t_a \geq 0$ corresponds to the specific ring count $K_a$. Then, for each block index $a$, the entries restricted to $B_a$ evolve according to the Ornstein-Uhlenbeck process for time $t_a$, independently of the other blocks\footnote{We note that while the specific values of $t_a$ may not be unique, any feasible choice suffices for the proof. }.
We define the function $f(G) \colonequals \partial_{ij} \lambda$ and $ f(G(K)) \colonequals \partial_{ij} \lambda^K$, then
    \begin{align}
     \EE [  \partial_{ij} \lambda  \partial_{ij} \lambda^K  ] &= \EE[ f(G_1,  \dots, G_m) f(G_1(t_1),  \dots, G_m(t_m))] \\
     &= \EE_{G \sim \mu} [  \EE[ f(G_1,  \dots, G_m) f(G_1(t_1),  \dots, G_m(t_m) ) \mid G ] ]\intertext{Let $P_t^{(a)}$ denote the OU process semigroup for each $B_a$, then by the independence of the blocks,}
     &= \EE [ f(G) (P_{t_1}^{(1)}  \dots P_{t_m}^{(m)} f)(G) ] \intertext{Then by the semigroup property $P_t  = P_{t/2}^2$ and reversibility,}
     &= \EE[ ( (P_{t_1/2}^{(1)} \dots P_{t_m/2}^{(m)} f)(G) )^2] \geq 0. \label{equ: semi_s}
    \end{align}
Furthermore, if we fix $K_C $ fixed for all ${C \neq B}$, then similar to \eqref{equ: semi_s}, by denoting $P_{t}^{(\sim B)} \colonequals \prod_{C \neq B} P_{t_C/2}^{(C)} $ and $g \colonequals P_{t_{-B}/2}^{(\sim B)} f$, we can write
\begin{align}
\EE [  \partial_{ij} \lambda  \partial_{ij} \lambda^K  ] &= \EE[ (P_{t_{-B}/2}^{(\sim B)} f)(G) , (P_{t_B}^{(B)} P_{t_{-B}/2}^{(\sim B)} f)(G) ] \\
&= \langle g, P_{t_B}^{(B)} g \rangle_\mu.
\end{align}
By differentiating with respect to $t_B$, we obtain
\begin{align}
\frac{d}{dt_B} \langle g, P_{t_B}^{(B)} g \rangle_\mu = \langle g, \sL P_{t_B}^{(B)} g \rangle_\mu = -D(g, P_{t_B}^{(B)} g ) \leq 0,
\end{align}
where the non-positivity follows from the non-negativity of the {Dirichlet form} given in Lemma~\ref{lem: dirichlet_mon}. Thus, $\EE [  \partial_{ij} \lambda  \partial_{ij} \lambda^K   ]$ is non-increasing in $t_B$. Since $t_B$ is a strictly increasing function of $K_B$, we have that $\EE [ \partial_{ij} \lambda \partial_{ij} \lambda^K  ]$ is non-increasing with respect to $K_B$.
\end{proof}

Also, as mentioned above, the other factor in the sum we are interested in has sign that depends on the vector $K$:
\begin{prop}
    For any $B \in \sB$, $(i, j) \in B$, and $K \in \ZZ_{\geq 0}^B$, we have
    \begin{align}
    (1+ \mathbbm{1}\{i\neq j\}) \EE [ \Delta_B G_{ij} \Delta_B G(K)_{ij}] = \begin{cases}
        4(1 - e^{-\tau}), &\text{for } K_B = 0; \\
        - 2(1-e^{-\tau})^2 e^{-\tau(K_B-1)}, &\text{for } K_B \geq 1.
    \end{cases}
\end{align}
\end{prop}
\begin{proof}
    We calculate
\begin{align}
    &   \EE [ \Delta_B G_{ij} \Delta_B G(K)_{ij}] \\
    &=  e^{-\tau K_B} \bigg( (1-e^{-\tau})^2\EE[G_{ij}^2] - e^{-\tau} \EE[ W(e^{2\tau}-1)_{ij} W(e^{2\tau K_B}-1){ij}] \\
    &\quad + e^{-2\tau} \EE[ W(e^{2\tau}-1)_{ij} W(e^{2\tau(K_B+1)}-1)_{ij}]\bigg) \\
    &= e^{-\tau K_B}\sigma_{ij}^2 \bigg( (1-e^{-\tau})^2 - e^{-\tau}\min\{e^{2\tau}-1, e^{2\tau K_B}-1\}+ e^{-2\tau}\min\{e^{2\tau}-1, e^{2\tau(K_B+1)}-1\}\bigg) \intertext{If $K_B=0$, it reduces to the variance given in Lemma~\ref{lem: delta_variance}, which equals to $2(1 - e^{-\tau}) \sigma_{ij}^2$. Let $K_B \geq 1$, then}
    &= e^{-\tau K_B}\sigma_{ij}^2 ((1-e^{-\tau})^2-e^{-\tau}(e^{2\tau}-1)+e^{-2\tau}(e^{2\tau}-1) ) \\
    &=- (1-e^{-\tau})^2 e^{-\tau(K_B-1)}\sigma_{ij}^2, \label{equ: delta_KB}
\end{align}
where, following our previous notation, $\sigma_{ij}^2 = 1 + \One\{i = j\} $.
Since for all $(i,j) \in [n]^2$ we have $(1+ \mathbbm{1}\{i\neq j\})  \sigma_{ij}^2 = 2$, the result follows.
\end{proof}

Now we look at how these extra factors will interact with a sum over $K$, which is what we will finally be interested in per Proposition~\ref{prop:inner_prod_K_sum}.
We will decompose a sum of the following form into its positive ($K_B=0$) and negative ($K_B \geq 1$) components (as we know from the last two above results): using \eqref{equ: delta_KB}, we have
\begin{align}\label{equ: delta_deri_pn}
    \sT(t) &\colonequals \sum_{K \in \ZZ^{\sB}_{\geq 0}} \PP(K(t) = K) \sum_{B \in \sB} \sum_{(i,j)\in B} (1+ \mathbbm{1}\{i\neq j\}) \EE [ \Delta_B G_{ij} \Delta_B G(K)_{ij}]     \EE [  \partial_{ij} \lambda  \partial_{ij} \lambda^K  ] \\
    &=  4(1 - e^{-\tau})\sum_{B \in \sB} \sum_{(i,j)\in B} \sum_{\substack{K \in \ZZ^{\sB}_{\ge 0},\\ K_B = 0}} \PP(K(t) = K)  \EE [  \partial_{ij} \lambda  \partial_{ij} \lambda^K  ] \\
    &\quad - 2(1-e^{-\tau})^2\sum_{B \in \sB} \sum_{(i,j)\in B} \sum_{\substack{K \in \ZZ^{\sB}_{\ge 0},\\ K_B \geq 1}} \PP(K(t) = K)  e^{-\tau(K_B-1)}  \EE [  \partial_{ij} \lambda  \partial_{ij} \lambda^K  ] \\
    &\equalscolon \sT_{+}(t) - \sT_{-}(t).  \label{equ: def_T_+}
\end{align}
We then state the following lemma, which establishes conditions under which the positive term dominates the negative contribution. For this purpose, for each $B \in \sB$ and $K \in \ZZ_{\geq 0}^{\sB}$, we define the notation $K^{\sim B} \in \ZZ_{\geq 0}^{\sB}$ by \begin{align}
    K^{\sim B}_C \colonequals \begin{cases}
        0, &\text{if } C = B; \\
        K_C, &\text{otherwise.}
    \end{cases}
\end{align}
In words, this is the vector $K \in \ZZ^{\sB}_{\geq 0}$ with the coordinate indexed by $B$ set to zero.
In particular, we always have
\[ K = K^{\sim B} + K_Be_B. \]

\begin{prop}\label{prop: dominate_t}
If \begin{align}\label{equ: range_t}
     0 \leq t \leq \frac{e^{\tau}}{\eta} \log( \frac{1}{1-e^{-\tau}}),
\end{align} 
then
    \begin{align}
          \frac{1}{2} \sT_{+}(t)  \geq \sT_{-}(t),
    \end{align}
    and therefore
    \[ \sT(t) = \sT_{+}(t) - \sT_{-}(t) \geq \frac{1}{2}\sT_{+}(t). \]
\end{prop}

\begin{remark}
    We now see in more detail the reason that an upper bound on $t$ is required in Theorem~\ref{thm: main_POU}, as we discussed briefly after the Theorem statement.
    The technical reason for this is that we need $t$ to be not too large for the above kind of result to hold, since when $t$ is large, then $K(t)$ is typically large, and therefore the (negative) contribution of $\sT_{-}(t)$ increasingly dominates the value of $\sT(t)$, and in particular $\sT(t) < 0$.
    Below, our strategy will be to use Lemma~\ref{lem: delta_partial_o1} to compare $\sT(t)$ to an evaluation of the Dirichlet form of the PDBOU, which is always non-negative.
    Thus, when $t$ is too large, the result of the Lemma becomes vacuous, just bounding both the Dirichlet form and $|\sT(t)|$.
    The reason for this is that, in the Lemma, we are analyzing terms of the form $\EE [ \Delta_{B} \lambda_\alpha  \Delta_{B} \lambda_\alpha^K]$ (and comparing to continuous versions).
    Our proof technique essentially estimates each of the two factors in this expectation separately.
    However, there is another important behavior involved, which is that as the entries of $K$ grow, the above expectation becomes smaller, since $G(0)$ and $G(K)$ become increasingly decorrelated.
    Our proof, giving an error bound uniform in $K$, does not take this into account, and thus our result becomes less precise (relative to the scale of this expectation) as $K$ gets larger, or, when averaged over the random $K(t)$, as $t$ gets larger.
\end{remark}

\begin{proof}
Starting from \eqref{equ: def_T_+}, 
\begin{align}
   & \frac{1}{2} \sT_+(t) - \sT_{-}(t)\\
    &= \sum_{B \in \sB} \sum_{(i,j)\in B}  \sum_{K^{\sim B} \in \ZZ^{\sB}_{\ge 0}} \PP(K^{\sim B} (t) = K^{\sim B}  ) \bigg(  2(1-e^{-\tau}) \PP(K_B(t) = 0)  \EE [  \partial_{ij} \lambda  \partial_{ij} \lambda^{K^{\sim B}}  ] \\
    &\quad - \sum_{K_B \geq 1} 2(1-e^{-\tau})^2 e^{-\tau(K_B-1)} \PP(K_B(t) = K_B)   \EE [  \partial_{ij} \lambda  \partial_{ij} \lambda^{K}  ] \bigg) 
    \intertext{By Lemma~\ref{lem: pos_E_derivative}, we have $ \EE [  \partial_{ij} \lambda  \partial_{ij} \lambda^{K^{\sim B}}  ] \geq \EE [  \partial_{ij} \lambda  \partial_{ij} \lambda^{K}  ] \geq 0$ for all $K \in \ZZ^{\sB}_{\ge 0}$ and $B \in \sB$, then }
    &\geq \sum_{B \in \sB} \sum_{(i,j)\in B} \sum_{K^{\sim B} \in \ZZ^{\sB}_{\ge 0}} \PP(K^{\sim B} (t) = K^{\sim B}  ) \EE [  \partial_{ij} \lambda  \partial_{ij} \lambda^{K^{\sim B}}  ]  (2-2e^{-\tau}) \\
    &\quad \bigg(  \PP(K_B(t) = 0) -  \sum_{K_B \geq 1} (1-e^{-\tau}) e^{-\tau(K_B-1)} \PP(K_B(t) = K_B) \bigg)
\end{align} 
So, it suffices to show that the last expression in parentheses is non-negative in every summand.
From \eqref{equ: P_Kt_K} (in the proof of the variance identity of the PDBOU) that
\begin{align}\label{equ: P_K_t}
    \PP(K(t) = K)  = \prod_{C \in \mathcal{B}} e^{-\eta  t}  \frac{{(\eta t)}^{K_{C}}}{K_{C}!}.
\end{align}
Then we can further calculate the probability term above, that \begin{align}
    & \PP(K_B(t) = 0) -  \sum_{K_B \geq 1} (1-e^{-\tau}) e^{-\tau(K_B-1)} \PP(K_B(t) = K_B) \\
    &=  e^{-\eta t } -  \sum_{K_B \geq 1} (1-e^{-\tau}) e^{-\tau(K_B-1)} e^{-\eta t} \frac{(\eta t)^{K_B}}{K_B!}  \intertext{Since $\sum_{K_B \geq 1} (\eta t e^{-\tau})^{K_B} / K_B!  = \exp(\eta t e^{-\tau} )-1 $, we have}
    &=  e^{-\eta t }  - (1-e^{-\tau}) e^{\tau} e^{- \eta t} (\exp(\eta t e^{-\tau} )-1) \\
    &= e^{-\eta t} (1 - (e^{\tau}-1)(\exp(\eta t e^{-\tau} )-1) ).
\end{align}
Thus, by solving \begin{align}
    (e^{\tau}-1)(\exp(\eta t e^{-\tau} )-1) \leq 1,
\end{align}
which gives \begin{align}
    0 \leq t \leq \frac{e^{\tau}}{\eta} \log\left( \frac{1}{1-e^{-\tau}}\right),
\end{align}
completing the proof.
\end{proof}

We are now ready for the proof of the main result.
\begin{proof}[Proof of Theorem~\ref{thm: main_POU}] 
As before, we fix $\alpha \in [n]$, and for simplicity, we write $(\lambda, v)$ for $(\lambda_\alpha, v_\alpha)$. 
Fix $t \in \RR_+$, satisfying the assumption \eqref{equ: teta} of the Theorem.
We restate this below:
\begin{align}\label{equ: range_t_pf}
     0 \leq t \leq \frac{e^{\tau}}{\eta} \log( \frac{1}{1-e^{-\tau}}).
\end{align} 

We start with the expansion from Proposition~\ref{prop:inner_prod_K_sum}, and control this expression by a summation over only $K_B \geq 1$, preparing to relate it to $\sT_+(t)$, and then to $\sT(t)$ using the bounds proved above.
\begin{align}
     &\EE [ \langle v(\widetilde{G}(0)), v(\widetilde{G}(t))  \rangle^2 ] \\
    &= \sum_{K \in \ZZ^{\sB}_{\ge 0}} \PP(K(t) = K)  \sum_{B \in \sB} \sum_{(i,j)\in B}     \EE [  \partial_{ij} \lambda  \partial_{ij} \lambda^K  ] \\
    &= \sum_{B \in \sB} \sum_{(i,j)\in B} \sum_{\substack{K \in \ZZ^{\sB}_{\ge 0},\\ K_B = 0}} \PP(K(t) = K)  \EE [  \partial_{ij} \lambda  \partial_{ij} \lambda^K  ] \\
    &\quad + \sum_{B \in \sB} \sum_{(i,j)\in B} \sum_{\substack{K \in \ZZ^{\sB}_{\ge 0},\\ K_B \geq 1}} \PP(K(t) = K)    \EE [  \partial_{ij} \lambda  \partial_{ij} \lambda^K  ]. \label{equ: v_vgt_1}
    \end{align}
    By using the monotonicity given in Lemma~\ref{lem: pos_E_derivative}, for each fixed $K^{\sim B}$, we have \begin{align}
        \sum_{K_B \geq 1} \PP(K_B(t) = K_B) \EE [  \partial_{ij} \lambda  \partial_{ij} \lambda^K  ] &\leq \EE [  \partial_{ij} \lambda  \partial_{ij} \lambda^{K^{\sim B}} ]  \sum_{K_B \geq 1} \PP(K_B(t) = K_B) \\
        &\leq \EE [  \partial_{ij} \lambda  \partial_{ij} \lambda^{K^{\sim B}} ] \PP(K_B(t) \geq 1). 
    \end{align}
    So, we have
    \begin{align*}
        &\sum_{B \in \sB} \sum_{(i,j)\in B} \sum_{\substack{K \in \ZZ^{\sB}_{\ge 0},\\ K_B \geq 1}} \PP(K(t) = K)    \EE [  \partial_{ij} \lambda  \partial_{ij} \lambda^K  ] \\
        &= \sum_{B \in \sB} \sum_{(i,j)\in B} \sum_{\substack{K^{\sim B} \in \ZZ^{\sB}_{\ge 0}, \\ K^{\sim B}_B = 0}} \sum_{K_B \geq 1} \PP(K(t) = K^{\sim B} + K_B e_B) \EE [  \partial_{ij} \lambda  \partial_{ij} \lambda^K  ] \\
        &= \sum_{B \in \sB} \sum_{(i,j)\in B} \sum_{\substack{K^{\sim B} \in \ZZ^{\sB}_{\ge 0}, \\ K^{\sim B}_B = 0}} \PP(K^{\sim B}(t) = K^{\sim B}) \sum_{K_B \geq 1} \PP(K_B(t) = K_B) \EE [  \partial_{ij} \lambda  \partial_{ij} \lambda^K  ] \\
        &\leq \sum_{B \in \sB} \sum_{(i,j)\in B} \sum_{\substack{K^{\sim B} \in \ZZ^{\sB}_{\ge 0}, \\ K^{\sim B}_B = 0}} \PP(K^{\sim B}(t) = K^{\sim B}) \PP(K_B(t) \geq 1) \EE [  \partial_{ij} \lambda  \partial_{ij} \lambda^{K^{\sim B}} ] \\
        &= \sum_{B \in \sB} \sum_{(i,j)\in B} \sum_{\substack{K^{\sim B} \in \ZZ^{\sB}_{\ge 0}, \\ K^{\sim B}_B = 0}} \PP(K(t) = K^{\sim B}) \frac{\PP(K_B(t) \geq 1)}{\PP(K_B(t) = 0)} \EE [  \partial_{ij} \lambda  \partial_{ij} \lambda^{K^{\sim B}} ].
    \end{align*}
    Substituting this into \eqref{equ: v_vgt_1}, we have
    \begin{align}
    &\EE [ \langle v(\widetilde{G}(0)), v(\widetilde{G}(t))  \rangle^2 ] \\
    &\leq \sum_{B \in \sB} \sum_{(i,j)\in B} \bigg(1+ \frac{ \PP(K_B(t) \geq 1) }{ \PP(K_B(t) = 0)} \bigg) \sum_{\substack{K^{\sim B} \in \ZZ^{\sB}_{\ge 0},\\ K_B^{\sim B} = 0}} \PP(K(t) = K^{\sim B}) \EE [  \partial_{ij} \lambda  \partial_{ij} \lambda^K  ] 
    \intertext{By \eqref{equ: P_K_t}, we have $1+ { \PP(K_B(t) \geq 1) }/{ \PP(K_B(t) = 0)} = e^{\eta t} $. Then, we can write the above equation in term of $\sT_+(t)$ in \eqref{equ: def_T_+},}
    &\leq \frac{ e^{\eta t} }{4(1-e^{-\tau})} \sT_+(t) 
    \intertext{Using our assumption that $t$ is in the range given in \eqref{equ: range_t_pf}, Proposition~\ref{prop: dominate_t} implies $\sT(t)  \geq \frac{1}{2} \sT_+(t) $, so we obtain }
    &\leq \frac{ e^{\eta t} }{2(1-e^{-\tau})} \sT(t)
    \intertext{Then with Lemma~\ref{lem: delta_partial_o1}, we have}
    &\leq \frac{ e^{\eta t} }{2(1-e^{-\tau})} \left( \sum_{K \in \ZZ^{\sB}_{\ge 0}} \PP(K(t) = K)  \sum_{B \in \sB} \EE [ \Delta_{B} \lambda_\alpha  \Delta_{B} \lambda_\alpha^K ] + \sR\right)
    \intertext{By Lemma~\ref{thm: var_identity_poi_OU}, this is equivalent to the {Dirichlet form},}
    &\leq \frac{ e^{\eta t} }{2(1-e^{-\tau})}  \bigg( \frac{2}{\eta} D(\lambda, P_t\lambda) + \sup_{K \in \ZZ^{\sB}_{\geq 0}} \sR(K) \bigg), \label{equ: v_tilde_v_b}
\end{align}
where the Lemma gives a uniform bound on $\sR(K)$.
So, let us write
\begin{equation}
    \sR \colonequals \sup_{K \in \ZZ^{\sB}_{\geq 0}} \sR(K) \leq C(1 - e^{-\tau})(\nu^{2 + \epsilon} n^{-1/2 + \delta + \epsilon} + \nu^{1 + \epsilon} n^{-\delta/4 + \epsilon}). \label{eq:R-estimate}
\end{equation} 

On the other hand, by the variance identity given in Lemma~\ref{lem: cov_identity} and using the nonnegativity and
monotonicity of $D(\lambda, P_t\lambda)$ stated in Lemma~\ref{lem: dirichlet_mon}, we have
\begin{align}\label{equ: var_pou_c}
D(\lambda, P_t\lambda)
&\leq \frac{1}{t}\int_{0}^{t} D(\lambda, P_s\lambda) ds \\
&\leq \frac{1}{t}\int_{0}^{\infty} D(\lambda, P_s\lambda) ds \\
&\leq \frac{1}{t}\Var(\lambda(G)) \intertext{and using the bound of Corollary~\ref{cor: eigenvalue_bound} on $\Var(\lambda(G))$, we have}
&\leq \frac{F(n,\alpha)\,\hat a^{-2/3} n^{-1/3}}{t},
\end{align}
where $F(n,\alpha)$ is as defined in \eqref{eq:F}.
Therefore, applying the above inequality to \eqref{equ: v_tilde_v_b}, we conclude \begin{align}
    \EE [ \langle v(\widetilde{G}(0)), v(\widetilde{G}(t))  \rangle^2 ] &\leq \frac{ e^{\eta t} }{1-e^{-\tau}} \bigg( \frac{F(n,\alpha)\,\hat a^{-2/3} n^{-1/3}}{ \eta t } +  \sR \bigg)\intertext{By our assumption \eqref{equ: range_t_pf} on $t$, we can further bound}
    &\leq (1-e^{-\tau})^{-(1+e^{\tau})}\bigg(\frac{F(n,\alpha) \hat a^{-2/3} n^{-1/3}}{ \eta t }+\sR \bigg)
    \intertext{Further, one may verify that for all $\tau > 0$ we have the bound $(1 - e^{-\tau})^{-(1 + e^{\tau})} \leq 10 (1 \vee \tau^{-2}) = 10 / (1 \wedge \tau)^2$, a form one may guess by noting that the singularity of this function near $\tau = 0$ is of the kind $\tau^{-2}$, while the function converges to a constant as $\tau \to \infty$. This gives, hiding the irrelevant constant:}
    &\lesssim \frac{F(n,\alpha) \hat a^{-2/3} n^{-1/3}}{(1 \wedge \tau)^2 \eta t }+ \frac{\sR}{(1 \wedge \tau)^2}
    \intertext{Here, we use \eqref{eq:R-estimate} together with the fact that $1 - e^{-\tau} \leq 1 \wedge \tau$, which gives}
    &\lesssim \frac{F(n,\alpha) \hat a^{-2/3} n^{-1/3}}{(1 \wedge \tau)^2 \eta t }+ \frac{\nu^{2 + \epsilon}n^{-1/2 + \delta + \epsilon} + \nu^{1 + \epsilon} n^{-\delta / 4 + \epsilon}}{1 \wedge \tau}.
\end{align}

Consider the conditions under which the second term above will go to zero.
Given that $\epsilon > 0$ here may be taken arbitrarily small, it suffices for there to exist some $\epsilon^{\prime}$ such that
\[ \nu = O\left(\frac{n^{1/4-\delta/2 - \epsilon^{\prime}}}{\sqrt{1 \wedge \tau}}\right) \text{ and } \nu = O\left(\frac{n^{\delta/4 - \epsilon^{\prime}}}{1 \wedge \tau}\right). \]
Further, to apply Lemma~\ref{lem: delta_partial_o1}, we need to have
\[ \nu = o\left(\frac{n^{5/6 - \delta}}{\log n} \hat{\alpha}^{-1/3}\right). \]
Thus, under these restrictions, if we also have
\begin{align}
   \frac{(1 \wedge \tau)^2 \eta t}{ F(n,\alpha) \hat a^{-2/3} n^{-1/3} } \to \infty,
\end{align}
then $\EE [ \langle v(\widetilde{G}(0)), v(\widetilde{G}(t))  \rangle^2 ]  = o(1) $, as claimed.
\end{proof}

\subsection{Resampling dynamics: Proof of Theorem~\ref{thm: resample_block_1}}

The proof of Theorem~\ref{thm: resample_block_1} first follows the same strategy as Theorem~\ref{thm: main_POU}, starting with the Taylor approximation in Lemma~\ref{lem: delta_partial_o1}. We state Lemma~\ref{lem: delta_partial_o1_ind}, the analog for this setting, below. Moreover, since decorrelation is controlled by the number of resampled blocks $k$ in the independent resampling dynamics, accordingly, we start from the variance identity in Theorem~\ref{thm: var_identity} and invoke the monotonicity of $T_k$ from Corollary~\ref{cor: T_k_mono}, instead of the monotonicity in $t$ of the Dirichlet form.

\begin{lemma}\label{lem: delta_partial_o1_ind}
  Let $\sB$ be an admissible partition of $[n] \times [n]$. For any $k \in [|\sB^{(n)}|]$ and $A \in \sA_k$, there exists $\delta > 0$ (any $\delta \in (0, \delta_0)$ for the $\delta_0$ in Theorem~\ref{thm: sum_of_inverse_spacing} can be used here) such that for each $\alpha \in [n]$, if \begin{align}\label{equ:nuB_1}
     \nu \leq \frac{n^{5/6 - \delta}}{\log n} \hat{\alpha}^{-1/3},
\end{align}
there exists a constant $C>0$ such that for arbitrarily small $\epsilon>0$,
    \begin{align}\label{equ: delta_partial_tau_ind}
      &\sum_{\substack{B \in \mathcal{B},\\  B \not\subseteq A} } \Bigl|  \EE \left[ \Delta_{B} \lambda_\alpha \Delta_{B} \lambda_\alpha^A \right] -  \sum_{(i,j) \in B } \tilde{\sigma}_{ij}^2  \EE \left[   \partial_{ij} \lambda_\alpha(X)  \partial_{ij}  \lambda_\alpha(X^A) \right] \Bigr| \\
      &\leq C  \left(   \nu^{2+\epsilon}   n^{-1/2 + \delta+\epsilon} + \nu^{1+\epsilon}  n^{-\delta/4+\epsilon}   \right) \colonequals \sR, \label{equ: ind_R}
    \end{align}
    where \begin{align}
    \Delta_{B} \lambda_\alpha &= \lambda_\alpha(X) - \lambda_\alpha(X^B) \\
    \Delta_{B} \lambda_\alpha^A &= \lambda_\alpha(X^A) - \lambda_\alpha(X^{A \cup B}) \\ 
        \tilde{\sigma}_{ij}^2 &\coloneqq \EE[(X_{ij}-X^B_{ij})^2]
+ \mathbbm{1}\{i\neq j\}\EE[(X_{ji}-X^B_{ji})^2] \\
&= \begin{cases}
            4 \sigma_{ij}^2 , &\text{if } (i,j) \in B \text{ and } i \neq j, \\
            2\sigma_{ij}^2, &\text{if } (i,j) \in B \text{ and } i = j.
        \end{cases}  
    \end{align} 
\end{lemma}

We omit the full proof of Lemma~\ref{lem: delta_partial_o1_ind} as it is entirely analogous to that of Lemma~\ref{lem: delta_partial_o1}.
Below, we just describe the modifications required to repeat that proof in the independent resampling setting.

First, the construction of the auxiliary matrices $Z$ and $Z(K)$ (denoted here as $Z^{A}$) in Proposition~\ref{prop: POU_inequality} simplifies significantly. We define $Z$ as an independent copy of $X$ that differs from $X$ only on the entries indexed by $B$, and similarly define $Z^{A}$ as an independent copy of $X^{A}$ restricted to those same entries. With this construction, the variables $X_{ij}$, $X^B_{ij}$, $Z_{ij}$, and $Z^A_{ij}$ are mutually independent for all $(i,j)\in B$.

Second, the summation in \eqref{equ: delta_partial_tau_ind} is restricted to $B \not\subseteq A$ because $\Delta_{B} \lambda_\alpha^A$ vanishes by Definition~\ref{def: block_resampling}; this restriction also appears in the proof of Theorem~\ref{thm: resample_block_1} due to same restriction in the variance identity for the PDBR process. Furthermore,  $ (1+ \mathbbm{1}\{i\neq j\}) \EE[\Delta_B G_{ij} \Delta_B G(K)_{ij}  ]$ in \eqref{equ: sum_expectation} simplifies to $\tilde{\sigma}_{ij}^2$ defined above. Unlike the PDBOU process, here $\tilde{\sigma}_{ij}^2$ does not depend on $k$ or $A$ and is just a statistic of the law of the generalized Wigner matrix $X$, allowing it to factor out of the expectation. Consequently, the bounds of Proposition~\ref{prop: dominate_t} are unnecessary, as the variance is simply absorbed into the constant $C$ in \eqref{equ: ind_R}.

\begin{proof}[Proof of Theorem~\ref{thm: resample_block_1}]  
As before, we fix $\alpha \in [n]$ and omit the dependence of all quantities involved on $\alpha$ as its value does not affect the proof.
 Let $k \in \NN$.
With Lemma~\ref{lem: smooth_lambda}, we can write the variance identity given in Theorem~\ref{thm: var_identity} with $f = \lambda$ that \begin{align}\label{equ: proof_var_identity}
    \Var (\lambda(X) )
= \frac{1}{2m} \sum_{k=0}^{m-1} \frac{1}{\binom{m-1}{k}} \sum_{B \in  \mathcal{B}}   \sum_{A \in \mathcal{A}_{k, B}}   \EE\left[ \Delta_{B} \lambda  \Delta_{B} \lambda^A \right].
\end{align}
We then compare $ \Var (\lambda(X) )$ to $\EE \left[ \langle v(X)  ,  v(X^A) \rangle^2 \right]$:
\begin{align}\label{equ: v_va_identity_2}
     &\EE \left[ \langle v(X)  ,  v(X^A) \rangle^2 \right] \\
     &\quad =  \frac{1}{\binom{m}{k}} \sum_{A \in \mathcal{A}_{k}}  \EE \left[  2 \sum_{1 \leq i< j \leq n} \partial_{ij} \lambda  \partial_{ij} \lambda^A + \sum_{i = 1}^n \partial_{ii} \lambda  \partial_{ii} \lambda^A \right] \\
     &\quad = \frac{1}{\binom{m}{k}}   \sum_{A \in \mathcal{A}_{k}}  \sum_{B \in \mathcal{B}}   \sum_{(i,j) \in B} \EE[ \partial_{ij} \lambda  \partial_{ij} \lambda^A ] \intertext{We can rewrite the sum to mirror \eqref{equ: proof_var_identity}, as}
     &\quad = \frac{1}{\binom{m}{k}}   \sum_{B \in  \mathcal{B}} \sum_{(i,j) \in B} \left( \sum_{A \in \mathcal{A}_{k,B}} \EE[ \partial_{ij} \lambda  \partial_{ij} \lambda^A ] + \sum_{A \in \mathcal{A}_{k-1,B}} \EE[ \partial_{ij} \lambda  \partial_{ij} \lambda^{A \cup B} ] \right), \label{equ: exp_v_k-1}
\end{align}
where we emphasize the expectation in \eqref{equ: exp_v_k-1} is now taken with only respect to $X, X^A$. 
\noindent If  $ (i,j) \in A$, we denote $ A' = A \setminus \{ (i,j) \}$. By Jensen's inequality, we have
\begin{align}\label{equ: resample_A_B}
    \EE[\partial_{ij}\lambda \partial_{ij}\lambda^{A}]
&= \EE[\EE[\partial_{ij}\lambda \mid \{X_{ab}\}_{(a,b)  \notin A}]^{2}] \\
&\leq 
\EE[\EE[\partial_{ij}\lambda \mid \{X_{ab}\}_{(a,b)  \notin A'}]^{2}]
= \EE[\partial_{ij}\lambda \partial_{ij}\lambda^{A'}].
\end{align}
Applying the above inequality \eqref{equ: resample_A_B} to \eqref{equ: v_va_identity_2}, we get
\begin{align}\label{equ: v_va_upper_bound}
   &\EE \left[ \langle v(X)  ,  v(X^A) \rangle^2 \right]  \\
   &\quad \leq \frac{1}{\binom{m}{k}}   \sum_{B \in  \mathcal{B}} \sum_{(i,j) \in B } \left( \sum_{A \in \mathcal{A}_{k,B}} \EE[ \partial_{ij} \lambda  \partial_{ij} \lambda^A ] + \sum_{A \in \mathcal{A}_{k-1,B}} \EE[ \partial_{ij} \lambda  \partial_{ij} \lambda^{A} ] \right) \\
   &\quad = \frac{1}{\binom{m}{k}}   \sum_{B \in  \mathcal{B}}  \left( \sum_{A \in \mathcal{A}_{k,B}} (\sum_{(i,j) \in B } \EE[ \partial_{ij} \lambda  \partial_{ij} \lambda^A ] ) + \sum_{A \in \mathcal{A}_{k-1,B}} (\sum_{(i,j) \in B } \EE[ \partial_{ij} \lambda  \partial_{ij} \lambda^{A} ] )\right). \label{equ: inner_v_ij_B}
\end{align}
Suppose $\nu$ satisfies the bound given in \eqref{equ:nuB_1}. Then with the nonnegativity of  $\EE \left[    \partial_{ij} \lambda  \partial_{ij}  \lambda^A \right]$  above in \eqref{equ: resample_A_B}, Lemma~\ref{lem: delta_partial_o1_ind} gives that
 \begin{align}
  \sum_{B \in \sB}  \sum_{(i,j) \in B } \EE \left[    \partial_{ij} \lambda  \partial_{ij}  \lambda^A \right]  &\leq \frac{1}{2c_1} \sum_{B \in \sB}  \sum_{(i,j) \in B } \tilde{\sigma}_{ij}^2 \EE \left[    \partial_{ij} \lambda  \partial_{ij}  \lambda^A \right] \\
    &\leq \frac{1}{2c_1} \left( \sum_{B \in \sB} \EE \left[ \Delta_{B} \lambda  \Delta_{B} \lambda^A \right] + \sR \right), \label{equ: delta_partial_bound_00}
\end{align}
where we recall that $\sR$ denotes the error given in \eqref{equ: ind_R}, which we will expand later.
 Then, combining \eqref{equ: inner_v_ij_B} and \eqref{equ: delta_partial_bound_00}, we have
\begin{align}
    &\EE \left[ \langle v(X)  ,  v(X^A) \rangle^2 \right]  \\
     &\leq \frac{1}{2c_1} \frac{1}{\binom{m}{k}}   \sum_{B \in  \mathcal{B}}  \left( \sum_{A \in \mathcal{A}_{k,B}}  \EE \left[ \Delta_{B} \lambda  \Delta_{B} \lambda^A \right] + \sum_{A \in \mathcal{A}_{k-1,B}}  \EE \left[ \Delta_{B} \lambda  \Delta_{B} \lambda^A \right]\right) +  \frac{1}{2c_1}  \sR \intertext{We write in terms of $T_k$ defined in \eqref{equ: def_Tk} in Theorem~\ref{thm: var_identity}, then}
     &= \frac{1}{2c_1} \frac{1}{\binom{m}{k}}   \left( \tbinom{m-1}{k} T_k +  \tbinom{m-1}{k-1} T_{k-1}  \right) +   \frac{1}{2c_1}   \sR  \intertext{We further apply Corollary~\ref{cor: T_k_mono} to write it in terms of $\Var(\lambda(X)) $,}
     &\leq \frac{1}{2c_1} \left( \frac{m-k}{m}   \frac{2m}{k+1} + \frac{k}{m}  \frac{2m}{k} \right)  \Var(\lambda(X)) + \frac{1}{2c_1}   \sR \\
     &= \frac{1}{c_1}  \frac{m+1}{k+1}  \Var(\lambda(X)) + \frac{1}{2c_1}   \sR. \label{equ: inner_product_main_me}
\end{align}
Sufficient conditions to have $\sR = o(1)$ are the same as we derived in proving Theorem~\ref{thm: main_POU}, but without the dependence on the $\tau$ parameter there; in particular, for sufficiently small $\delta > 0$ (according to whether we are working with a Wigner or generalized Wigner matrix) it suffices to have $\nu = O(n^{\delta})$.
Therefore, under this restriction on $\nu$, we have
\begin{align}
   \EE \left[ \langle v(X)  ,  v(X^A) \rangle^2 \right]  &\leq   \frac{1}{c_1}  \frac{m+1}{k+1}  \Var(\lambda(X)) + o(1) \intertext{Applying the upper bound on $\Var ( \lambda(X) )$ from Corollary~\ref{cor: eigenvalue_bound}, we have}
   &\leq   \frac{1}{c_1}  \frac{m+1}{k+1}  F(n, \alpha) \hat{\alpha}^{-2/3} n^{-1/3} + o(1),
\end{align}
where $F(n, \alpha)$ is as defined in \eqref{eq:F}. Therefore,
 if \begin{align}   \frac{k}{m  F(n, \alpha) \hat{\alpha}^{-2/3}n^{-1/3}}  \to \infty,\end{align} then $\EE \left[ \langle v(X)  ,  v(X^A) \rangle^2 \right]   = o(1) $.
\end{proof}

\section*{Acknowledgments}
\addcontentsline{toc}{section}{Acknowledgments}
Thanks to Benjamin McKenna for helpful discussions during this project.

\addcontentsline{toc}{section}{References}
\bibliographystyle{alpha}
\bibliography{ref}

\newpage
\appendix  

\section{Simplicity of eigenvalues along lines: Proof of Proposition~\ref{prop: F_B}}\label{sec: simplicity_eigenvalues}

Recall that we want to study the set
\[ E_{\sym} \colonequals \{X \in \RR^{n \times n}_{\sym}: X \text{ has a repeated eigenvalue}\}. \]
We do this by considering a suitable complex generalization.

Recall that, to any matrix $X \in \CC^{n \times n}$, we may associate the \emph{minimal polynomial} $\mu_X(z)$, the unique monic polynomial of minimum degree such that $\mu_X(X) = 0$.
By the Cayley-Hamilton Theorem, we always have $\deg(\mu_X) \leq n$.
Consider the set
\[ E \colonequals \{X \in \CC^{n \times n}: \deg(\mu_X) \leq n - 1 \}, \]
sometimes called the set of \emph{derogatory} matrices.
Since $\mu_X$ can have at most $\deg(\mu_X)$ zeroes, we have
\[ E \cap \RR^{n \times n}_{\sym} = \{X \in \RR^{n \times n}_{\sym}: X \text{ has a repeated eigenvalue}\} = E_{\sym}. \]

Before continuing, we also mention the following necessary condition for a matrix being derogatory that will be useful below.
\begin{prop}
    \label{prop:derogatory}
    If $X \in \CC^{n \times n}$ is derogatory, i.e.\ $X \in E$, then there exists $\lambda \in \CC$ (an eigenvalue of $X$) such that $\rank(\lambda I - X) \leq n - 2$.
    In particular, the determinant of every $(n - 1) \times (n - 1)$ submatrix of $\lambda I - X$ is zero.
\end{prop}

Note that, over $\CC^{n \times n}$, $E$ is not the same as the set of matrices with repeated (complex) eigenvalues.
Indeed, $E$ is the set of matrices having an eigenvalue of \emph{geometric} multiplicity greater than 1, while matrices with repeated eigenvalues have an eigenvalue of \emph{algebraic} multiplicity greater than 1.
It turns out that $E$ is easier to describe: given $X \in \CC^{n \times n}$, let $A(X) \in \CC^{n^2 \times n}$ have $\vec(I), \vec(A), \dots, \vec(A^{n - 1})$ as its columns.
By definition then, $\deg(\mu_X) \leq n - 1$ if and only if $A(X)$ is rank-deficient, i.e., has rank at most $n - 1$.
Writing $\Sigma \colonequals \binom{[n]^2}{n}$ and, for $\sigma \in \Sigma$ a set of $n$ positions of matrix entries, $A_{\sigma}(X)$ for the corresponding subset of rows of $A$, we have
\begin{align*}
    E 
    &= \{X \in \CC^{n \times n}: \mathrm{rank}(A(X)) \leq n - 1\} \\
    &= \{X \in \CC^{n \times n}: \det(A_{\sigma}(X)) = 0 \text{ for all } \sigma \in \Sigma\}.
\end{align*}
We note that similar definitions are discussed in \cite{parlett2002matrix}, but instead with the aim of understanding characterizations of the set of matrices having repeated eigenvalues (in the sense of algebraic multiplicity), which may be defined by a more complicated relation among the determinants $\det(A_{\sigma}(X))$.
From the above, we note in particular that $E \subset \CC^{n \times n}$ is a complex algebraic variety.

Recall that Proposition~\ref{prop: F_B} concerns the set
\begin{align*}
F_B 
&\colonequals \{(X, \Delta) \in \RR^{n \times n}_{\sym} \times \RR^{B}_{\sym}: \text{there exists } s \in \RR \text{ such that } X + s\Delta \in E_{\sym}\}
\intertext{Using the above observations, we may rewrite}
&= \{(X, \Delta) \in \RR^{n \times n}_{\sym} \times \RR^{B}_{\sym}: \text{there exists } s \in \RR \text{ such that } X + s\Delta \in E\} \\
&= \{(X, \Delta) \in \RR^{n \times n}_{\sym} \times \RR^{B}_{\sym}: \text{there exists } s \in \RR \text{ such that } \det(A_{\sigma}(X + s\Delta)) = 0 \text{ for all } \sigma \in \Sigma\} \\
&\subseteq \{(X, \Delta) \in \RR^{n \times n}_{\sym} \times \RR^{B}_{\sym}: \text{there exists } s \in \CC \text{ such that } \det(A_{\sigma}(X + s\Delta)) = 0 \text{ for all } \sigma \in \Sigma\} \\
&\equalscolon \widetilde{F}_B.
\end{align*}

Note that $\widetilde{F}_B$ is the set of $(X, \Delta)$ such that a certain family of univariate polynomials $p_\sigma(s)$ have a simultaneous root $s \in \CC$, where the coefficients of the $p_{\sigma}$ depend on $X$ and $\Delta$.
This property can be expressed as a system of polynomial equations in $(X, \Delta)$ themselves using the resultant of multiple polynomials; see Section 2 of Chapter 3 of \cite{cox1998using}.
In particular, $\widetilde{F}_B$ is a real algebraic variety.

\begin{proof}[Proof of Proposition~\ref{prop: F_B}] 
By the above reasoning, the result will follow if we can show that $\widetilde{F}_B \neq \RR^{n \times n}_{\sym} \times \RR^B_{\sym}$.
Equivalently, for any $B$, we want to find $X \in \RR^{n \times n}_{\sym}$ and $\Delta \in \RR^B_{\sym}$ such that, for all $s \in \CC$, $X + s\Delta$ is not derogatory.
We will use Proposition~\ref{prop:derogatory} for this purpose.

Since we can take $\Delta$ arbitrary supported on the entries in $B$, without loss of generality we may suppose that $B$ consists of just one entry, up to symmetry.
We decompose our task into several cases, as below.

\emph{Case 1: $B = \{(i, i)\}$.}
Further without loss of generality we may assume $i = 1$.
Take $\Delta = e_1e_1^{\top}$ and
\[ X = \left[\begin{array}{ccccc} 0 & 1 & 0 & \cdots & 0 \\
1 & 0 & 1 & \cdots & 0 \\
0 & 1 & 0 & \cdots & \vdots \\
\vdots & \vdots & \vdots & \ddots & 1 \\
0 & 0 & \cdots & 1 & 0 \end{array}\right]. \]
Then, we have
\[ X + s\Delta = \left[\begin{array}{ccccc} s & 1 & 0 & \cdots & 0 \\
1 & 0 & 1 & \cdots & 0 \\
0 & 1 & 0 & \cdots & \vdots \\
\vdots & \vdots & \vdots & \ddots & 1 \\
0 & 0 & \cdots & 1 & 0 \end{array}\right]. \]
Let $\lambda \in \CC$.
Consider the determinant of the submatrix of $X + s\Delta - \lambda I$ formed by deleting the first row and the last column, which is
\[ \det\left(\left[\begin{array}{ccccc} 1 & -\lambda & 1 & \cdots & 0 \\
0 & 1 & -\lambda & \cdots & 0 \\
0 & 0 & 1 & \cdots & \vdots \\
\vdots & \vdots & \vdots & \ddots & -\lambda \\
0 & 0 & \cdots & 0 & 1 \end{array}\right]\right) = 1, \]
since this matrix is upper triangular.
Thus, by Proposition~\ref{prop:derogatory} we have that $X + s\Delta$ is never derogatory.

\emph{Case 2: $B = \{(i, j), (j, i)\}$ for $i \neq j$.}
Again without loss of generality we may assume $i = 1$ and $j = 2$.
We take $\Delta = e_1e_2^{\top} + e_2e_1^{\top}$, analogous to the idea from Case 1.
However, for a subtle reason we will point out when it comes up, we need to take $X$ slightly different.
Namely, here we take:
\begin{equation}
X = \left[\begin{array}{ccccc} 0 & 1 & 0 & \cdots & 0 \\
1 & 3 & 1 & \cdots & 0 \\
0 & 1 & 3 & \cdots & \vdots \\
\vdots & \vdots & \vdots & \ddots & 1 \\
0 & 0 & \cdots & 1 & 3 \end{array}\right]. \label{eq:derogatory-X}
\end{equation}
(The specific choice is not important provided that the diagonal entries are big enough.)
We then have
\[ X + s\Delta - \lambda I = \left[\begin{array}{ccccc} - \lambda & 1 + s & 0 & \cdots & 0 \\
1 + s & 3 - \lambda & 1 & \cdots & 0 \\
0 & 1 & 3 - \lambda & \cdots & \vdots \\
\vdots & \vdots & \vdots & \ddots & 1 \\
0 & 0 & \cdots & 1 & 3 - \lambda \end{array}\right]. \]
We want to show that, for all $s, \lambda \in \CC$, this matrix has rank at least $n - 1$, so that by Proposition~\ref{prop:derogatory} it again is not derogatory.
We consider a few further cases:

\emph{Case 2.1: $s \neq -1$.}
In this case, the submatrix formed by deleting the first row and last column is upper triangular with non-zero entries on the diagonal, so it has non-zero determinant and the result follows.

\emph{Case 2.2: $s = -1$.}
We will break this up into two more cases, but we first note that in this case we have
\[ X + s\Delta - \lambda I = \left[\begin{array}{ccccc} - \lambda & 0 & 0 & \cdots & 0 \\
0 & 3 - \lambda & 1 & \cdots & 0 \\
0 & 1 & 3 - \lambda & \cdots & \vdots \\
\vdots & \vdots & \vdots & \ddots & 1 \\
0 & 0 & \cdots & 1 & 3 - \lambda \end{array}\right] = [-\lambda] \oplus \left[\begin{array}{cccc}
3 - \lambda & 1 & \cdots & 0 \\
1 & 3 - \lambda & \cdots & \vdots \\
\vdots & \vdots & \ddots & 1 \\
0 & \cdots & 1 & 3 - \lambda \end{array}\right], \]
a direct sum decomposition that will be useful below.

\emph{Case 2.2.1: $s = -1$ and $\lambda \neq 0$.}
In this case, the first direct summand above has rank 1.
Similar to what we argued before, the second direct summand, whose dimensions are $(n - 1) \times (n - 1)$, has a non-singular submatrix formed by deleting the first row and last column, since that leaves an upper triangular matrix with 1's on the diagonal.
Thus, the second direct summand has rank at least $n - 2$, and thus in total $\rank(X + s\Delta - \lambda I) \geq n - 1$, as needed.

\emph{Case 2.2.2: $s = -1$ and $\lambda = 0$.}
In this case, the first direct summand above is zero, while the second is the $(n - 1) \times (n - 1)$ matrix
\[ \left[\begin{array}{cccc}
3 & 1 & \cdots & 0 \\
1 & 3 & \cdots & \vdots \\
\vdots & \vdots & \ddots & 1 \\
0 & \cdots & 1 & 3 \end{array}\right]. \]
This matrix is diagonally dominant, and thus non-singular, so again $\rank(X + s\Delta - \lambda I) \geq n - 1$, completing the proof.
\end{proof}

Note that if we had not introduced the 3's above in \eqref{eq:derogatory-X}, then in Case 2.2.2 we would need to show that the adjacency matrix of the path graph on $n - 1$ vertices has full rank.
But, this graph has a zero eigenvalue whenever $n - 1$ is odd, so our argument would fail.

\section{Markov semigroups, Dirichlet forms, and variance identities}
\label{sec: variance_identity}

\subsection{Energy dissipation in Markov processes: Proof of Lemma~\ref{lem: dirichlet_mon}}

\begin{proof}[Proof of Lemma~\ref{lem: dirichlet_mon}]
    Since $(P_t)_{t \geq 0}$ is reversible and the generator $\sL$ commutes with the semigroup operators $P_t$, for all $t \geq 0$ we have \begin{align}
        D(f, P_tf) = -\langle f, \sL P_t f\rangle 
        &= -\langle f, \sL P_{t/2} P_{t/2} f\rangle \\ &= -\langle f, P_{t/2} \sL P_{t/2} f\rangle \\
        &= -\langle P_{t/2} f, \sL P_{t/2} f\rangle  \\
        &= D(P_{t/2} f, P_{t/2} f)\\
        &\ge 0
    \end{align}
    by the positivity of the Dirichlet form.
    To prove the monotonicity, from above and by the chain rule, we have \begin{align}
        \frac{d}{dt} D(f, P_tf) =  \frac{d}{dt} D(P_{t/2} f, P_{t/2} f) &= \frac{1}{2} \frac{d}{ds} D(P_{s} f, P_{s} f) \bigg|_{s = t/2} \\
        &= \frac{1}{2}\left(-\langle \sL P_sf, \sL P_sf \rangle -\langle P_sf, \sL^2 P_sf \rangle\right) \\
        &= - \|\sL P_{t/2}f \|^2 \\
        &\leq 0,
    \end{align}
    as claimed.
\end{proof}

\subsection{OU process: Proof of Lemma~\ref{lem: var_identity_OU}}

We begin the proof of Lemma~\ref{lem: var_identity_OU} by first recalling the properties of the one-dimensional OU process on $\RR$.
Recall that this is the process $X(t) \in \RR$ given by, for $X = X(0) \sim \sN(0, 1)$ and $W(t)$ a standard Brownian motion, $X(t) = e^{-t} X + e^{-t}W(e^{2t} - 1)$.

\begin{lemma}[\cite{vanHandel-PHD}, Lemma 2.22]
    \label{lem: OU_process_property}
    The process $X(t)$ above is ergodic and has $\mu = \sN(0, 1)$ as its stationary measure.
    For $X \sim \sN(0, 1)$, its semigroup, generator, and Dirichlet form are given by
    \begin{align*}
        P_tf(y) &= \Ex_{X \sim \sN(0, 1)} f\left(e^{-t}y + \sqrt{1 - e^{-2t}}X\right), \\
        \sL f(y) &= -yf^{\prime}(y) + f^{\prime\prime}(y), \\
        D(f, g) &= \Ex_{X \sim \sN(0, 1)} f^{\prime}(X)g^{\prime}(X),
    \end{align*}
    where the domain of the generator $\Dom(\sL)$ is the space of functions $f \in L^2(\mu)$ whose first two weak derivatives are also in $L^2(\mu)$ (i.e., the Sobolev space $W^{2,2}(\mu)$), while the Dirichlet form can be taken over the larger space of $f \in L^2(\mu)$ whose first weak derivative is also in $L^2(\mu)$ (i.e., the Sobolev space $W^{1,2}(\mu)$).
\end{lemma}

\begin{proof}[Proof of Lemma~\ref{lem: var_identity_OU}]
    Recall that the Lemma gives a calculation of the Dirichlet form $D(f, P_tf)$ as well as a variance identity for $\Var(f(G))$.

    For the former, we first consider the semigroup.
    Let $G^{\prime} \sim \GOE(n)$.
    Then, we have
    \begin{align*}
        P_tf(G) 
        &= \EE[f(G(t)) \mid G(0) = G] \\
        &= \EE[f(e^{-\tau t}G + \sqrt{1 - e^{-2\tau t}}G^{\prime}], \\
        \nabla P_tf(G) 
        &= \EE[e^{-\tau t} \cdot \nabla f(e^{-\tau t}G + \sqrt{1 - e^{-2\tau t}}G^{\prime}] \\
        &= e^{-\tau t} \cdot \EE[\nabla f(G(t)) \mid G(0) = G] \\
        &= e^{-\tau t} \cdot P_t \nabla f(G)
    \end{align*}

    On the other hand, since (up to scaling appropriately) the matrix OU process just has independent scalar OU processes in its coordinates, from Lemma~\ref{lem: OU_process_property} we find
    \begin{align*}
    D(f, P_t f) 
    &= \Ex_{G \sim \GOE(n)} \langle \nabla f(G), \nabla P_t f(G) \rangle \\
    &= e^{-\tau t} \Ex_{G \sim \GOE(n)} \langle \nabla f(G), P_t \nabla f(G) \rangle \\
    &= e^{-\tau t} \Ex_{G(t) \sim \OU(n, \tau)} \langle \nabla f(G), \nabla f(G(t)) \rangle,
    \end{align*}
    where the last step follows by conditioning on $G(0)$ inside the expectation.
    
    The stated form of the variance identity then follows from plugging this into Lemma~\ref{lem: cov_identity}.
    That $D(f, P_t f)$ is non-negative and non-increasing is proved in general in Lemma~\ref{lem: dirichlet_mon}.
\end{proof}

\subsection{PDBOU process: Proof of Lemma~\ref{thm: var_identity_poi_OU}}

Before proof the Lemma~\ref{thm: var_identity_poi_OU}, we first prove the semigroup properties of the Poisson-driven block OU process. 

\begin{lemma}\label{lem: prop_POU_semigroup}
    For any $n \geq 1$, any covering $\sB$ of $[n] \times [n]$, and any $\eta, \tau > 0$, the Poisson-driven block OU process $\PDBOU(n, \sB, \eta, \tau)$ is a Markov process with reversible ergodic Markov semigroup having stationary measure $\mu = \GOE(n)$.
\end{lemma}

\begin{proof}[Proof of Lemma~\ref{lem: prop_POU_semigroup}]
   The stationarity of $\mu$ follows directly from the stationarity of $\sN(0, 1)$ for the one-dimensional OU process in Lemma~\ref{lem: OU_process_property}.

    Next, we calculate the semigroup:
    \begin{align}\label{equ: com_semigroup}
        (P_t f)(G) &=\EE[ f(\widetilde{G}(t)) \mid \widetilde{G}(0)=G ] \\ 
         &= \sum_{K \in \ZZ_{\geq 0}^{\sB}} \PP(K(t)=K)\Ex_{G(t) \sim \OU(n, \tau)} \left[ f(G(K)) \mid G(0) = G \right]
         \intertext{where we recall the notation $G(K)$ for an outcome $K \in \ZZ^{\sB}_{\geq 0}$ of the underlying Poisson process from \eqref{equ: G_ij_K}. Writing this differently, for each $1 \leq i \leq j \leq n$, let $Q^{(i, j)}_t$ be the Markov semigroup of a one-dimensional OU process with rate $\tau$ acting on coordinate $(i, j)$ of a matrix (and its symmetric counterpart), and suitably rescaled for the diagonal case $i = j$. Then, we may write this as}
         &= \sum_{K \in \ZZ_{\geq 0}^{\sB}} \PP(K(t)=K) \left(\left(\prod_{1 \leq i \leq j \leq n} Q^{(i, j)}_{\bar{K}_{ij}}\right) f\right)(G)
    \end{align}
    Note that these entrywise transition kernels are mutually commutative operators and are all self-adjoint in $L^2(\mu)$, since they are reversible and $\mu$ is the product measure of their respective stationary measures.

    Reversibility of $P_t$ then follows by expanding by linearity and using this individual self-adjointness.
    Also, by the covering property every entry $(i, j)$ belongs to at least one block $C \in \sB$, so $\bar{K}_{ij}(t) \to \infty$ for all $(i, j) \in [n] \times [n]$ almost surely.
    Then, for ergodicity we may bound
    \begin{align*}
        \|P_t f - \mu(f)\|_{L^2(\mu)}
        &\leq \sum_{K \in \ZZ_{\geq 0}^{\sB}} \PP(K(t)=K) \left\|\left(\prod_{1 \leq i \leq j \leq n} Q^{(i, j)}_{\bar{K}_{ij}}\right) f - \mu(f)\right\|_{L^2(\mu)} \\
        &= \sum_{K \in \ZZ_{\geq 0}^{\sB}} \PP(K(t)=K) \left\|\left(\prod_{1 \leq i \leq j \leq n} Q^{(i, j)}_{\bar{K}_{ij}}\right) f - \left(\prod_{1 \leq i \leq j \leq n} Q^{(i, j)}_{\infty}\right) f\right\|_{L^2(\mu)},
    \end{align*}
    then bounding again by triangle inequality in a telescoping sum where each $\bar{K}_{ij}$ is replaced by $\infty$, and using ergodicity of the entrywise one-dimensional OU processes.
\end{proof}

\begin{proof}[Proof of Lemma~\ref{thm: var_identity_poi_OU}]
Calculating the semigroup in more detail from \eqref{equ: com_semigroup} above, we have
\begin{align}\label{equ: com_semigroup_2}
        (P_t f)(G) &=\EE[ f(\widetilde{G}(t)) \mid \widetilde{G}(0)=G ] \\ 
         &= \sum_{K \in \ZZ_{\geq 0}^{\sB}} \PP(K(t)=K)\Ex_{G(t) \sim \OU(n, \tau)} \left[ f(G(K)) \mid G(0) = G \right] \\
        &=\sum_{K \in \ZZ_{\geq 0}^{\sB}} \left( \prod_{C \in \mathcal{B}} e^{-\eta  t}  \frac{{(\eta t)}^{K_{C}}}{K_{C}!}   \right) \Ex_{G(t) \sim \OU(n, \tau)} \left[ f(G(K)) \mid G(0) = G \right].  \label{equ: P_Kt_K}
    \end{align}

To calculate the generator $ \sL f(G) = \partial_t P_t f(G)|_{t=0}$, we first note that, given $K$ fixed, letting $m \colonequals |\sB|$, and recalling our notation $|K| \colonequals \sum_{B \in \mathcal{B}} K_{B}$, we have
\begin{align}
     \partial_t \left( \prod_{C \in \mathcal{B}} e^{-\eta  t}  \frac{{(\eta t)}^{K_{C}}}{K_{C}!} \right)  \bigg|_{t=0} = \begin{cases}
          - \eta  m &\text{if } |K| = 0, \\
          \eta &\text{if } |K| = 1, \\
          0 &\text{if } |K| \geq 2.
     \end{cases}
\end{align}
We denote $e_B$ as the vector such that $K_{B} = 1$ and $K_{C} = 0$ for all $C \neq  B$, the standard basis vector of index $B$ in $\ZZ_{\geq 0}^{\sB}$.
Then,
\begin{align}\label{equ: Lf_POU}
    \sL f(G) 
    &= -\eta  mf(G) + \eta \sum_{B \in \mathcal{B} } \Ex_{G(t) \sim \OU(n, \tau)} \left[ f(G(e_B)) \mid G(0) = G \right] \\
    &= \eta \sum_{B \in \mathcal{B} } \left( \Ex_{G(t) \sim \OU(n, \tau)} \left[ f(G(e_B)) \mid G(0) = G \right]  - f(G) \right).
\end{align}

Towards calculating the Dirichlet form, note that by the tower property and the reversibility of the process, for each $B \in \sB$,
\begin{align}
    \Ex_{G \sim \mu}\left[f(G)  \cdot \Ex_{G(t) \sim \OU(n, \tau)} [  g(G(e_B) )\mid G(0) = G] \right] 
    &= \Ex_{G(t) \sim \OU(n, \tau)} \left[f(G(0))  g(G(e_B)) \right] \\
    &= \Ex_{G(t) \sim \OU(n, \tau)} \left[f(G(e_B) )  g(G(0)) \right],
\end{align} 
and by stationarity, and omitting the subscripts of expectations when only the one process $G(t) \sim \OU(n, \tau)$ is involved,
\begin{align}
\mathbb{E}\left[f(G(0))g(G(0))\right] = \mathbb{E}\left[f(G(e_B))g(G(e_B))\right].
\end{align}
Combining the above two observations gives
\begin{align}\label{equ: f_g_K_sym}
&\mathbb{E}\left[f(G(0))g(G(0))\right]
- \mathbb{E}\left[f(G(e_B))g(G(0))\right]
\\
&\hspace{2cm} = \frac{1}{2}\mathbb{E}\!\left[
(f(G(0))-f(G(e_B)))
(g(G(0))-g(G(e_B)))
\right].
\end{align}

With \eqref{equ: Lf_POU} and \eqref{equ: f_g_K_sym}, we can then calculate the Dirichlet form for general inputs,
\begin{align}
     D(f,g) &\coloneqq -\langle {f}, {\sL g} \rangle_\mu \\
     &= - \eta  \Ex_{G \sim \mu} \left[ f(G) \sum_{B \in \mathcal{B} } \left(  \Ex_{G(t) \sim \OU(n, \tau)}  [ g(G(e_B) ) \mid G(0) = G ] - g(G) \right) \right] \\
     &= \eta  \sum_{B \in \mathcal{B} }  \EE[ f(G(0))  g(G(0))] - \EE[ f(G(e_B) )  g(G(0))] \\
     &= \frac{\eta}{2} \sum_{B \in \mathcal{B} } \mathbb{E} \left[ (f(G(0)) - f(G(e_B) )) (g(G(0)) - g(G(e_B) )) \right]
     \intertext{which in our notation from the statement of the Lemma is}
     &= \frac{\eta}{2}\sum_{B \in \sB} \EE[\Delta_B f \Delta_B g].
\end{align}

We next calculate the specific Dirichlet form $D(f, P_tf)$ as in the statement of the Lemma.
We first note that given a fixed block $B \in \sB$ and conditional on the draw of $G(t) \sim \OU(n, \tau)$, by \eqref{equ: com_semigroup_2},
\begin{align}
     &(P_t f)(G(0))-(P_t f)(G(e_B)) \\
     &= \sum_{K \in \ZZ_{\geq 0}^{\sB}} \left( \prod_{C \in \mathcal{B}} e^{-\eta  t}  \frac{{(\eta t)}^{K_{C}}}{K_{C}!} \right) \left(\EE\left[ f(H(K)) \mid H(0) = G(e_B)\right] - \EE\left[ f(H(K)) \mid H(0) = G(0)\right]\right),
\end{align}
where $H(t) \sim \OU(n, \tau)$.
Then, by the tower and semigroup properties, we find
\begin{align}
    &D(f,P_tf) \\
    &= \frac{\eta}{2} \sum_{B \in \mathcal{B} } \mathbb{E} \left[ (f(G) - f(G(e_B) )) ((P_tf)(G) - (P_tf)(G(e_B) )) \right] \\
    &= \frac{\eta}{2} \sum_{B \in \mathcal{B} } \sum_{K \in \ZZ_{\geq 0}^{\sB}} \left( \prod_{C \in \mathcal{B}} e^{-\eta  t}  \frac{{(\eta t)}^{K_{C}}}{K_{C}!} \right) \EE \left[ (f(G) - f(G(e_B) )) \left(f(G(K)) -f(G(K+e_B)) \right) \right] \\
    &= \frac{\eta}{2} \sum_{B \in \mathcal{B} } \sum_{K \in \ZZ_{\geq 0}^{\sB}} \left( \prod_{C \in \mathcal{B}} e^{-\eta  t}  \frac{{(\eta t)}^{K_{C}}}{K_{C}!} \right) \EE \left[ \Delta_{B} f \Delta_{B} f^K \right],
\end{align}
using our notation from the statement of the Lemma.
This expression gives the first result of the Lemma, and the non-negativity and monotonicity of this expression follow from the general Lemma~\ref{lem: dirichlet_mon}.

For the variance identity, we substitute the calculation above into Lemma~\ref{lem: cov_identity}, which gives\begin{align}\label{equ: var_f_G_1}
    \Var (f(G) ) &= \int_0^{\infty} D(f, P_t f) dt \\
    &= \frac{\eta}{2} \sum_{B \in \mathcal{B} } \sum_{K \in \ZZ_{\geq 0}^{\sB}} \left( \int_0^{\infty}  \prod_{C \in \mathcal{B}} e^{-\eta  t}  \frac{{(\eta t)}^{K_{C}}}{K_{C}!}    dt  \right) \EE \left[ \Delta_{B} f \Delta_{B} f^K \right],
\end{align}
where the inner integral is a standard Gamma function integral, giving
    \begin{align}
  \int_0^{\infty} \prod_{C \in \mathcal{B}} e^{-\eta  t}  \frac{{(\eta t)}^{K_{C}}}{K_{C}!}   dt 
  &=  \int_0^{\infty}  e^{-\eta  t  m }   {(\eta  t)}^{|K|}  \frac{1}{\prod_{C \in \mathcal{B}} K_{C}!}   dt \\
  &=\eta^{|K|}  \frac{|K|!}{ {(\eta  m)}^{|K|+1}}   \frac{1}{\prod_{C \in \mathcal{B}} K_{C}!} \\
  &= \frac{1}{\eta}  \frac{|K|!}{ {m}^{|K|+1}}   \frac{1}{\prod_{C \in \mathcal{B}} K_{C}!}. \label{equ: int_prob_K_t}
\end{align}
Substituting \eqref{equ: int_prob_K_t} into \eqref{equ: var_f_G_1} then gives the last result:
\begin{align}
\Var(f(G))
&= \frac{1}{2}  \sum_{B \in \mathcal{B}}  \sum_{K\in\ZZ_{\geq 0}^{\sB}} \frac{|K|!}{m^{|K|+1}}  \frac{1}{\prod_{C \in \mathcal{B}} K_{C}!}  \EE \left[ \Delta_{B} f \Delta_{B} f^K \right] \\
&= \frac{1}{2}  \sum_{N=0}^{\infty}\frac{N!}{m^{N+1}}   \sum_{B \in \mathcal{B}}  \sum_{\substack{K\in\ZZ_{\geq 0}^{\sB}\\ |K|=N}} \frac{1}{\prod_{C \in \mathcal{B}} K_{C}!}  \EE \left[ \Delta_{B} f \Delta_{B} f^K \right]. \qedhere
\end{align}
\end{proof}

\subsection{PDBR process: Proof of Lemma~\ref{thm: var_identity}}

Similarly, we first prove the semigroup properties of the independent sampling process.
We recall some of the notation: below $\mu$ will be the law of a sub-Gaussian generalized Wigner matrix.
The PDBR is defined in terms of independent draws $X = X^{(0)}, X^{(1)}, \dots \sim \mu$.
For $K \in \ZZ^{\sB}_{\geq 0}$, we write $X(K)$ for the matrix with entries
\[ X(K)_{ij} = X^{(\bar{K}_{ij})}_{ij}, \]
so that the PDBR may be defined as $X(t) = X(K(t))$.

When we compare $X$ to $X(K)$, some entries have been resampled several times and some only once, but repeated resamplings will not affect many components of our calculations.
So, let us define $\One(K) \in \{0, 1\}^{\sB}$ to have entries
\[ \One(K)_B = \One\{K_B \geq 1\}. \]

\begin{lemma}\label{lem: prop_indep_semigroup}
    Let $\mu$ be the law of a sub-Gaussian generalized Wigner matrix and $\sB$ a covering of $[n] \times [n]$.
    Then, the process $\PDBR(\sB, \mu)$ generates a reversible ergodic Markov semigroup with stationary measure $\mu$.
\end{lemma}

\begin{proof}[Proof of Lemma~\ref{lem: prop_indep_semigroup}]
    That $\mu$ is a stationary measure is immediate from the definition and that $\sB$ is a covering.
    
    We note that, whenever $K, K^{\prime} \in \mathbb{Z}_{\geq 0}^{\sB}$ have the same support, i.e.\ $\One(K) = \One(K^{\prime})$ in the above notation, then the pairs $(X, X(K))$ and $(X, X(K^{\prime}))$ are identically distributed.
    In particular, this is true of $K^{\prime} = \One(K)$.
    We can therefore simplify the transition kernel:
    \begin{align}
        (P_tf)(X)
        &= \EE[f(X(t)) \mid X(0) = X] \\
        &= \sum_{K \in \ZZ_{\geq 0}^{\sB}} \PP(K(t) = K) \EE[f(X(K)) \mid X(0) = X] \\
        &= \sum_{K \in \ZZ_{\geq 0}^{\sB}} \PP(K(t) = K) \EE[f(X(\One(K))) \mid X(0) = X] \\
        &= \sum_{K \in \{0, 1\}^{\sB}} \PP(\One(K(t)) = K) \EE[f(X(K)) \mid X(0) = X]. \label{equ: PDBR-kernel}
    \end{align}

    Then, to establish reversibility, we calculate
    \begin{align}
        \langle f,P_t g\rangle_\mu
&= \EE\left[f(X)(P_t g)(X)\right] \\
&= \sum_{ K \in \{0,1\}^{\sB} }  \mathbb{P}\left( \mathbbm{1}(K(t)) = K \right)\EE \left[f(X(0))g(X(K)) \right],  \label{equ: Ptg}
    \end{align}
    and reversibility follows since $(X(0), X(K))$ are an exchangeable pair.

    For ergodicity, note that $\One(K(t))$ converges almost surely to the all-ones vector, call it $\bm 1 \in \ZZ_{\geq 0}^{\sB}$.
    For this vector, since $X(\bm 1)$ is independent of $X(0)$, we have $\EE[f(X(\bm 1)) \mid X(0) = X] = \EE[f(X(\bm 1))] = \mu(f)$, and ergodicity follows from the above expression for $P_t f$.
\end{proof}

\begin{proof}[Proof of Lemma~\ref{thm: var_identity}]
    First, we expand the formula for the transition kernel from \eqref{equ: PDBR-kernel} fully by evaluating the probabilities, using that the underlying Poisson clocks determining $K$ are independent.
    Recall that we write $m \colonequals |\sB|$.
    We have:
 \begin{align}\label{equ: p_t_f_dij_c}
     (P_t f)(X) &= \sum_{K \in \{0, 1\}^{\sB}} \PP(\One(K(t)) = K) \EE[f(X(K)) \mid X(0) = X]  \\
 &= \sum_{K \in \{0, 1\}^{\sB}}  \left( (1 - e^{-t})^{|K|}  e^{-t(m -  |K|)} \right)    \EE \left[f(X(K)) \mid X(0) = X \right] \\
 &= \sum_{k=0}^{m} \left( (1 - e^{-t})^{k}  e^{-t(m - k)} \right)  \sum_{A \in A_{k}} n_k(A) \EE f(X^A),
\end{align}
where \begin{align}
 n_k(A) &\coloneqq \#\left\{ K \in \{0,1\}^{\sB} : |K| = k,  \bigcup_{B: K_B = 1} B = A \right\}
\end{align}
and where we switch to the notation that $X^A$ is a copy of the matrix $X$ where entries in indices belonging to $A$ have been resampled according to their marginal distributions under $\mu$.
We further denote $w_t(k) \coloneqq  (1 - e^{-t})^{k}  e^{-t(m - k)}$, then, for the generator $ \sL f(X)$, we have
\begin{align}
    \sL f(X) = \partial_t  P_t f(X)|_{t=0} =  \sum_{k=0}^{m} \left( \partial_t  w_t(k)|_{t=0} \right) \sum_{A \in A_{k}} n_k(A)  \EE f(X^A). \end{align}
    Since
    \begin{align}
\partial_t  w_t(k)  |_{t=0} =  \begin{cases}
        -m &\text{if } k = 0, \\
        1 &\text{if } k = 1, \\
        0 &\text{if } k \geq 2,
    \end{cases}
\end{align}
we can simplify $ \sL f(X)$ as
\begin{align}
     \sL f(X) = -m  f(X) + \sum_{A \in \mathcal{A}_1} \EE f(X^A) = \sum_{B \in \mathcal{B}} \left( \EE f(X^B) - f(X)\right).
\end{align} 
As before, since $(X, X^{B})$ is an exchangeable pair, so
\[ \EE \left[f(X)(g(X)-g(X^{B}))\right]
= \EE \left[f (X^{B}) (g(X^{B})-g(X))\right] \]
and therefore
\begin{align}\label{equ: resample_sym}
    \EE \left[f(X)(g(X)-g(X^{B}))\right]
= \frac{1}{2}  \EE \left[(f(X)-f(X^{B}))(g(X)-g(X^{B}))\right]. 
\end{align}

In general, by the tower property and \eqref{equ: resample_sym}, the Dirichlet form is given by
\begin{align}
    D(f,g)\coloneqq -\langle {f}, {\sL g} \rangle_\mu &= \sum_{B \in \mathcal{B}}  \EE \left[f(X)(g(X)-g(X^{B}))\right] \\
    &= \frac{1}{2}    \sum_{B \in \mathcal{B}} 
 \EE\left[(f(X)-f(X^{B}))(g(X)-g(X^{B}))\right]. \label{equ: dirichlet_form}
\end{align}
Therefore, to calculate $D(f,P_t f)$, from \eqref{equ: dirichlet_form}, \begin{align}\label{equ: dirichlet_form_2}
    D(f,P_t f) = \frac{1}{2}    \sum_{B \in \mathcal{B}} 
 \EE\left[(f(X)-f(X^{B}))(P_t f(X)-P_t f(X^{B}))\right].
\end{align}
To calculate this, we recall from above that we may write
\begin{align}\label{equ: Ptf_ptfij}
 (P_t f)(X)-(P_t f)(Y) = \sum_{k=0}^{m-1}  w_t(k) \sum_{A \in \mathcal{A}_{k}} n_k(A)
    \left( \EE f(X^A)
         - \EE f(Y^A) \right).
\end{align}
When we evaluate this with $Y = X^B$, then in terms where $B \subseteq A$, we will have that $X^A$ and $Y^A$ are identically distributed (since those entries on which $X$ and $Y$ might disagree, which are only those in $B$, are resampled again when all of $A$ is resampled) and thus such terms will be zero. 
In other words, only those terms with $A \in \sA_{k, B}$ will contribute.

Using this observation, we can rewrite \eqref{equ: dirichlet_form_2} as:
\begin{align}
     D(f,P_t f) &= \frac{1}{2} \sum_{k=0}^{m-1} w_t(k) \sum_{B \in  \mathcal{B}}  \sum_{A \in \mathcal{A}_{k, B}} n_k(A) 
     \EE\left[(f(X)-f(X^{B}))(f(X^A)-f(X^{A\cup B}))\right] \\
     &= \frac{1}{2} \sum_{k=0}^{m-1} w_t(k) \sum_{B \in  \mathcal{B}}  \sum_{A \in \mathcal{A}_{k, B}} n_k(A)  \EE\left[ \Delta_{B} f \Delta_{B} f^A \right].
\end{align}
With Lemma~\ref{lem: prop_indep_semigroup}, we can apply the general covariance identity from Lemma~\ref{lem: cov_identity} to this process, and we find by evaluating the integral involved that
\begin{align}
      \Var (f(X) ) &= \int_0^{\infty} D(f, P_t f)\,dt \\
       &= \frac{1}{2} \sum_{k=0}^{m-1} \left(  \int_0^\infty  w_t(k)  dt \right) \sum_{B \in  \mathcal{B}}  \sum_{A \in \mathcal{A}_{k, B}} n_k(A)  \EE\left[ \Delta_{B} f \Delta_{B} f^A \right] \\
     &= \frac{1}{2m} \sum_{k=0}^{m-1} \frac{1}{\binom{m-1}{k}}  \sum_{B \in  \mathcal{B}}  \sum_{A \in \mathcal{A}_{k, B}} n_k(A)  \EE\left[ \Delta_{B} f \Delta_{B} f^A \right],
\end{align}
completing the proof.
\end{proof}

\section{Bounds on moments of maxima: Proof of Corollary~\ref{cor: moment_bounds}}\label{sec: concentration_ine}

\begin{proof}[Proof of Corollary~\ref{cor: moment_bounds}]
Recall that the statement concerns the moments of $\|X\|_{\ell^{\infty}}$ for $X$ a sub-Gaussian generalized Wigner matrix with parameters $(c_1, c_2, K)$.
We first prove the moment bounds in this general case, and then the results specialized to Gaussian matrices.
For the second moment,
\begin{align}
    \EE \| X\|_{\ell^{\infty}}^2 &= \int_0^\infty \PP( \| X\|_{\ell^{\infty}}^2 \geq u ) \,du \\
    &= \int_0^\infty \PP( \| X \|_{\ell^{\infty}} \geq t ) 2t \,dt \intertext{By Lemma~\ref{lem: infty_norm_inequality}, there exists a $c = c(c_1, c_2, K) >0$ such that we have the bound $ \PP( \|X\|_{\ell^{\infty}} \geq t ) \leq  \min\{1, 2 n^2 \exp( -ct^2 / K^2) \}$, and using this we find}
    &\leq \int_0^\infty \min\{1, 2 n^2 \exp( -c t^2 / K^2)\} 2t\, dt 
    \intertext{Let $t_0 \colonequals (K \sqrt{
    \log(2n^2)})/\sqrt{c}$, so that $2 n^2 \exp( -c t^2 / K^2) \geq 1$ when $t \leq t_0$, and split the integral,}
    &\leq \int_0^{t_0} 2t\, dt + \int_{t_0}^\infty 2 n^2 \exp( -c t^2 / K^2) 2t\, dt \\
    &= \frac{K^2}{c} \log(2 n^2) + \frac{2K^2}{c},
\end{align}
as claimed.

Similarly for the fourth moment,
\begin{align}
     \EE \| X\|_{\ell^{\infty}}^4 &=  \int_0^\infty \PP( \| X\|_{\ell^{\infty}} \geq t ) 4t^3\, dt \\
    &\leq \int_0^{t_0} 4t^3 dt + \int_{t_0}^\infty 2 n^2 \exp( -c t^2 / K^2)  4t^3\, dt  \\
    &=  t_0^4 + 8 n^2 \int_{t_0}^\infty t^3 \exp( -c t^2 / K^2)\,  dt. 
\end{align}
For the remaining tail integral, we substitute $u = c t^2 / K^2 $ with $u_0 = c t^2_0 / K^2  = \log(2n^2)$, then 
\begin{align}
    \int_{t_0}^\infty t^3 \exp( -c t^2 / K^2)\,  dt &= \frac{K^4}{2c^2} \int_{u_0}^\infty 2 u \exp(-u)\, du \\
    &= \frac{K^4}{2c^2} \frac{\log(2n) +1}{n^2}.
\end{align}
Together, \begin{align}
     \EE \| X\|_{\ell^{\infty}}^4 \leq \frac{K^4}{c^2} \log^2(2n^2) + \frac{2K^4}{c^2} \log(2n^2) + \frac{2K^4}{c^2},
\end{align}
as claimed.

Lastly, in the Gaussian case where $X_{ij} \sim N(0, \sigma_{ij}^2)$ with $\sigma_{ij}^2 \leq \sigma^2$, we may take $K^2 = 2\sigma^2$ and $c = 1$. The same calculation then gives
\begin{align}
       \EE \| X\|_{\ell^{\infty}}^2 &\leq 2 \sigma^2 \log(2n^2) + 4 \sigma^2, \\
       \EE \| X\|_{\ell^{\infty}}^4 &\leq 4 \sigma^4 \log^2 (2n^2) +  8 \sigma^4 \log (2n^2) +  8 \sigma^4.
   \end{align}
We obtain the simplified bounds stated in the Corollary by then noting that, for all $n\ge 1$, \begin{align}
    \log(2n^2)+2 &\leq (2+\log (2)(\log(n) +1 ) \leq 3 (\log(n) +1 ), \\
    \log^2(2n^2)+2\log(2n^2) + 2 &\leq (2+\log (2)^2(\log(n) +1 )^2 \leq 8 (\log(n) +1 )^2. \qedhere
\end{align}
\end{proof}

\section{Consequences of rigidity estimates}
\label{sec: eigidity_eigenvalues}

\subsection{Spacing estimate: Proof of Corollary~\ref{cor: lambda_i-lambda_j}}

\begin{proof}[Proof of Corollary~\ref{cor: lambda_i-lambda_j}]
Recall that this Corollary deduces from a rigidity estimate a lower bound on $|\lambda_{\alpha} - \lambda_{\beta}|$ for $|\alpha - \beta|$ sufficiently large.
To discuss the semicircle law in its standard normalization, we rescale our matrix by setting $\overline{X} \colonequals X / \sqrt{n}$ and denote its eigenvalues by $\Bar{\lambda}_1 \geq \dots \geq \Bar{\lambda}_n$, so that $\Bar{\lambda}_{\alpha} = \lambda_{\alpha} / \sqrt{n}$.
In this notation, we want to show that, with high probability,
\[ |\Bar{\lambda}_{\alpha} - \Bar{\lambda}_{\beta}| \geq c^{\prime}|\alpha - \beta|n^{-1} \text{ whenever } |\alpha - \beta| \geq C^{\prime}(\log n)^{L}. \]

 \cite[Theorem 2.2]{erdHos2012rigidity} shows that, for an event defined in terms of a parameter $L \geq 0$ as
 \begin{align}\label{equ: lambda_gamma_e}
 E \colonequals
  \{
       |\bar\lambda_\beta-\gamma_\beta|
       \le(\log n)^L\hat{\beta}^{-1/3}n^{-2/3}\text{ for all } \beta \in [n]
  \},
\end{align} there exist positive constants \( A_0 > 1 \), \( C \), \( c \), and \( \phi < 1 \) such that for $n$ sufficiently large depending only on these constants, for any \( L \) with
\begin{align}\label{equ: choice_A0}
A_0 \log \log n \leq L \leq \frac{\log(10n)}{10 \log \log n},
\end{align}
    we have $\mathbb P( E^c)\le
  C\exp  [-c(\log n)^{\phi L}]$. 
  For any $\beta \neq \alpha$, the triangle inequality gives \begin{align}\label{equ: lambda_bar_diff}
  |\bar\lambda_\alpha-\bar\lambda_\beta|
&\ge |\gamma_\alpha-\gamma_\beta|
 - |\bar\lambda_\alpha-\gamma_\alpha|
 - |\bar\lambda_\beta-\gamma_\beta|. 
    \end{align}
   If the classical spacing $|\gamma_{\alpha} - \gamma_{\beta}|$ dominates the rigidity errors  $  | \Bar{\lambda}_{\alpha} - \gamma_{\alpha} |$ and $| \Bar{\lambda}_{\beta} - \gamma_{\beta} |$, then the eigenvalue spacing is controlled by the spacing of the classical locations. 
 Hence the remaining proof reduces to identifying the range of $|\alpha-\beta|$ for which \begin{align}\label{equ: triangle_inequality_1}
  |\gamma_\alpha-\gamma_\beta| \stackrel{\text{(?)}}{\geq} 2 \left( |\bar\lambda_\alpha-\gamma_\alpha|+|\bar\lambda_\beta-\gamma_\beta| \right),
 \end{align}
 in which case \eqref{equ: lambda_bar_diff} implies \begin{align}\label{equ: triangle_inequality_2}
    |  \Bar{\lambda}_{\alpha} -  \Bar{\lambda}_{\beta}|
\geq \frac{1}{2} |\gamma_{\alpha} - \gamma_{\beta}|.
 \end{align}
Further, we have that, by definition, when $\alpha > \beta$ so that $\gamma_{\alpha} < \gamma_{\beta}$, we have
\begin{equation}
\frac{\alpha - \beta}{n} = \int_{\gamma_{\alpha}}^{\gamma_{\beta}} \varrho_{\scc}(x)\,dx \leq \frac{1}{\pi}(\gamma_{\beta} - \gamma_{\alpha}). \label{eq:gamma-int-1}
\end{equation}
Thus, the result will follow provided that we show that, on the event $E$, \eqref{equ: triangle_inequality_1} holds whenever $|\alpha - \beta| \geq C^{\prime}(\log n)^{L}$.

Towards establishing this, note that on the event $E$ we have, for the right-hand side of \eqref{equ: triangle_inequality_1},
\[ 2 \left( |\bar\lambda_\alpha-\gamma_\alpha|+|\bar\lambda_\beta-\gamma_\beta| \right) \leq (\hat{\alpha}^{-1/3} + \hat{\beta}^{-1/3}) \cdot 2(\log n)^L n^{-2/3}. \]
We first handle a few special cases.
First, suppose that $\alpha, \beta \in [\eta n, (1 - \eta)n]$ for some given $\eta \in (0, \frac{1}{4})$.
In this case, we have $\hat{\alpha}, \hat{\beta} \leq \eta n$, and so
\[ 2 \left( |\bar\lambda_\alpha-\gamma_\alpha|+|\bar\lambda_\beta-\gamma_\beta| \right) \leq 4\eta^{1/3} (\log n)^L n^{-1}, \]
and the claim follows from \eqref{eq:gamma-int-1}.
Second, suppose that $|\alpha - \beta| \geq 4(\log n)^L n^{1/3}$.
Then, \eqref{eq:gamma-int-1} gives $|\gamma_{\beta} - \gamma_{\alpha}| \geq 8\pi n^{-2/3}$, while we always have
\[ 2 \left( |\bar\lambda_\alpha-\gamma_\alpha|+|\bar\lambda_\beta-\gamma_\beta| \right) \leq 4 (\log n)^L n^{-2/3}, \]
and the result again holds.

Thus, we may assume that at least one of $\alpha, \beta$ is in either $[1, \eta n]$ or $[(1 - \eta)n, n]$, and that $|\alpha - \beta| \leq 4(\log n)^L n^{1/3} \leq \eta n$ (the last holding for $n$ sufficiently large).
Since $\eta < 1/4$, in particular $\alpha$ and $\beta$ are either both in $[1, n/2]$ or both in $[n/2, n]$.
By symmetry, without loss of generality we may suppose that $\alpha, \beta \in [1, n/2]$, and so $\hat{\alpha} = \alpha$ and $\hat{\beta} = \beta$.

We have $\varrho_{\text{sc}}(2-u) = \frac{1}{2 \pi} \sqrt{u}  \sqrt{4-u}$, and for $u \in [0,2]$, \begin{align}
   \frac{\sqrt{2}}{2\pi}\sqrt{u} \leq  \varrho_{\mathrm{sc}}(2-u) \le \frac{1}{\pi}\sqrt{u}.
\end{align}
Thus, from the definition of $\gamma_{\beta}$, we have
\begin{align}\label{equ: s_beta_bound}
    \frac{\beta}{n} &= \int_0^{2 - \gamma_{\beta}}  \varrho_{\text{sc}}(2-u)\, du,
\end{align}
which together with the above gives
\[ \frac{\sqrt{2}}{3\pi}  (2 - \gamma_{\beta})^{3/2} \leq \frac{\beta}{n} \leq \frac{2}{3\pi} (2 - \gamma_{\beta})^{3/2}. \]
Hence, by solving \eqref{equ: s_beta_bound}, there exist absolute constants $0 < c_1 \leq C_1 < \infty$ such that for all $1 \leq \beta \leq n/2$, \begin{align}\label{equ: rho_gammaj}
 c_1  \left( \frac{\beta}{n}\right)^{1/3}  \leq \varrho_{\text{sc}}(\gamma_\beta)  \leq C_1  \left( \frac{\beta}{n}\right)^{1/3}.
\end{align} 
Now, writing a more precise version of \eqref{eq:gamma-int-1}, since $\varrho_{\scc}(x)$ is increasing on $[0, 2]$, we have
\begin{equation}
\frac{|\alpha - \beta|}{n} = \int_{\gamma_{\alpha}}^{\gamma_{\beta}} \varrho_{\scc}(x)\,dx \leq C_1 \left(\frac{\max\{\alpha, \beta\}}{n}\right)^{1/3} |\gamma_{\alpha} - \gamma_{\beta}|,
\end{equation}
and rearranging gives
\[ |\gamma_{\alpha} - \gamma_{\beta}| \geq \frac{1}{C_1} |\alpha - \beta| \max\{\alpha, \beta\}^{-1/3} n^{-2/3}. \]
On the other hand, on the event $E$ and with the above simplifications, we have
\[ 2 \left( |\bar\lambda_\alpha-\gamma_\alpha|+|\bar\lambda_\beta-\gamma_\beta| \right) \leq 4(\log n)^L \max\{\alpha, \beta\}^{-1/3} n^{-2/3}. \]
Thus, provided that $|\alpha - \beta| \geq 4(\log n)^L$ we have
\[ |\gamma_{\alpha} - \gamma_{\beta}| \geq  2 \left( |\bar\lambda_\alpha-\gamma_\alpha|+|\bar\lambda_\beta-\gamma_\beta| \right) \]
as desired, completing the proof.
\end{proof}

\subsection{Eigenvalue variance estimate: Proof of Corollary~\ref{cor: eigenvalue_bound}}
\begin{proof}[Proof of Corollary~\ref{cor: eigenvalue_bound}]
Recall that this Corollary uses rigidity estimates to bound the variance of an eigenvalue.
We bound the variance by:
\begin{align}
    \Var({\lambda}_\alpha) &\leq \EE[  | {\lambda}_\alpha -\sqrt{n}\gamma_\alpha|^2  ] \\
    &= \int_0^\infty \PP\left( |{\lambda}_\alpha -\sqrt{n}\gamma_\alpha|^2 \geq t  \right)  dt \intertext{Based on the spacing for ${\lambda}_\alpha$ and $\gamma_\alpha$ given by Theorem~\ref{thm: rigidity_eigenvalue}, we split the integral into three parts. Denote $T_0 = \left((\log n)^L\hat{\alpha}^{-1/3}n^{-1/6}\right)^2$, then}
    &= \int_0^{T_0}  \PP\left( |{\lambda}_\alpha -\sqrt{n}\gamma_\alpha|^2 \geq t  \right)  dt +  \int_{T_0}^M  \PP\left( |{\lambda}_\alpha -\sqrt{n}\gamma_\alpha|^2 \geq t  \right)  dt  \\&\quad + \int_{M}^\infty  \PP\left( |{\lambda}_\alpha -\sqrt{n}\gamma_\alpha|^2 \geq t  \right)  dt,
\end{align} 
where $M$ is to be chosen later. For the first term, we use trivial bound that $\PP\left( |{\lambda}_\alpha -\sqrt{n}\gamma_\alpha|^2 \geq t  \right) \leq 1 $, then $\int_0^{T_0}  \PP\left( |{\lambda}_\alpha -\sqrt{n}\gamma_\alpha|^2 \geq t  \right)  dt \leq T_0$. For the second term, we apply the estimate given in \eqref{equ: rigidity_prob}, which gives
\begin{align}
    \int_{T_0}^M  \PP\left( |{\lambda}_\alpha -\sqrt{n}\gamma_\alpha|^2 \geq t  \right)  dt  \leq M \cdot C \exp [-c(\log n)^{\phi L}].
\end{align} 
For the third term, suppose we choose $M \geq 16n$.
Then, for all $t$ in the integral we have $t \geq 16n$, and so \begin{align}
    \{ |\lambda_\alpha - \sqrt{n} \gamma_\alpha| \geq \sqrt{t} \} &\subseteq \{ \| X\| \geq \sqrt{t}-2\sqrt{n}  \} \\
    &\subseteq \left\{ \| X\| \geq \frac{1}{2}\sqrt{t} \right\} \intertext{We apply the concentration inequality given in Lemma~\ref{lem: norm_inequality}, by choosing $M = \max\{  16n,  4C^2n\}$.
    Set $s = \frac{1}{2C} \sqrt{\frac{t}{n}}$, then for all $t \geq M$, we have $s \geq 1$ and}
    &= \{ \| X\| \geq   C s \sqrt{n} \}.
\end{align}
Thus, we can bound the probability as
\begin{align}
    \int_{M}^\infty  \PP\left( |{\lambda}_\alpha -\sqrt{n}\gamma_\alpha|^2 \geq t  \right)  dt  
    &\leq \int_{M}^\infty \PP \left( \|X\| > C   s \sqrt{n} \right)dt \\
    &= \int_{M}^\infty 2\exp\left(- \frac{2}{4C^2}  \frac{t}{n}  \cdot n\right)   dt = 4C^2 \exp\left( - \frac{M}{2C^2} \right).
\end{align} 
Together, since the exponential term decays faster than any polynomial, we obtain \begin{align}
    \Var({\lambda}_\alpha) &\leq (\log n)^L\hat{\alpha}^{-2/3}n^{-1/3} + M \cdot  C\exp  [-c(\log n)^{\phi L}] + 2C^2 \exp\left( - \frac{M}{2C^2} \right)  \\
    &\leq \tilde{C}  (\log n)^L\hat{\alpha}^{-2/3}n^{-1/3},
\end{align}
for some constant $\tilde{C}  > 0 $, giving the result.
\end{proof}

\end{document}